\documentclass[preprint,1p]{elsarticle}
\pdfoutput=1
\usepackage{graphicx}
\usepackage[utf8]{inputenc}
\usepackage[english]{babel}

\usepackage{amsmath,amssymb,amsthm,dsfont,enumerate,url,appendix,tabularx}
\usepackage{todonotes,comment,afterpage}

\usepackage{graphicx}
\usepackage{epstopdf}

\usepackage{booktabs}

\usepackage{siunitx}

\theoremstyle{plain}
\newtheorem{theorem}{Theorem}[section]
\newtheorem{lemma}[theorem]{Lemma}
\newtheorem{corollary}[theorem]{Corollary}
\newtheorem{proposition}[theorem]{Proposition}

\newtheorem{definition}[theorem]{Definition}
\newtheorem{example}[theorem]{Example}

\newtheorem{remark}[theorem]{Remark}

\numberwithin{equation}{section}
\numberwithin{figure}{section}
\numberwithin{table}{section}

\title{Optimal additive Schwarz methods for the $hp$-BEM: the hypersingular integral operator in 3D on locally refined meshes}
\author[puc]{T. F\"uhrer}
\ead{tofuhrer@mat.puc.cl}
\author[asc]{J. M. Melenk\corref{cor1}}
\ead{melenk@tuwien.ac.at}
\author[asc]{D. Praetorius}
\ead{dirk.praetorius@tuwien.ac.at}
\author[asc]{A. Rieder}
\ead{alexander.rieder@tuwien.ac.at}

\address[puc]{Pontificia Universidad Cat\'olica de Chile,
  Facultad de Matem\'aticas, Vicu\~na Mackenna 4860, Santiago, Chile.}
\address[asc]{Technische Universit\"at Wien, Institut f\"ur Analysis und Scientific Computing, Wiedner Hauptstra{\ss}e 8-10, A-1040 Vienna}
\cortext[cor1]{Corresponding author}

\date{\today}

\begin{document}
\begin{abstract}
  We propose and analyze an overlapping Schwarz preconditioner for the $p$ and $hp$ boundary element method 
  for the hypersingular integral equation in 3D.
  We consider surface triangulations consisting of triangles.
  The condition number is bounded uniformly in the mesh size $h$ and the polynomial order $p$.
  The preconditioner handles adaptively refined meshes and is based on a local multilevel preconditioner for the
  lowest order space. Numerical experiments on different geometries illustrate its robustness. 
\end{abstract}

\maketitle

\section{Introduction}
Many elliptic boundary value problems that are solved in practice are linear 
and have constant (or at least piecewise constant) coefficients. In this setting, 
the boundary element method (BEM, \cite{book_hsiao_wendland,book_sauter_schwab,book_steinbach,book_mclean})
 has established itself as an effective alternative to the finite element method (FEM). Just as in the FEM applied 
to this particular problem class, high order methods are very attractive since they
can produce rapidly convergent schemes on suitably chosen adaptive meshes.  
The discretization leads to large systems of equations, and a use of 
iterative solvers brings the question of preconditioning to the fore.  

In the present work, we study high order Galerkin discretizations of the 
hypersingular operator. This is an operator of order $1$, and we therefore have to 
expect the condition number of the system matrix to increase as the mesh size $h$ 
decreases and the approximation order $p$ increases. We present an additive overlapping Schwarz
preconditioner that offsets this degradation and results in condition numbers that 
are bounded independently of the mesh size and the approximation order. This is achieved 
by combining the recent $H^{1/2}$-stable decomposition of spaces of piecewise polynomials of degree $p$
of \cite{melenk_appendix} and the multilevel diagonal scaling preconditioner of \cite{ffps,dissTF}
for the hypersingular operator discretized by piecewise linears.

Our additive Schwarz preconditioner is based on stably decomposing the approximation space of piecewise polynomials 
into the lowest order space (i.e., piecewise linears) and spaces of higher order polynomials supported
by the vertex patches. Such stable localization procedures were first developed for the $hp$-FEM 
in \cite{pavarino_94} for meshes consisting of quadrilaterals (or, more generally, tensor product elements). 
The restriction to tensor product elements stems from the fact that the localization is achieved by exploiting
stability properties of the 1D-Gau{\ss}-Lobatto interpolation operator, which, when applied to polynomials, 
is simultaneously stable in $L^2$ and $H^1$ (see, e.g., \cite[eqns. (13.27), (13.28)]{bernardi-maday97}). 
This simultaneous stability raises the hope for $H^{1/2}$-stable localizations and was 
pioneered in \cite{heuer_asm_indef_hyp} for the $hp$-BEM for the hypersingular operator on meshes consisting
of quadrilaterals. Returning to the $hp$-FEM, $H^1$-stable localizations on triangular/tetrahedral meshes
were not developed until \cite{schoeberl_asm_fem}. The techniques developed there were subsequently used in
\cite{melenk_appendix} to design $H^{1/2}$-stable decompositions on triangular meshes and thus paved the way 
for condition number estimates that are uniform in the approximation $p$ for overlapping Schwarz methods for 
the $hp$-version BEM applied to the hypersingular operator. Non-overlapping additive Schwarz preconditioners
for high order discretizations of the hypersingular operator are also available in the literature, 
\cite{ainsworth-guo00}; as it is typical of this class of preconditioners, the condition number 
still grows polylogarithmically in $p$.

Our preconditioner is based on decomposing the approximation space into the space of piecewise linears
and spaces associated with the vertex patches. It is highly desirable to decompose the space of piecewise linears
further in a multilevel fashion.  For sequences of uniformly refined meshes, the first such 
multilevel space decomposition appears to be \cite{tran_stephan_asm_h_96} (see also \cite{oswald_99}). 
For adaptive meshes, local multilevel diagonal scaling was first analyzed in \cite{amcl03}, where 
for a sequence $\mathcal{T}_{\ell}$ of successively refined adaptive meshes a uniformly bounded condition number
for the preconditioned system is established. Formally, however, \cite{amcl03} requires 
that $\mathcal{T}_{\ell} \cap \mathcal{T}_{\ell+1} \subset \mathcal{T}_{\ell+k}$ for all $\ell,k \in \mathbb{N}_0$, i.e., as soon as an element $K \in \mathcal{T}_{\ell}$ is not refined, it remains non-refined in all 
succeeding triangulations. While this can be achieved implementationally, the recent works \cite{ffps,dissTF} 
avoid such a restriction by considering sequences of meshes that are obtained 
in typical $h$-adaptive environments with the aid of {\em newest vertex bisection} (NVB). 
We finally note that the additive Schwarz decomposition on adaptively refined meshes is a subtle issue. 
Hierarchical basis preconditioners (which are based on the new nodes only) lead to a growth of the 
condition number with $\mathcal{O}(\left|\log{h_{min}}\right|^2)$; see \cite{tran_stephan_mund_hierarchical_prec}.
Global multilevel diagonal preconditioning (which is based on all nodes) leads to a growth $\mathcal{O}(\left|\log{h_{min}}\right|)$; 
see \cite{maischak_multilevel_asm,ffps}. 

The paper is organized as follows: In Section~\ref{sec:model_problem} we introduce the hypersingular equation 
and the discretization by high order piecewise polynomial spaces. 
Section~\ref{sec:properties-of-honehalftilde} collects properties of the 
fractional Sobolev spaces including the scaling properties. 
Section~\ref{sec:p_condition} studies
in detail the $p$-dependence of the condition number of the unpreconditioned system. The polynomial
basis on the reference triangle chosen by us is a hierarchical basis of the form first proposed by 
Karniadakis \& Sherwin, \cite[Appendix~{D.1.1.2}]{karniadakis_sherwin}; the precise form is the one from 
\cite[Section~{5.2.3}]{zaglmayr_diss}.
We prove bounds for the condition number of the stiffness matrix not only in the $H^{1/2}$-norm but also 
in the norms of $L^2$ and $H^1$. This is also of interest for $hp$-FEM and could not be found in the literature. 
Section~\ref{sec:main_results} develops several preconditioners. The first one (Theorem~{~\ref{thm:ppreconditioner}}) is based on 
decomposing the high order approximation space into the global space of piecewise linears and local 
high order spaces of functions associated with the vertex patches. The second one (Theorem~{~\ref{thm:p_precond_with_multilevel}}) is based on a further 
multilevel decomposition of the global space of piecewise linears. The third one (Theorem~{~\ref{thm:hp_reference_solver_preconditioner}}) 
exploits the observation that topologically, only a finite number of vertex patches can occur. Hence, significant
memory savings for the preconditioner are possible if the exact bilinear forms for the vertex patches are replaced 
with scaled versions of simpler ones defined on a finite number of reference configurations. 
Numerical experiments in Section~\ref{sec:numerics} illustrate that the proposed preconditioners are indeed robust
with respect to both $h$ and $p$.  

We close with a remark on notation: 
The expression $a \lesssim b$ signifies the existence of a constant $C>0$ such that $a \leq C \, b$. The constant $C$ 
does not depend on the mesh size $h$ and the approximation
order $p$, but may depend on the geometry and the shape regularity of the triangulation. We also write $a \sim b$ to
abbreviate $a \lesssim b \lesssim a$.

\section{$hp$-discretization of the hypersingular integral equation}
\label{sec:model_problem}
\subsection{Hypersingular integral equation}
Let $\Omega\subset\mathbb{R}^3$ be a bounded Lipschitz polyhedron with a connected boundary $\partial \Omega$, 
and let $\Gamma \subseteq \partial \Omega$ be
an open, connected subset of $\partial \Omega$. If $\Gamma \ne \partial\Omega$, we assume it to be 
a Lipschitz hypograph, \cite{book_mclean}; the key property needed is that 
$\Gamma$ is such that the ellipticity condition (\ref{eq:ellipticity}) holds. 
Furthermore, we will use affine, shape regular
triangulations of $\Gamma$, which further imposes conditions on $\Gamma$. 
In this work, we are concerned with preconditioning high order discretizations of the hypersingular integral operator, which is given by
\begin{align}
\label{hypsing_operator_definition}
\left(D u \right)(x)&:=  - \partial_{n_x }^{int} \int_{\Gamma}{ \partial_{n_y}^{int} G(x,y)u(y) \; ds_y} \quad \text{for } x \in \Gamma, 
\end{align}
where $G(x,y):= \frac{1}{4\pi} \frac{1}{\left|x-y\right|}$ is the fundamental solution of the 3D-Laplacian and $\partial_{n_y}^{int}$ denotes the
(interior) normal derivative with respect to $y \in \Gamma$.

We will need some results from the theory of Sobolev and interpolation spaces, see \cite[Appendix B]{book_mclean}.
For an open subset $\omega \subset \partial \Omega$, let $L^2(\omega)$ and $H^1(\omega)$ denote the usual 
Sobolev spaces. The space $\widetilde H^1(\omega)$ consists of those 
functions whose zero extension to $\partial\Omega$ is in $H^1(\partial\Omega)$.  (In particular, for 
$\omega = \partial\Omega$, $H^1(\partial\Omega) = \widetilde H^1(\partial\Omega)$.) 
When the surface measure of the set $\partial\Omega\setminus \omega$ is positive, we use the equivalent norm
$\left\|u\right\|_{\widetilde{H}^1(\omega)}^2:=\left\|\nabla_{\Gamma} u\right\|^2_{L^2(\omega)}$.

We will define fractional Sobolev norms by interpolation. The following Proposition~\ref{interpolation_theorem}
collects key properties of interpolation spaces that we will need; we refer to 
\cite{tartar07,triebel95} for a comprehensive treatment.  
For two Banach spaces $\left(X_0, \left\|\cdot\right\|_0\right)$ and $\left(X_1,\left\|\cdot\right\|_{1}\right)$, 
with continuous inclusion $X_1 \subseteq X_0$ and a parameter $s \in (0,1)$ the interpolation norm is defined as
\begin{align*}
  \left\|u\right\|^2_{[X_0,X_1]_s} &:= \int_{t=0}^\infty t^{-2s} \left( \inf_{v \in X_1} \|u - v\|_{0} + t \|v\|_1\right)^2 \frac{dt}{t}.
\end{align*}
The interpolation space is given by $\left[X_0,X_1\right]_{s}:=\left\{ u \in X_0 : \left\|u\right\|_{[X_0,X_1]_{s}} < \infty \right\}$.

An important result, which we use in this paper, is the following interpolation theorem:
\begin{proposition}
\label{interpolation_theorem}
  Let $X_i$, $Y_i$, $i\in \{0,1\}$, be two pairs of Banach spaces with continuous inclusions $X_1 \subseteq X_0$ and $Y_1 \subseteq Y_0$. Let $s \in (0,1)$. 
\begin{enumerate}[(i)]
\item
\label{item:interpolation_theorem:i}
  If a linear operator $T$ is bounded as an operator $X_0 \to Y_0$ and $X_1 \to Y_1$, then it is also bounded
  as an operator $[X_0,X_1]_{s} \to [Y_0,Y_1]_{s}$ with
  \begin{align*}
    \left\|T\right\|_{[X_0,X_1]_{s} \to [Y_0,Y_1]_{s}} \leq \left\|T\right\|_{X_0 \to Y_0}^{1-s} \left\|T\right\|_{X_1 \to Y_1}^s.  
  \end{align*}
\item
\label{item:interpolation_theorem:ii}
  There exists a constant $C > 0$ such that for all $x \in X_1$: $\displaystyle \left\|x\right\|_{[X_0,X_1]_{s} } \leq C \left\|x\right\|_{X_0}^{1-s} \left\|x\right\|_{X_1}^s.$ 
\end{enumerate}
\end{proposition}

We define the fractional Sobolev spaces by interpolation. For $s \in (0,1)$, we set:
\begin{align*}
  H^{s}(\omega)&:=\left[ L^2(\omega), H^1(\omega) \right]_s, \quad
  \widetilde{H}^{s}(\omega):=\left[ L^2(\omega), \widetilde{H}^1(\omega) \right]_s.
\end{align*}
Here, we will only consider the case $s=1/2$.
We define $H^{-1/2}(\Gamma)$ as the dual space of $\widetilde{H}^{1/2}(\Gamma)$, where
duality is understood with respect to the (continuously) extended $L^2(\Gamma)$-scalar product and denoted by  $\left<\cdot,\cdot \right>_{\Gamma}$.
An equivalent norm on $H^{1/2}(\Gamma)$ is given by ${\left\|u\right\|_{H^{1/2}(\Gamma)}^2 \sim \left\|u\right\|_{L^2(\Gamma)}^2 + \left|u\right|_{H^{1/2}(\Gamma)}^2}$,
where  $\left|\cdot\right|_{H^{1/2}(\Gamma)}$ is given by the Sobolev-Slobodeckij seminorm (see \cite{book_sauter_schwab} for the exact definition).

We now state some important properties of the hypersingular operator $D$ from \eqref{hypsing_operator_definition},  see, e.g., \cite{book_sauter_schwab,book_mclean,book_hsiao_wendland,book_steinbach}.
First, the operator $D: \widetilde{H}^{1/2}(\Gamma) \to H^{-1/2}(\Gamma)$ is a bounded linear operator.

For open surfaces $\Gamma \subsetneqq \partial \Omega$ the operator is elliptic
\begin{align}
\label{eq:ellipticity}
  \left<D u, u \right>_{\Gamma} &\geq c_{ell} \left\|u\right\|_{\widetilde{H}^{1/2}(\Gamma)}^2 \quad \quad \forall u \in \widetilde{H}^{1/2}(\Gamma),
\end{align}
with some constant $c_{ell} > 0$ that only depends on $\Gamma$.
In the case of a closed surface, i.e. $\Gamma=\partial \Omega$ we note that $\widetilde{H}^{1/2}(\Gamma)=H^{1/2}(\Gamma)$ and the operator $D$ is still semi-elliptic, i.e.
\begin{align*}
  \left<D u,u \right>_{\Gamma} &\geq  c_{ell} \left|u\right|_{H^{1/2}(\Gamma)}^2 \quad \quad \forall u \in \widetilde{H}^{1/2}(\Gamma).
\end{align*}
Moreover, the kernel of $D$ then consists of the constant functions only: $\operatorname{ker}(D)~=~\operatorname{span}(1)$.

To get unique solvability and strong ellipticity for the case of a closed surface, it is customary to introduce 
a stabilized operator $\widetilde{D}$ given by the bilinear form 
\begin{align}
  \label{eq:def_stabilized_op}
  \left<\widetilde{D} u ,v \right>_{\Gamma}:= \left<Du,v \right>_{\Gamma} + \alpha^2 \left<u,1 \right>_{\Gamma} \left<v,1 \right>_{\Gamma},  \quad \alpha > 0.
\end{align}
In order to avoid having to distinguish the two cases $\Gamma=\partial \Omega$ and $\Gamma \subsetneqq \partial \Omega$, we 
will only work with the stabilized form on $\widetilde{H}^{1/2}(\Gamma)$ and just set $\alpha = 0$ in the
case of $\Gamma \subsetneqq \partial \Omega$.
The basic integral equation involving the hypersingular operator $D$ then reads:
For given $g \in H^{-1/2}(\Gamma)$, find $u \in \widetilde{H}^{1/2}(\Gamma)$ such that
\begin{align}
\label{eq:weakform}
\left<\widetilde{D} u ,v \right>_{\Gamma} &= \left<g,v \right>_{\Gamma} \quad \forall v \in \widetilde{H}^{1/2}(\Gamma).
\end{align}

We note that in the case of the closed surface $\Gamma=\partial \Omega$, the solution of the stabilized system above is equivalent
to the solution of $\left<Du,v \right>_{\Gamma}=\left<g,v \right>_{\Gamma}$ under the side constraint $\left<u,1 \right>_{\Gamma}=\frac{\left<g,1 \right>_{\Gamma}}{\alpha^2 \, \left|\Gamma\right|}$.
Moreover, it is well known that $\left<\widetilde{D} \cdot ,\cdot \right>_{\Gamma}$ is symmetric, elliptic and induces an equivalent norm on $\widetilde{H}^{1/2}(\Gamma)$, i.e.,
\begin{align*}
  \left<\widetilde{D} u ,u \right>_{\Gamma} \sim \left\|u\right\|_{\widetilde{H}^{1/2}(\Gamma)}^2 \quad \quad \forall u \in \widetilde{H}^{1/2}(\Gamma).
\end{align*}

\subsection{Discretization}
\label{sect:discretization}
Let $\mathcal{T} = \{K_1,\dots,K_N\}$ denote a regular (in the sense of Ciarlet) triangulation of the two-dimensional
manifold $\Gamma\subseteq\partial \Omega$ into compact, non-degenerate planar surface triangles.
We say that a triangulation is $\gamma$-shape regular, if there exists a constant $\gamma> 0$ such that
\begin{align}
\label{eq:shape-regularity}
  \max_{K\in\mathcal{T}} \frac{\mathrm{diam}(K)^2}{|K|} \leq \gamma.
\end{align}

Let $\widehat{K}:=\operatorname{conv}\{(0,0),(1,0),(0,1)\}$ be the reference triangle.
With each element $K$ we associate an affine, bijective element map $F_K: \widehat{K} \to K$. 
We will write  $P^{p}(\widehat{K})$ for the space of polynomials of degree $p$ on $\widehat{K}$.
The space of piecewise polynomials on $\mathcal{T}$ is given by
\begin{align}
  P^p(\mathcal{T}) := \left\{ u\in L^2(\Gamma) \,:\, u \circ F_K  \in P^p(\widehat{K}) \text{ for all } K\in\mathcal{T} \right\}.
\end{align}
The elementwise constant mesh width
function $h:=h_\mathcal{T} \in P^0(\mathcal{T})$ is defined by $(h_\mathcal{T})|_K := \mathrm{diam}(K)$ for all $K\in\mathcal{T}$.

Let ${\mathcal{V}} = \{{\boldsymbol{z}}_1,\dots,{\boldsymbol{z}}_M\}$ denote the set of all vertices of the triangulation $\mathcal{T}$ that are not on the boundary of $\Gamma$.
We define the (vertex) patch $\omega_{\boldsymbol{z}}$ for a vertex ${\boldsymbol{z}} \in{\mathcal{V}}$ by
\begin{align}\label{def:patch}
  \omega_{\boldsymbol{z}} := \operatorname*{interior} \left(\bigcup_{\left\{K\in\mathcal{T}: \;{\boldsymbol{z}} \in K \right\}} K \right), 
\end{align}
where the interior is understood with respect to the topology of $\Gamma$. 
For $p \geq 1$, define
\begin{align}
  \widetilde{S}^p(\mathcal{T}) := P^p(\mathcal{T}) \cap \widetilde{H}^{1/2}(\Gamma).
\end{align}
Then, the Galerkin discretization of~\eqref{eq:weakform} consists in replacing $\widetilde{H}^{1/2}(\Gamma)$ with 
the discrete subspace $\widetilde{S}^p(\mathcal{T})$, i.e.: 
Find $u_h \in \widetilde{S}^p(\mathcal{T})$ such that
\begin{align}\label{eq:weakform:discrete}
  \left<\widetilde{D} u_h ,v_h \right>_{\Gamma} = \left<g,v_h \right>_{\Gamma} \quad\text{for all } v_h \in \widetilde{S}^p(\mathcal{T}).
\end{align}

\begin{remark}
  We employ the same polynomial degree for all elements. This is not essential and done for simplicity of 
  presentation. For details on the more general case, see \cite{melenk_appendix}.
\hbox{}\hfill\rule{0.8ex}{0.8ex}
\end{remark}
After choosing a basis of $\widetilde{S}^p(\mathcal{T})$,
the problem (\ref{eq:weakform:discrete}) can be written as a linear system of equations, and we write 
$\widetilde{D}^p_h$ for the resulting system matrix. Our goal is to construct a preconditioner for $\widetilde{D}^p_h$.
It is well-known that the condition number of $\widetilde{D}_h^p$ depends on the choice of the basis of 
$\widetilde{S}^p(\mathcal{T})$, which we fix in Definition~\ref{definition_basis} below. 
We remark in passing that the  preconditioned system of 
Section~\ref{sec:hpprecond} will no longer depend on the basis. 

\subsection{Polynomial basis on the reference element}
\label{sec:basis_reference_element}
For the matrix representation of the Galerkin formulation (\ref{eq:weakform:discrete}) we have to 
specify a polynomial basis on the reference triangle $\widehat{K}$. We use a basis that relies 
on a collapsed tensor product representation of the triangle and is given in \cite[Section 5.2.3]{zaglmayr_diss}. 
This kind of basis was first proposed for the $hp$-FEM 
by Karniadakis \& Sherwin, \cite[Appendix D.1.1.2]{karniadakis_sherwin}; closely related earlier works on polynomial
bases that rely on a collapsed tensor product representation of the triangle are \cite{koornwinder75,dubiner91}. 

\begin{definition}[Jacobi polynomials]
  \label{def:ortho_polys}
  For coefficients $\alpha$, $\beta > -1$ the family of {\em Jacobi polynomials} on the interval $\left(-1,1\right)$ is denoted by $P_n^{(\alpha,\beta)}$, $n \in \mathbb{N}_0$.
  They are orthogonal with respect to the $L^2(-1,1)$ inner product with weight $(1-x)^\alpha (1+x)^\beta$.
  (See for example \cite[Appendix A]{karniadakis_sherwin} or \cite[Appendix A.3]{zaglmayr_diss} for the exact definitions and
  a list of important properties).
  The \emph{Legendre polynomials} are a special case of the Jacobi polynomials for $\alpha=\beta=0$ and denoted by $\ell_n(s):=P^{(0,0)}_{n}(s)$.
  The \emph{integrated Legendre polynomials} $L_n$ and the \emph{scaled polynomials} are defined by 
  \begin{align}
    \label{eq:def_legendre_poly}
    L_n(s)&:=\int_{-1}^{s}{\ell_{n-1}(t) dt} \quad \text{for } n \in \mathbb{N},
    &P_{n}^{\mathcal{S},(\alpha,\beta)}(s,t)&:=t^{n}P^{(\alpha,\beta)}_n(s/t), 
    &L_n^{\mathcal{S}}(s,t)&:=t^n L_n(s/t).
  \end{align}
\end{definition}

On the reference triangle, our basis reads as follows:
\begin{definition}[polynomial basis on the reference triangle]
\label{definition_basis}
  Let $p \in {\mathbb N}$ and let $\lambda_1,\lambda_2,\lambda_3$ be the barycentric coordinates on the reference triangle $\widehat{K}$. 
  Then the basis functions for the reference triangle consist of three vertex functions, 
$p-1$ edge functions per edge, and $(p-1)(p-2)/2$ cell-based functions: 
    \begin{enumerate}[(a)]
    \item for $i=1$, $2$, $3$ the vertex functions are: 
      \begin{align*}    
      \varphi^{\mathcal{V}}_i&:=\lambda_i;
    \intertext{\item  for $m=1$, $2$, $3$ and an edge $\mathcal{E}_m$ with edge vertices $e_1,e_2$,
        the edge functions are given by:}
      \varphi_{i}^{\mathcal{E}_m}&:=\sqrt{\frac{2i+3}{2}}\,L_{i+2}^{\mathcal{S}}(\lambda_{e_2}-\lambda_{e_1},\lambda_{e_1}+\lambda_{e_2}), \quad \quad 0\leq i\leq p-2;
    \intertext{\item for $0 \leq i+j\leq p-3$ the cell based functions are:}
      \varphi_{(i,j)}^{\mathcal{I}}&:=c_{ij} \lambda_1 \lambda_2 \lambda_3 P_{i}^{\mathcal{S},(2,2)}(\lambda_1-\lambda_2,\lambda_1+\lambda_2) \, P_j^{(2i+5,2)}(2\lambda_3-1).
    \end{align*}
    with $c_{ij}$ such that $\left\|\varphi_{(i.j)}^{\mathcal{I}}\right\|_{L^2(\widehat{K})}=1$.
  \end{enumerate}
\end{definition}

\begin{remark}
  In order to get a basis of $\widetilde{S}^p(\mathcal{T})$ we take the composition with the element mappings $\varphi \circ F_K$.
  To ensure continuity along edges we take an arbitrary orientation of the edges and observe that
  the edge basis functions $\varphi_i^\mathcal{E}$ are symmetric under permutation of $\lambda_{e_1}$ and $\lambda_{e_2}$ up to a sign change $(-1)^{i}$.
\hbox{}\hfill\rule{0.8ex}{0.8ex}
\end{remark}

 \section{Properties of $\widetilde{H}^{1/2}(\Gamma)$} 
\label{sec:properties-of-honehalftilde}
\subsection{Quasi-interpolation in $\widetilde{H}^{1/2}(\Gamma)$}
Several results of the present paper depend on results in \cite{melenk_appendix}. 
Therefore we present a short summary of the main results of that paper in this section.
In \cite{melenk_appendix} the authors propose an $H^{1/2}$-stable space decomposition on meshes consisting of triangles.
It is based on quasi-interpolation operators constructed by local averaging on elements. 

We introduce the following product spaces:
\begin{align*}
  X_0&:=\prod_{\boldsymbol{z} \in {\mathcal{V}}} { L^2(\omega_{\boldsymbol{z}}) }, \quad \quad  X_1:=\prod_{\boldsymbol{z}\in {\mathcal{V}}}{ \widetilde{H}^1(\omega_{\boldsymbol{z}})}.
\end{align*}
The spaces $L^2(\omega_{\boldsymbol{z}})$ and $\widetilde{H}^1(\omega_{\boldsymbol{z}})$ are endowed with the 
$L^2$- and $\widetilde{H}^1$-norm, respectively. 

\begin{proposition}[localization, \protect{\cite{melenk_appendix}}]
  \label{prop:stable_space_decomposition}
  There exists an operator $J: L^2(\Gamma) \to \left(\widetilde{\mathcal{S}}^1(\mathcal{T}),\left\|\cdot\right\|_{L^2(\Gamma)}\right) \times X_0$ with the following properties:
  \begin{enumerate}[(i)]
    \item 
  \label{item:prop:stable_space_decomposition-i}
  $J$ is linear and bounded.
    \item 
  \label{item:prop:stable_space_decomposition-ii}
  $J|_{\widetilde{H}^1(\Gamma)}$ is also bounded as an operator
      $\widetilde{H}^1(\Gamma) \to \left(\widetilde{\mathcal{S}}^1(\mathcal{T}),\left\|\cdot\right\|_{\widetilde{H}^1(\Gamma)}\right) \times X_1$.    
    \item 
  \label{item:prop:stable_space_decomposition-iii}
  If $u \in \widetilde{S}^p(\mathcal{T})$ then each component of $J u$ is in $\widetilde{S}^p(\mathcal{T})$.
    \item 
  \label{item:prop:stable_space_decomposition-iv}
If we write $J u=:(u_1,U)$, and furthermore $U_{\boldsymbol{z}}$ for the component of $U$ in $X_0$ corresponding
      to the space $L^2(\omega_{\boldsymbol{z}})$, then $Ju$ represents an $\widetilde{H}^{1/2}(\Gamma)-stable$ decomposition of $u$, i.e.,      
      \begin{align}
        \label{eq:stable_space_decomposition_sum}
          u&=u_1 + \sum_{\boldsymbol{z} \in {\mathcal{V}}}{U_{\boldsymbol{z}}} \quad \text{and} \quad 
          \left\|u_1\right\|_{\widetilde{H}^{1/2}(\Gamma)}^2 + \sum_{\boldsymbol{z} \in {\mathcal{V}}}{\left\|U_{\boldsymbol{z}}\right\|_{\widetilde{H}^{1/2}(\omega_{\boldsymbol{z}})}^2}\leq
          C \left\|u\right\|_{\widetilde{H}^{1/2}(\Gamma)}^2.
      \end{align}   
  \end{enumerate}
\vspace*{-\baselineskip}
  The norms of $J$ in 
(\ref{item:prop:stable_space_decomposition-i})---(\ref{item:prop:stable_space_decomposition-ii}) and the 
constant $C > 0$ in (\ref{item:prop:stable_space_decomposition-iv})
depend only on $\Gamma$ and the shape regularity constant $\gamma$.
  \begin{proof}[Sketch of proof: \nopunct]
    The first component of $J$ (i.e., the mapping $u \mapsto u_1$) consists of the Scott-Zhang 
projection operator, as modified in \cite[Section 3.2]{aff_hypsing}.
    The local components (i.e., the functions $U_{\boldsymbol{z}}$, $z \in {\mathcal{V}}$) then are based 
    on a successive decomposition into vertex, edge and interior parts, similar to what is done in 
    \cite{schoeberl_asm_fem}.  We give a flavor of the procedure. Set $u_2:= u - u_1$. 
    In order to define the vertex parts for a vertex $\boldsymbol{z}$, 
    we select an element $K \subset \omega_{\boldsymbol{z}}$ of the patch $\omega_{\boldsymbol{z}}$ and perform a 
    suitable local averaging of $u_2$ on that element; this averaged 
    function  $u_{loc,K}$ is defined on $K$ in terms of $u_2|_K$ and vanishes on the edge opposite $\boldsymbol{z}$.  
    In order to extend $u_{loc,K}$ 
    to the patch $\omega_{\boldsymbol{z}}$ and thus obtain the function $u_{\boldsymbol{z}}$, we define $u_{\boldsymbol{z}}$ 
    by ``rotating'' $u_{loc,K}$ around the vertex $\boldsymbol{z}$.  The averaging process can be done 
    in such a way that for continuous functions $u_2$ one has 
    $u_2(\boldsymbol{z}) = u_{\boldsymbol{z}}(\boldsymbol{z})$ and that one has appropriate 
    stability properties in $L^2$ and $H^1$. The edge contributions are constructed from the function 
    $u_3:= u_2 - \sum_{\boldsymbol{z} \in {\mathcal{V}}} u_{\boldsymbol{z}}$. 
    Let ${\mathcal E}(\mathcal{T})$ denote the set of interior edges of $\mathcal{T}$. 
    For an edge ${\mathcal E} \in {\mathcal E}(\mathcal{T})$ one selects 
    an element (of which ${\mathcal E}$ is an edge), averages there, and extends the obtained averaged function to 
    the edge patch by symmetry across the edge ${\mathcal E}$. In this way, the function 
    $u_{{\mathcal E}}$ is constructed for each edge ${\mathcal E}$.
    It again holds for sufficiently smooth $u_3$ that 
    $u_3(x)=u_{{\mathcal E}}(x) \; \forall x \in {\mathcal E}$.
        For $u \in \widetilde{H}^1(\Gamma)$ and in turn $u_2 \in \widetilde{H}^1(\Gamma)$, we have that 
    $u_4:= u- \sum_{ {\mathcal E} \in \mathcal{E}(\mathcal{T})}{u_{\mathcal E}} -  \sum_{\boldsymbol{z} \in {\mathcal{V}}}{u_{\boldsymbol{z}}}$  
    vanishes on all edges; hence $u_4|_K \in \widetilde{H}^1(K)$ for all $K \in \mathcal{T}$. 
    The terms $u_{\boldsymbol{z}}$, $u_{\mathcal E}$, $u_4|_K$ can be rearranged to take the form of patch contributions 
    $U_{\boldsymbol{z}}$ as given in the statement of the proposition. (The decomposition is not unique.) 
    The $\widetilde{H}^{1/2}$ stability
    is a direct consequence of the $L^2$ and $H^1$ stability and interpolation properties given in 
Proposition~\ref{interpolation_theorem}, (\ref{item:interpolation_theorem:i}).
  We finally mention that assertion (\ref{item:prop:stable_space_decomposition-iii}) follows from the fact 
  that the averaging operators employed at the various stages of the decomposition are polynomial preserving. 
  \end{proof}
\end{proposition}
\begin{remark}
Independently, a decomposition similar to Proposition~\ref{prop:stable_space_decomposition} 
was presented in \cite{falk-winther13}.
\hbox{}\hfill\rule{0.8ex}{0.8ex}
\end{remark}

The construction in Proposition~\ref{prop:stable_space_decomposition} can be modified and used 
to relate the $\widetilde{H}^{1/2}(\omega_{\boldsymbol{z}})$-norm to the 
$\widetilde{H}^{1/2}({\mathcal O})$-norm if ${\mathcal O} \supset \omega_{\boldsymbol{z}}$: 
\begin{corollary}
\label{cor:restriction_extension}
  Let $\boldsymbol{z} \in {\mathcal{V}}$ and $\mathcal{O}$ be  the union of some triangles of $\mathcal{T}$ with $\omega_{\boldsymbol{z}} \subseteq \mathcal{O} \subseteq \Gamma$.
  Then there exist constants $c_1,c_2$ that depend only on $\mathcal{O}$, $\Gamma$, 
and the $\gamma$-shape regularity of $\mathcal{T}$ such that
  for all $u \in \widetilde{H}^{1/2}(\mathcal{O})$ with $\operatorname{supp}(u)\subseteq \overline{\omega_{\boldsymbol{z}}}$, 
we can estimate:
  \begin{align*}
    c_1\;\left\|u\right\|_{\widetilde{H}^{1/2}(\omega_{\boldsymbol{z}})}&\leq \left\|u\right\|_{\widetilde{H}^{1/2}(\mathcal{O})}\leq c_2 \left\|u\right\|_{\widetilde{H}^{1/2}(\omega_{\boldsymbol{z}})}.
  \end{align*}
  \begin{proof}
    To see the second inequality, consider the extension operator $E$ that extends the function $u$ 
    by $0$ outside of $\omega_{\boldsymbol{z}}$.
    This operator is continuous $L^2(\omega_{\boldsymbol{z}}) \rightarrow L^2(\mathcal{O})$ and 
    $\widetilde{H}^1(\omega_{\boldsymbol{z}}) \rightarrow \widetilde{H}^1(\mathcal{O})$, both with constant $1$. Applying 
    Proposition~\ref{interpolation_theorem}, (\ref{item:interpolation_theorem:i}) to this extension operator $E$ 
    gives the second inequality with $c_2=1$.
    The first inequality is more involved. We start by noting that the stability assertion \eqref{eq:stable_space_decomposition_sum} of 
    Proposition~\ref{prop:stable_space_decomposition} gives 
    $u = u_1 + \sum_{\boldsymbol{z^\prime} \in {\mathcal{V}}} U_{\boldsymbol{z^\prime}}$ and 
    \begin{align}
\label{eq:foo-1001}
      \left\|u_1\right\|^2_{\widetilde{H}^{1/2}(\mathcal{O})}
      + \sum_{\boldsymbol{z^\prime} \in {\mathcal{V}}}{\left\|U_{\boldsymbol{z^\prime}}\right\|^2_{\widetilde{H}^{1/2}(\omega_{\boldsymbol{z^\prime}})}}&\leq C \left\|u\right\|^2_{\widetilde{H}^{1/2}(\mathcal{O})}.
    \end{align}
   The constant $C$ depends only on the set ${\mathcal O}$ and the shape regularity of the triangulation,
   when Proposition \ref{prop:stable_space_decomposition} is applied with $\Gamma$ replaced by $\mathcal{O}$.
   The decomposition in Proposition~\ref{prop:stable_space_decomposition} is not unique, and we will now exploit 
   this by requiring more. Specifically, we assert that the operator $J$, which effects the decomposition, can be 
   chosen such that, for given $\omega_{\boldsymbol{z}}$, we have 
   $\operatorname*{supp} u_1 \subset \overline{\omega_{\boldsymbol{z}}}$ and $U_{\boldsymbol{z}^\prime} = 0$ for 
   $\boldsymbol{z}^\prime \ne \boldsymbol{z}$.   
     If this can be achieved, we get $u = u_1 + U_{\boldsymbol{z}}$.
     Since \eqref{eq:foo-1001} contains a term $\left\|u_1\right\|_{\widetilde{H}^{1/2}(\mathcal{O})}$,
     we also need to reinvestigate the stability proof. The decomposition is $L^2$- and $H^1$-stable, and maps to functions with
     $\operatorname*{supp} u_1 \subset \overline{\omega_{\boldsymbol{z}}}$. Therefore, we can interpret the first component of $J$ as an operator mapping to
     $\widetilde{H}^1(\omega_{\boldsymbol{z}})$ and apply Proposition \ref{interpolation_theorem} (\ref{item:interpolation_theorem:i}) to get:
 
 $$
\left\|u_1\right\|_{\widetilde{H}^{1/2}(\omega_{\boldsymbol{z}})} + \left\|U_{\boldsymbol{z}}\right\|_{\widetilde{H}^{1/2}(\omega_{\boldsymbol{z}})} 
\leq C \left\|u\right\|_{\widetilde{H}^{1/2}(\mathcal{O})}. 
$$

The triangle inequality 
$\left\|u\right\|_{\widetilde{H}^{1/2}(\omega_{\boldsymbol{z}})} \leq 
\left\|u_1\right\|_{\widetilde{H}^{1/2}(\omega_{\boldsymbol{z}})} + 
\left\|U_{\boldsymbol{z}}\right\|_{\widetilde{H}^{1/2}(\omega_{\boldsymbol{z}})}$ then concludes the proof.  

It therefore remains to see that we can construct the operator $J$ of Proposition~\ref{prop:stable_space_decomposition}
with the additional property that $u = u_1 + U_{\boldsymbol{z}}$ if $u$ is such that 
$\operatorname*{supp} u \subset \overline{\omega_{\boldsymbol{z}}}$. This follows by carefully selecting the 
elements on which the local averaging is done, namely, whenever one has to choose an element on which to average, 
one selects, if possible, an element that is {\em not} contained in $\omega_{\boldsymbol{z}}$. 
For example for vertex contributions $\boldsymbol{z'} \in \partial \omega_{\boldsymbol{z}}$ we make sure to use 
elements $K'$ which are not in $\omega_{\boldsymbol{z}}$. This implies that for 
$\operatorname{supp}(u) \subseteq \overline{\omega_{\boldsymbol{z}}}$ we get $u_{\boldsymbol{z'}}=0$. 
A similar choice is made when defining the edge contributions. 
  \end{proof}
\end{corollary}

The stable space decomposition of Proposition~\ref{prop:stable_space_decomposition} is one of several ingredients 
of the proof that the interpolation space obtained by interpolating the space $\widetilde{S}^{p}(\mathcal{T})$ endowed 
with the $L^2$-norm and the $H^1$-norm yields the space $\widetilde{S}^p(\mathcal{T})$ endowed with 
the appropriate fractional Sobolev norm: 
\begin{proposition}[\protect{\cite{melenk_appendix}}]
  \label{prop:characterization_poly_interp}
  Let $s \in (0,1)$ and let $\mathcal{T}$ be a shape regular triangulation of $\Gamma$. 
Let $p \ge 1$.  Then:  
  \begin{align*}
    \left[ \left( \widetilde{S}^{p}(\mathcal{T}), \left\|\cdot\right\|_{L^2(\Gamma)}\right), \left( \widetilde{S}^{p}(\mathcal{T}), \left\|\cdot\right\|_{\widetilde{H}^1(\Gamma)}\right) \right]_s 
    &=\left( \widetilde{S}^{p}(\mathcal{T}), \left\|\cdot\right\|_{\widetilde{H}^s(\Gamma)}\right).
  \end{align*}
  The constants implied in the norm equivalence depend only on $\Gamma$, $s$, and the $\gamma$-shape regularity of $\mathcal{T}$.
\end{proposition}

We note that such a result is clearly valid for fixed $p \geq 1$ (see, e.g., \cite[Proof of Prop. 5]{aff_hypsing}), but the essential observation of 
\cite{melenk_appendix} is that the norm equivalence constants do not depend on $p$.

\subsection{Geometry of vertex patches}

\begin{figure}
\begin{center}
\includegraphics[width=0.4\textwidth]{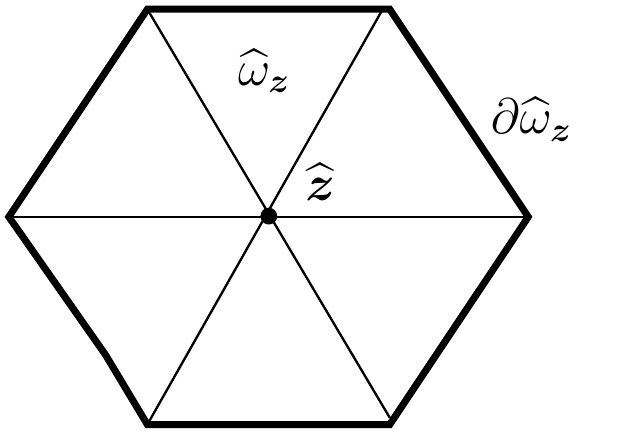}
\caption{\label{fig:reference-patches} Example of an interior reference patch $\widehat\omega_{\boldsymbol{z}}$ }
\end{center}
\end{figure}
We recall that ${\mathcal{V}}$ is the set of all inner vertices, i.e., $\boldsymbol{z} \notin \partial\Gamma$ for all $\boldsymbol{z} \in {\mathcal{V}}$. 
We define the patch size for a vertex $z \in {\mathcal{V}}$ as $h_{\boldsymbol{z}}:=\operatorname{diam}(\omega_{\boldsymbol{z}})$ and
stress that $\gamma$-shape regularity implies $h_{\boldsymbol{z}} \sim h|_K$ for all elements $T \subseteq \omega_{\boldsymbol{z}}$.

Due to shape regularity, the number of elements meeting at a vertex is bounded. 
The following definition allows us to transform the vertex patches $\omega_{\boldsymbol{z}}$ to a finite number of reference configurations.

\begin{definition}[reference patch] 
\label{definition:reference-patches}
Let $\omega_{\boldsymbol{z}}$ be an interior patch consisting of $n$ triangles. 
We may define a Lipschitz continuous bijective map $F_{\boldsymbol{z}}: \widehat \omega_{\boldsymbol{z}} \rightarrow \omega_{\boldsymbol{z}}$, 
where $\widehat \omega_{\boldsymbol{z}} \subseteq \mathbb{R}^2$ is a regular polygon with $n$ edges, 
see Fig.~\ref{fig:reference-patches}.  The map $F_{\boldsymbol{z}}$ is piecewise defined
as a concatenation of affine maps from triangles comprising the regular polygon to the reference element $\widehat K$
with the element maps $F_K$. We note that $F_{\boldsymbol{z}}(\partial\widehat\omega_{\boldsymbol{z}})  = \partial \omega_{\boldsymbol{z}}$. 
\end{definition}

The following lemma tells us how the hypersingular integral operator behaves under
the patch transformation.
\begin{lemma}
\label{item:lemma:reference-patch-hypsing_scaling}
Let $u \in \widetilde{H}^{1/2}(\Gamma)$ with $\operatorname{supp}(u) \subseteq \overline{\omega_{\boldsymbol{z}}}$.
Define $\widehat u:= u \circ F_{\boldsymbol{z}}$ and the integral operator $\widehat{D}$ as
$\widehat{D} \widehat{u}\,(x):=  - \partial_{n_x }^{int} \int_{\widehat{\omega_{\boldsymbol{z}}}}{ \partial_{n_y}^{int} G(x,y) \widehat{u}(y) \; ds(y)}
\quad \text{for } x \in \widehat{\omega}_{\boldsymbol{z}}$, where we treat $\widehat{\omega_{\boldsymbol{z}}} \subseteq \mathbb{R}^2$ as a screen embedded in $\mathbb{R}^3$ and $\partial_{n_x}$
is the derivative in direction of the vector $(0,0,1)$.
Then the hypersingular operator scales like
$$
\left<D u,u \right>_{\omega_{\boldsymbol{z}}} \sim h_{\boldsymbol{z}} \left<\widehat{D} \widehat u,\widehat u \right>_{\widehat{\omega}_{\boldsymbol{z}}},
$$
where the implied constants depend only on $\Gamma$ and  the $\gamma$-shape regularity of $\mathcal{T}$.
\end{lemma}
\begin{proof}
We will prove this in three steps:
\begin{enumerate}[(i)]
  \item $\left<D u,u \right>_{\Gamma} \sim \left\|u\right\|_{\widetilde{H}^{1/2}(\omega_{\boldsymbol{z}})}^2$,
  \item $\left\|u\right\|_{\widetilde{H}^{1/2}(\omega_{\boldsymbol{z}})}^2 \sim h_{\boldsymbol{z}} \left\|\widehat{u}\right\|_{\widetilde{H}^{1/2}(\widehat{\omega}_{\boldsymbol{z}})}^2$,
  \item $\left\|\widehat{u}\right\|_{\widetilde{H}^{1/2}(\widehat{\omega}_{\boldsymbol{z}})}^2 \sim \left<\widehat{D} \widehat{u},\widehat{u}\right>_{\widehat{\omega}_{\boldsymbol{z}}}$.
\end{enumerate}
\emph{Proof of (i):} It is well-known that $D$ is continuous and elliptic on $\widetilde{H}^{1/2}(\omega_{\boldsymbol{z}})$. In
our case, the ellipticity constants can be chosen independently of the patch $\omega_{\boldsymbol{z}}$ and instead only depend on
$\Gamma$. To do so we embed the spaces $\widetilde{H}^{1/2}(\omega_{\boldsymbol{z}})$ into finitely 
many larger spaces $\widetilde{H}^{1/2}(\mathcal{O}_j)$, where the 
sub-surfaces $\mathcal{O}_j$ are {\em open} and for each $\boldsymbol{z}$ there is ${\mathcal O}_j$ 
such that $\omega_{\boldsymbol{z}} \subseteq \mathcal{O}_j\subseteq \Gamma$. 
(For the screen problem we may use the single sub-surface $\mathcal{O}:=\Gamma$, for the case
of closed surfaces we can, for example, use $\mathcal{O}_j:=\Gamma \setminus F_j$, where $F_j$ is the $j$-th face 
of the polyhedron $\Omega$ such that $\omega_{\boldsymbol{z}} \cap F_j = \emptyset$).
It is important that these surfaces are open, since for closed surfaces $\Gamma$ we do not have full ellipticity 
of $D$ but only for $\widetilde{D}$, and the stabilization term has a different scaling behavior.
Since $u$ vanishes outside of $\omega_{\boldsymbol{z}}$ we can use the ellipticity on $\widetilde{H}^{1/2}(\mathcal{O}_j)$ to see 
$\left<Du,u \right>_{\Gamma} \sim \left\|u\right\|_{\widetilde{H}^{1/2}(\mathcal{O}_j)}^2$.
By Corollary \ref{cor:restriction_extension} the norms on $\omega_{\boldsymbol{z}}$ and $\mathcal{O}_j$ are equivalent, which implies the statement (i).

\emph{Proof of (ii):}
The scalings of the $L^2$-norm and $H^1$-seminorm (we can use the seminorm, since we are working on $\widetilde{H}^{1/2}$ of an open surface) is well-known to be 
\begin{align*}
  \left\|u\right\|_{L^2(\omega_{\boldsymbol{z}})} &\sim h_{\boldsymbol{z}} \left\|\widehat{u}\right\|_{L^2(\widehat{\omega}_{\boldsymbol{z}})}, \quad
  \left\|\nabla u\right\|_{L^2(\omega_{\boldsymbol{z}})} \sim \left\|\nabla \widehat{u}\right\|_{L^2(\widehat{\omega}_{\boldsymbol{z}})}.
\end{align*}
The interpolation theorem (Proposition~\ref{interpolation_theorem}, (\ref{item:interpolation_theorem:i})) 
then proves part (ii).

\emph{Proof of (iii):} We again use ellipticity and continuity of $\widehat{D}$. Since there are only finitely many reference patches, the constants can be
chosen independently of the individual patches.
\end{proof}

 \section{Condition number of the $hp$-Galerkin matrix}
\label{sec:p_condition}
In this section we investigate the condition number of the unpreconditioned Galerkin matrix to motivate the need
for good preconditioning.
We will work on the reference triangle $\widehat{K}=\operatorname{conv}\left\{(0,0),(1,0),(0,1)\right\}$ and
will need the following well-known inverse inequalities for polynomials on $\widehat{K}$:
\begin{proposition}[{Inverse inequalities, \cite[Theorem 4.76]{schwab_p_fem}}]
  Let $\widehat{K}$ denote the reference triangle and let ${\mathcal E}$ be one of its edges.
  There exists a constant $C$ such that for all $p \in \mathbb{N}$ and for all $v \in P^p(\widehat{K})$ the following 
  estimates hold:
  \begin{align}
    \left\|v\right\|_{L^{\infty}(\widehat{K})}&\leq C p^2 \left\|v\right\|_{L^2(\widehat{K})} \label{eqn:inv_est_infty_l2}, \\
    \left\|v\right\|_{L^{\infty}(\widehat{K})}&\leq C \sqrt{\log(p+1)} \left\|v\right\|_{H^1(\widehat{K})} \label{eqn:inv_est_infty_h1}, \\
    \left\|v\right\|_{H^1(\widehat{K})}&\leq C p^2 \left\|v\right\|_{L^2(\widehat{K})} \label{eqn:inv_est_h1_l2}, \\
    \left\|v\right\|_{L^2({\mathcal E})}&\leq C p \left\|v\right\|_{L^2(\widehat{K})} \label{eqn_inv_est_trace}. 
  \end{align}
  \qed
\end{proposition}

First we investigate the $L^2$ and $H^1$ conditioning of our basis on the reference triangle.
\begin{lemma}
  \label{condition_upper_bound_l2}
  Let $u \in P^p(\widehat{K})$ and let $\alpha^{\mathcal{V}}_{j}$, $\alpha_j^{\mathcal{E}_m}$, $\alpha^\mathcal{I}_{(ij)}$ be the coefficients with respect to
  the basis in Definition~\ref{definition_basis}, i.e., 
  we decompose $u= u_{\mathcal{V}} + u_{\mathcal{E}_1} + u_{\mathcal{E}_2} + u_{\mathcal{E}_3} + u_\mathcal{I}$ with  
    \begin{align}
  \label{eq:condition_upper_bound_l2-10}
      u_{\mathcal{V}}&=\sum_{j=1}^{3}{\alpha^{\mathcal{V}}_j \varphi^{\mathcal{V}}_j}, \quad
      u_{\mathcal{E}_m}=\sum_{j=0}^{p-2}{\alpha^{\mathcal{E}_m}_j \varphi^{\mathcal{E}_m}_j}, \quad
      u_{\mathcal{I}}=\sum_{i+j \leq p-3}{\alpha^\mathcal{I}_{(ij)} \varphi^{\mathcal{I}}_{(ij)} }.
    \end{align}
  Then for a constant $C > 0$ that does not depend on $u$ or $p$:
  \begin{align}
    \left\|u_{\mathcal{V}}\right\|_{L^2(\widehat{K})}^2 &  \leq C \sum_{j=1}^{3}{\left|\alpha^{\mathcal{V}}_j\right|^2}, 
    &\left\|u_{\mathcal{E}_m}\right\|_{L^2(\widehat{K})}^2 &  \leq C \sum_{j=0}^{p-2}{\left|\alpha^{\mathcal{E}_m}_j\right|^2}, 
    &\left\|u_{\mathcal{I}}\right\|_{L^2(\widehat{K})}^2 &  = \sum_{i+j\leq p-3}{\left|\alpha^{\mathcal{I}}_{(ij)}\right|^2}. \label{eqn:inner_l2_eqality}
  \end{align}
  Combined this gives:
  \begin{align}
    \label{eqn:condition_upper_bound_l2_full}
    \left\|u\right\|_{L^2(\widehat{K})}^2 &\leq C \Bigl( \sum_{j=1}^{3}{\left|\alpha^{\mathcal{V}}_{j}\right|^2} + \sum_{m=1}^{3}{\sum_{j=0}^{p-2}{ \left|\alpha_j^{\mathcal{E}_m}\right|}^2 }+  \sum_{i+j \leq p-3}{\left|\alpha^\mathcal{I}_{(ij)}\right|^2}\Bigr).
  \end{align}
  \begin{proof}
    The estimate for $u_{\mathcal{V}}$ in (\ref{eqn:inner_l2_eqality}) is clear. 
    For the edge contributions in (\ref{eqn:inner_l2_eqality}), we restrict ourselves to the edge $(0,0)-(1,0)$, i.e., $m=1$ and drop the index $m$ in the notation. The other edges can be treated analogously. 
    \begin{align*}
      \left\|u_{\mathcal{E}}\right\|_{L^2(\widehat{K})}^2 &= \left\|\sum_{j=0}^{p-2}{ \alpha^{\mathcal{E}}_j 
\sqrt{\frac{2i+3}{2}} L^{\mathcal{S}}_{i+2}(\lambda_1 - \lambda_2, \lambda_1 + \lambda_2)}\right\|^2_{L^2(\widehat{K})}  \\
      &= 
      \sum_{i,j=0}^{p-2}{ \alpha^\mathcal{E}_j \alpha^\mathcal{E}_i
        \int_{\widehat{K}}{ \sqrt{\frac{2i+3}{2}} L^{\mathcal{S}}_{i+2}(\lambda_1 - \lambda_2, \lambda_1 + \lambda_2)} \sqrt{\frac{2j+3}{2}}L^{\mathcal{S}}_{j+2}(\lambda_1 - \lambda_2, \lambda_1 + \lambda_2)} dx\\
      &= \frac{1}{4} \sum_{i,j=0}^{p-2}{\alpha^\mathcal{E}_j \alpha^\mathcal{E}_i
        \int_{-1}^{1}{ \int_{-1}^{1}{\sqrt{\frac{2i+3}{2}}
        L_{i+2}(\xi) \sqrt{\frac{2j+3}{2}}L_{j+2}(\xi) \left(\frac{1-\eta}{2}\right)^{i+j+5}   d\xi d\eta }}}.
    \end{align*}
    In the last step we transformed the reference triangle to $(-1,1)\times (-1,1)$
    via the map 
    ${(\xi,\eta)\mapsto(\frac{1}{4}(1+\xi)(1-\eta),\frac{1}{2}(1+\eta))}$.

    It is well-known (see \cite[p.~{65}]{schwab_p_fem}) that the 1D-mass matrix M of 
the integrated Legendre polynomials is 
    pentadiagonal, and the non-zero entries satisfy $\left|M_{(ij)}\right| \sim \frac{1}{(i+1)\,(j+1)}$. 
    It is easy to check that $ \left|\int_{-1}^{1}{\left(\frac{1-\eta}{2}\right)^{i+j+5} d\eta}\right| \leq C (i+j+6)^{-1}$.
    Together with a Cauchy-Schwarz estimate, we obtain:
    \begin{align*}
      \left\|u_{\mathcal{E}}\right\|_{L^2(\widehat{K})}^2 &
      \lesssim \sum_{j=0}^{p-2}{\left|\alpha^{\mathcal{E}}_j\right|^2 \frac{1}{(j+1)^3}}
      \lesssim \sum_{j=0}^{p-2}{\left|\alpha^{\mathcal{E}}_j\right|^2}.
    \end{align*}

    The bubble basis functions are chosen $L^2$-orthogonal. Thus, using our scaling of the 
bubble basis functions, 
    \begin{align*}
      \left\|u_{\mathcal{I}}\right\|_{L^2(\widehat{K})}^2 & = 
\sum_{i+j \leq p-3}{ \left|\alpha^\mathcal{I}_{(ij)}\right|^2 \left\|\varphi_{(ij)}\right\|_{L^2(\widehat{K})}^2} = 
\sum_{i+j \leq p-3}  \left|\alpha^\mathcal{I}_{(ij)}\right|^2 . 
    \end{align*}
    Finally, we split the function $u$ into vertex, edge and inner components, apply the triangle inequality and get
    \begin{align*}
      \left\|u\right\|_{L^2(\widehat{K})}^2 &\leq 5 \left\|u_{\mathcal{V}}\right\|^2_{L^2(\widehat{K})} + 5 \sum_{m=1}^{3}{\left\|u_{\mathcal{E}_j}\right\|_{L^2(\widehat{K})}^2} + 5 \left\|u_\mathcal{I}\right\|_{L^2(\widehat{K})}^2, 
    \end{align*}    
    which is \eqref{eqn:condition_upper_bound_l2_full}.
  \end{proof}
\end{lemma}
More interesting are the reverse estimates.
\begin{lemma}
  \label{condition_lower_bound}
There is a constant $C > 0$ independent of $p$ such for every $u \in P^p(\widehat{K})$ the coefficients of its 
representation in the basis of Definition~\ref{definition_basis} as in Lemma~\ref{condition_upper_bound_l2} satisfy: 
     \begin{align*}
\mbox{vertex parts:} &\qquad     \sum_{j=1}^{3}{\left|\alpha^{\mathcal{V}}_j\right|^2} \leq C  p^{4} \left\|u\right\|_{L^2(\widehat{K})}^2 \quad  \text{ as well as } \quad
    \sum_{j=1}^{3}{\left|\alpha^{\mathcal{V}}_j\right|^2} \leq C \log(p+1) \left\|u\right\|_{H^1(\widehat{K})}^2,\\
  \mbox{edge parts:} & \qquad 
    \sum_{j=0}^{p-2}{\left|\alpha^{\mathcal{E}_m}_j\right|^2} \leq C p^6 \left\|u\right\|_{L^2(\widehat{K})}^2 \quad \text{ and } \quad
    \sum_{j=0}^{p-2}{\left|\alpha^{\mathcal{E}_m}_j\right|^2} \leq C p^2 \left\|u\right\|_{H^1(\widehat{K})}^2.
  \end{align*}
  Moreover, if $u$ vanishes on $\partial \widehat{K}$, then 
  \begin{align}
    \label{eq:condition_inner_l2_equality}
    \sum_{i+j\leq p-3}{\left|\alpha^{\mathcal{I}}_{(ij)}\right|^2}&= \left\|u\right\|_{L^2(\widehat{K})}^2. 
  \end{align}
  \begin{proof}
    Since $\alpha^{\mathcal{V}}_j=u(\boldsymbol{z}_j)$ where $\boldsymbol{z}_j$ denotes the $j-th$ vertex,
    we can use the $L^\infty$-inverse estimates \eqref{eqn:inv_est_infty_l2} and \eqref{eqn:inv_est_infty_h1} to get estimates
    for the vertex part.

    For the edge parts, we again only consider the bottom edge, ${\mathcal E} = {\mathcal E}_m$ with $m=1$.
    First we assume that $u$ vanishes in all vertices. If we consider the restriction of $u$ to the edge ${\mathcal E}$
    we only have contributions by the edge basis, i.e., we can write
    \begin{align*}
      u\left(x,0\right)&=\sum_{i=0}^{p-2}{\alpha^{\mathcal{E}}_i \sqrt{\frac{2i+3}{2}} L_{i+2}( 2x-1)}, \quad x \in (0,1), \\
      \frac{\partial}{\partial x} u(x,0)&=\sum_{i=0}^{p-2}{\alpha^{\mathcal{E}}_i 2\;\sqrt{\frac{2 i+3}{2}} \ell_{i+1}(2x-1)}, \quad x \in (0,1). 
    \end{align*}
    
    The factor was chosen to get an $L^2$-normalized basis, since we have $\left\|\ell_{i+1}\right\|_{L^2(-1,1)}^2=\frac{2}{2 i+3}$.
    The Legendre polynomials are orthogonal on $(-1,1)$, and therefore simple calculations show
    \begin{align*}
      \left\|\frac{\partial u}{\partial x}\right\|_{L^2(\mathcal{E})}^2&=2\sum_{i=0}^{p-2}{\left|\alpha^{\mathcal{E}}_i\right|^2}.
    \end{align*}

    If we consider a general $u \in P^p(\widehat{K})$, we apply the previous estimate to $u_2:=u - I^1 u$ where $I^1$ denotes the
    nodal interpolation operator to the linears. Then we get from the triangle inequality
    \begin{align*}
      \sum_{i=0}^{p-2}{\left|\alpha^{\mathcal{E}}_i\right|^2} &
      \leq 2\left\|\frac{\partial u}{\partial x}\right\|_{L^2(\mathcal{E})}^2 + 
      2 \left\|\frac{\partial I^1u}{\partial x}\right\|_{L^2(\mathcal{E})}^2.
    \end{align*}
    We apply the trace estimate \eqref{eqn_inv_est_trace} to the first part
    and the trace and norm equivalence for the second. We obtain:
    \begin{align*}
      \sum_{i=0}^{p-2}{\left|\alpha^{\mathcal{E}}_i\right|^2} &\lesssim  p^2 \left\|\frac{\partial u}{\partial x}\right\|_{L^2(\widehat{K})}^2 + \left\|I^1u\right\|_{L^2(\widehat{K})}^2.
    \end{align*}
    The $H^1$ estimate then follows from the $L^\infty$ estimate 
    for the nodal interpolant \eqref{eqn:inv_est_infty_h1}. 
    For the $L^2$ estimate we then simply use the inverse estimate \eqref{eqn:inv_est_h1_l2}.

    For the equality \eqref{eq:condition_inner_l2_equality} we note that if $u|_{\partial T} = 0$ then 
    $u=u_{\mathcal{I}}$ and thus we can use the equality in \eqref{eqn:inner_l2_eqality}.
  \end{proof}
\end{lemma}

\begin{lemma}
\label{thm:condition_estimates_tref}
  There exist constants $c_0$, $C_0$, $c_1$, $C_1 >0$ independent of $p$ such that
  for every $u \in P^p(\widehat{K})$ its coefficients in the basis
of Definition~\ref{definition_basis} as in Lemma~\ref{condition_upper_bound_l2} satisfy: 
  \begin{align}
    c_0 \left\|u\right\|_{L^2(\widehat{K})}^2 &\leq \sum_{j=1}^{3}{\left|\alpha^{\mathcal{V}}_{j}\right|^2} + \sum_{m=1}^{3}{\sum_{j=0}^{p-2}{ \left|\alpha_j^{\mathcal{E}_m}\right|}^2 }+
    \sum_{i+j \leq p-3}{\left|\alpha^\mathcal{I}_{(ij)}\right|^2} \leq C_0 p^{6} \left\|u\right\|_{L^2(\widehat{K})}^2,  \label{eqn:condition_estimates_tref_l2}\\
    c_1 p^{-4} \left\|u\right\|_{H^1(\widehat{K})}^2 &
    \leq \sum_{j=1}^{3}{\left|\alpha^{\mathcal{V}}_{j}\right|^2} + \sum_{m=1}^{3}{\sum_{j=0}^{p-2}{ \left|\alpha_j^{\mathcal{E}_m}\right|}^2 }+ \sum_{i+j \leq p-3}{\left|\alpha^\mathcal{I}_{(ij)}\right|^2}
    \leq C_1 p^{2} \left\|u\right\|_{H^{1}(\widehat{K})}^2  \label{eqn:condition_estimates_tref_h1}.
  \end{align}
  \begin{proof}[Proof of \eqref{eqn:condition_estimates_tref_l2}: \nopunct]
    The lower bound was already shown in Lemma~\ref{condition_upper_bound_l2}.
    For the upper bound we apply the preceding Lemma~\ref{condition_lower_bound} to $u$ and see
    $\sum_{j=1}^{3}{\left|\alpha^{\mathcal{V}}_j\right|^2} \lesssim  p^{4} \left\|u\right\|_{L^2(\widehat{K})}^2$.
    For any edge ${\mathcal E}_m$ we get
    \begin{align*}
      \sum_{j=0}^{p-2}{\left|\alpha^{\mathcal{E}_m}_j\right|^2} &\lesssim p^6 \left\|u\right\|_{L^2(\widehat{K})}^2.
    \end{align*}
    Next we set $u_2:=u- u_{\mathcal{V}} - u_{\mathcal{E}}$, where $u_{\mathcal{E}}$ is the sum of the edge contributions $u_{\mathcal{E}_m}$. This function vanishes on the boundary of $\widehat{K}$ and we can apply Lemma~\ref{condition_lower_bound} to get:
    \begin{align*}
      \sum_{i+j\leq p-3}{\left|\alpha^{\mathcal{I}}_{(ij)}\right|^2}&\lesssim \left\|u_2\right\|_{L^2(\widehat{K})}^2.
    \end{align*}
    Since we can always estimate the $L^2$ norms by the $\ell^2$ norms of the coefficients 
(Lemma~\ref{condition_upper_bound_l2}), we obtain 
    \begin{align*}
      \left\|u_2\right\|_{L^2(\widehat{K})}^2&\leq 3\left\|u\right\|_{L^2(\widehat{K})}^2 + 3\left\|u_{\mathcal{V}}\right\|_{L^2(\widehat{K})}^2 + 3\left\|u_{\mathcal{E}}\right\|_{L^2(\widehat{K})}^2
      \lesssim \left\|u\right\|_{L^2(\widehat{K})}^2 +  p^6 \left\|u\right\|_{L^2(\widehat{K})}^2 + p^4 \left\|u\right\|_{L^2(\widehat{K})}^2.
    \end{align*}
    {\em Proof of \eqref{eqn:condition_estimates_tref_h1}:} The proof for the upper $H^1$ estimate works along the same lines, but using the sharper $H^1$-estimates from 
    Lemma \ref{condition_lower_bound} for the vertex and edge parts.
    For the lower estimate, we just make use of the inverse estimate \eqref{eqn:inv_est_h1_l2} 
    and the fact that the $L^2$-norm is uniformly bounded by the coefficients, to conclude the proof.
  \end{proof}
\end{lemma}

\begin{figure}
\includegraphics{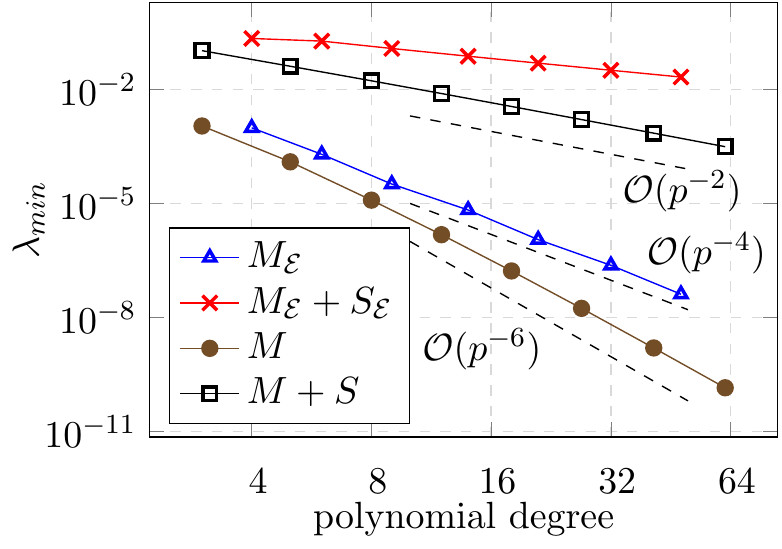}
\includegraphics{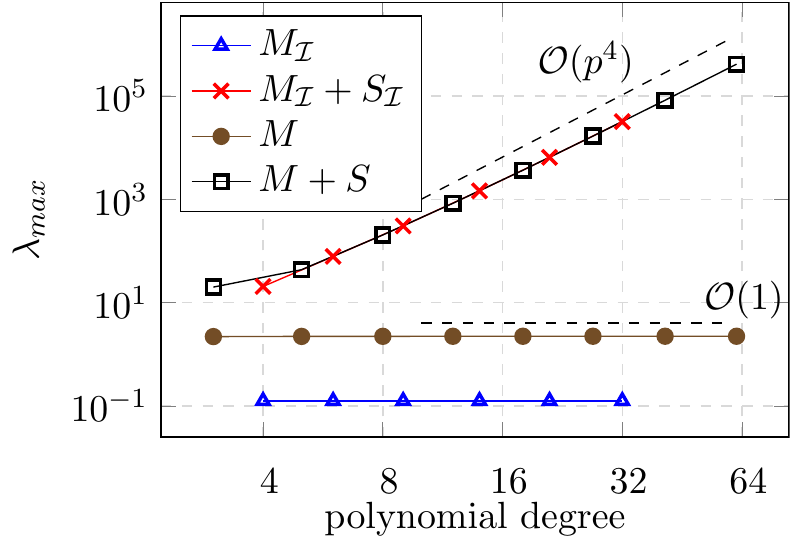}
\caption{Numerical computation of the extremal eigenvalues of the mass matrices and the sum of mass and $H^1$-stiffness matrix $M+S$
for the full system and different sub-blocks.}
\label{fig:numerical_comparison_tref_condition}
\end{figure}

\begin{example}
In Figure~\ref{fig:numerical_comparison_tref_condition} we compare our theoretical bounds on the reference element 
from Lemma~\ref{thm:condition_estimates_tref} with a numerical experiment that studies the maximal and minimal
eigenvalues of the mass matrix $M$ and the stiffness matrix $S$ (corresponding to the bilinear form 
$(\nabla \cdot,\nabla \cdot)_{L^2(\widehat{K})}$). 
We focus on the full system and the subblocks that contributed the highest order in our theoretical
investigations, i.e. the edge blocks $M_{\mathcal{E}}$, $S_{\mathcal{E}}$, and the block of inner basis functions
$M_{\mathcal{I}}$, $S_{\mathcal{I}}$. We see that the estimates
on the full condition numbers are not overly pessimistic: the numerics show a
behavior of the minimal eigenvalue of $\mathcal{O}(p^{5.5})$ instead of $\mathcal{O}(p^6)$. If we focus solely on the edge contributions, we see that the
bound we used for the lower eigenvalue is not sharp there. This can partly be explained by the fact that if no inner basis functions are present
it is possible to improve the estimate \eqref{eqn_inv_est_trace} by a factor of $p$. But since we also need to 
include the coupling of inner and edge basis functions this improvement in order is lost again when looking at the 
full systems.
\hbox{}\hfill\rule{0.8ex}{0.8ex}
\end{example}

The estimates on the reference triangle can now be transferred to the global space $\widetilde{S}^p(\mathcal{T})$ on quasiuniform meshes.
\begin{theorem}
\label{thm:condition_numbers_h1}
Let $\mathcal{T}$ be a quasiuniform triangulation with mesh size $h$. With the polynomial basis on the reference
triangle $\widehat{K}$ given by Definition~\ref{definition_basis}, let $\{\varphi_i\,|\, i=1,\ldots,N\}$ be 
the basis of $\widetilde{S}^p(\mathcal{T})$. 
Then there exist constants $c_0$, $c_{1/2}$, $c_1$, $C_0$, $C_{1/2}$, $C_1>0$ that depend only
on $\Gamma$ and the $\gamma$-shape regularity of $\mathcal{T}$, such that, for every 
$\mathfrak{u} \in \mathbb{R}^N$ and $u =\sum_{j=1}^{N}{\mathfrak{u}_j \varphi_j} \in \widetilde S^p(\mathcal{T})$: 
\begin{align}
  c_0 \frac{1}{h^2} \left\|u\right\|_{L^2(\Gamma)}^2 &\leq \left\|\mathfrak{u}\right\|_{\ell^2}^2  \leq C_0 \frac{p^6}{h^2}\left\|u\right\|_{L^2(\Gamma)}^2,  \label{eqn:condition_numbers_l2}\\
  c_1 p^{-4} \left\|u\right\|_{H^1(\Gamma)}^2 &\leq \left\|\mathfrak{u}\right\|_{\ell^2}^2  \leq C_1 \left(p^2 + h^{-2}\right) \left\|u\right\|_{H^1(\Gamma)}^2, \label{eqn:condition_numbers_h1}\\
  c_{1/2} h^{-1}\,p^{-2}  \left\|u\right\|_{\widetilde{H}^{1/2}(\Gamma)}^2 &\leq \left\|\mathfrak{u}\right\|_{\ell^2}^2  \leq C_{1/2} \left(\frac{p^4}{h} + h^{-2}\right)\left\|u\right\|_{\widetilde{H}^{1/2}(\Gamma)}^2. \label{eqn:condition_numbers_honehalf}
\end{align}
\begin{proof}
  The $L^2$-estimate (\ref{eqn:condition_numbers_l2}) can easily be shown by transforming to 
  the reference element and applying \eqref{eqn:condition_estimates_tref_l2}.

  To prove the other estimates (\ref{eqn:condition_numbers_h1}), (\ref{eqn:condition_numbers_honehalf}), 
  we need the Scott-Zhang projection operator 
$J_{h}:L^2(\Gamma) \rightarrow \widetilde S^1(\mathcal{T})$ as modified in \cite[Section 3.2]{aff_hypsing}. 
It has the following important properties:
  \begin{enumerate}
    \item $J_{h}$ is a bounded linear operator from  $L^2(\Gamma) $ to 
    $(\widetilde{S}^1(\mathcal{T}),\|\cdot\|_{L^2(\Gamma)})$. 
    \item For every $s \in [0,1]$ there holds $\left\|J_{h} v\right\|_{\widetilde{H}^{s}(\Gamma)} \leq C_{stab}(s) \left\|v\right\|_{\widetilde{H}^s(\Gamma)}$ $\forall v \in \widetilde{H}^s(\Gamma)$. 
    \item For every $K \in \mathcal{T}$ let $\omega_K:=\bigcup\left\{K' \in \mathcal{T}: K \cap K' \neq \emptyset\right\}$ 
denote the element patch, i.e., the union of all elements that touch $K$. Then, for all 
$v \in \widetilde{H}^1(\Gamma)$ 
      \begin{align}
\label{scott_zhang_local_l2_approx}
        \left\|\left(1-J_h\right) v\right\|_{L^2(K)}&\leq C_{sz} h_K \left\|\nabla v\right\|_{L^2(\omega_K)},  \\
\label{scott_zhang_local_h1_stab}
        \left\|\nabla \left(1-J_h\right) v\right\|_{L^2(K)}&\leq C_{sz} \left\|\nabla v\right\|_{L^2(\omega_K)}.
      \end{align}
  \end{enumerate}
  The constant $C_{sz}$ depends only on the $\gamma$-shape regularity of $\mathcal{T}$, and $C_{stab}(s)$ additionally
  depends on $\Gamma$ and $s$.

  We will use the following notation: 
  For a function $u \in \widetilde{S}^p(\mathcal{T})$ we will write $\mathfrak{u} \in \mathbb{R}^{N}$ for its 
  representation in the basis $\{\varphi_i\,|\, i=1,\ldots,N\}$. 
  For an element $K \in \mathcal{T}$ we write $\mathfrak{u}|_K$ for
  the part of the coefficient vector that belongs to basis functions whose support intersects the interior of $K$. 
  In addition to the function $u \in \widetilde{S}^p(\mathcal{T})$, we will employ the function $\tilde u:= u - J_h u$. 
Its vector
  representation will be denoted $\tilde {\mathfrak{u}} \in \mathbb{R}^{N}$. Finally, the vector representation 
  of $J_h u$ (again $u \in \widetilde{S}^p(\mathcal{T})$) will be $\mathfrak{J}_h \mathfrak{u} \in \mathbb{R}^{N}$. 

{\em 1.~step:}
  We claim the following stability estimates:
  \begin{align}
    \left\|\mathfrak{J}_h \mathfrak{u}\right\|_{\ell^2}^2&\lesssim h^{-2} \, \left\|J_h u\right\|_{L^2(\Gamma)}^2 \lesssim h^{-2} \left\|u\right\|_{L^2(\Gamma)}^2 \lesssim \left\|\mathfrak{u}\right\|_{\ell^2}^2, 
    \label{est_jhu_coeff}\\
    \left\|J_h u\right\|_{\widetilde{H}^{1/2}(\Gamma)}^2&\lesssim~ h^{-1} \left\|J_h u\right\|_{L^{2}(\Gamma)}^2,
    \label{est_jhu_coeff-foo}\\
    \left\|J_h u\right\|_{\widetilde{H}^{1/2}(\Gamma)}^2&\lesssim h \left\|\mathfrak{u}\right\|_{\ell^2}^2 \label{est_jhu_honehalf}.
  \end{align}
  The inequalities (\ref{est_jhu_coeff}) are just a simple scaling argument combined with the $L^2$ stability 
  of the Scott-Zhang projection and \eqref{eqn:condition_numbers_l2}.
  The inequality (\ref{est_jhu_coeff-foo}) follows from the inverse inequality (note that $J_h u$ has degree 1). 
  Finally, \eqref{est_jhu_honehalf} follows from combining (\ref{est_jhu_coeff-foo}) and \eqref{est_jhu_coeff}.
 
{\em 2.~step:} 
  Next, we investigate the function $\tilde{u}=u - J_h u$. We claim the following estimates:
  \begin{align}
    \left\|\tilde{u}\right\|_{L^2(\Gamma)}^2 &\lesssim h^2 \left\|\mathfrak{u}\right\|_{\ell^2}^2 \label{est_utilde_l2}, \\
    \left\|\tilde{u}\right\|_{H^1(\Gamma)}^2 &\lesssim p^4 \left\|\mathfrak{u}\right\|_{\ell^2}^2 \label{est_utilde_h1}, \\
    \left\|\tilde{u}\right\|_{\widetilde{H}^{1/2}(\Gamma)}^2 &\lesssim p^2\,h \left\|\mathfrak{u}\right\|_{\ell^2}^2 \label{est_utilde_honehalf}.
  \end{align}
  The estimate \eqref{est_utilde_l2} 
is a simple consequence of
  \eqref{eqn:condition_numbers_l2} and the $L^2$-stability of the Scott-Zhang operator $J_h$. 
  For the proof of \eqref{est_utilde_h1},  we combine a simple scaling argument 
  with \eqref{eqn:condition_estimates_tref_h1} and the stability estimate \eqref{est_jhu_coeff} to get
  \begin{align*}
    \left\|\tilde{u}\right\|_{H^1(\Gamma)}^2 
&\stackrel{~\eqref{eqn:condition_estimates_tref_h1}}{\lesssim}
  p^4 \left\|\mathfrak{\tilde{u}}\right\|_{\ell^2}^2
  \stackrel{~\eqref{est_jhu_coeff}}{\lesssim} p^4 \left\|\mathfrak{u}\right\|_{\ell^2}^2.
  \end{align*}
  The bound \eqref{est_utilde_honehalf} follows from the interpolation estimate of 
  Proposition~\ref{interpolation_theorem}, (\ref{item:interpolation_theorem:ii}) and the 
  estimates \eqref{est_utilde_l2}--\eqref{est_utilde_h1}.

 {\em 3.~step:}  We assert: 
  \begin{align}
    \left\|\mathfrak{\tilde{u}}\right\|_{\ell^2}^2&\lesssim \frac{p^6}{h^2} \left\|u\right\|_{L^2(\Gamma)}^2 \label{est_utilde_coeff_l2},\\
    \left\|\mathfrak{\tilde{u}}\right\|_{\ell^2}^2&\lesssim p^2 \left\|u\right\|_{H^1(\Gamma)}^2 \label{est_utilde_coeff},\\
    \left\|\mathfrak{\tilde{u}}\right\|_{\ell^2}^2&\lesssim \frac{p^4}{h} \left\|u\right\|_{\widetilde{H}^{1/2}(\Gamma)}^2 \label{est_utilde_coeff_honehalf}.
  \end{align}
 Again, \eqref{est_utilde_coeff_l2} is a simple consequence of 
  \eqref{eqn:condition_numbers_l2} and the $L^2$-stability of the Scott-Zhang operator $J_h$. 
 For the bound \eqref{est_utilde_coeff} we calculate, 
  using the equivalence \eqref{eqn:condition_estimates_tref_h1} of the coefficient vector and the $H^1$-norm 
  on the reference triangle, together with the scaling properties of the $H^1$- and $L^2$-norms,
  \begin{align*}
    \left\|\mathfrak{\tilde{u}}\right\|_{\ell^2}^2&\leq\sum_{K \in \mathcal{T}}{ \left\|\mathfrak{\tilde{u}}|_K \right\|_{\ell^2}^2}
    \stackrel{~\eqref{eqn:condition_estimates_tref_h1}, \text{scaling}}{\lesssim} p^2 \sum_{K \in \mathcal{T}}{\left(  h^{-2} \left\|\tilde{u}\right\|_{L^2(K)}^2 + \left|\tilde{u}\right|_{H^1(K)}^2\right)}.
  \end{align*}
  By applying the local $L^2$-interpolation estimate \eqref{scott_zhang_local_l2_approx} and $H^1$-stability \eqref{scott_zhang_local_h1_stab} we get

  \begin{align*}
    \left\|\mathfrak{\tilde{u}}\right\|_{\ell^2}^2&\lesssim p^2 \sum_{K \in \mathcal{T}}{ \left\|u\right\|_{H^1(\omega_K)}^2 }
    \lesssim p^2 \left\|u\right\|_{H^1(\Gamma)}^2, 
  \end{align*}
  where in the last step we used the fact for shape regular meshes each element is contained in at most $M$ different patches, where $M$ depends solely on the shape regularity constant $\gamma$.

  We next prove \eqref{est_utilde_coeff_honehalf}. 
  We apply Proposition~\ref{interpolation_theorem}, (\ref{item:interpolation_theorem:i}) 
  to the map $\ell^2 \to \widetilde{S}^p(\mathcal{T}):\; \mathfrak{u} \mapsto u$, where the space $\widetilde{S}^p(\mathcal{T})$ is
  once equipped with the $L^2$- and once with the $H^1$-norm. By Proposition~\ref{prop:characterization_poly_interp} 
  interpolating between \eqref{est_utilde_l2} and~\eqref{est_utilde_h1} yields 
   $\left\|\mathfrak{\tilde{u}}\right\|_{\ell^2}^2 \lesssim \left( p^6 h^{-2} \,  p^2\right)^{1/2} \left\|u\right\|^2_{\widetilde{H}^{1/2}(\Gamma)} \lesssim p^4 h^{-1} \left\|u\right\|^2_{\widetilde{H}^{1/2}(\Gamma)}$.

{\em 4.~step:}
  The above steps allow us to obtain the $H^1$ and $H^{1/2}$ estimates (\ref{eqn:condition_numbers_h1})---(\ref{eqn:condition_numbers_honehalf}) of the theorem.
  We decompose $u = \tilde{u} + J_h u$ and correspondingly 
  $\mathfrak{u}=\mathfrak{\tilde{u}}+ \mathfrak{J}_h \mathfrak{u}$. Then:

  \begin{align*}
    \left\|\mathfrak{u}\right\|_{\ell^2}^2 
    \lesssim \left\|\tilde {\mathfrak{u}}\right\|_{\ell^2}^2 + \left\|\mathfrak{J}_h \mathfrak{u}\right\|_{\ell^2}^2,
    &\stackrel{(\ref{est_utilde_coeff_honehalf}), (\ref{est_jhu_coeff})}{\lesssim} \frac{p^4}{h} \left\|u\right\|_{\widetilde{H}^{1/2}(\Gamma)}^2 + \frac{1}{h^2} \left\|u\right\|_{L^2(\Gamma)}^2 \lesssim \frac{p^4\,h+1}{h^2} \left\|u\right\|_{\widetilde{H}^{1/2}(\Gamma)}^2, \\
    \left\|u\right\|_{\widetilde{H}^{1/2}(\Gamma)}^2  \lesssim 
    \left\|\tilde u\right\|_{\widetilde{H}^{1/2}(\Gamma)}^2  + \left\|J_h u\right\|_{\widetilde{H}^{1/2}(\Gamma)}^2 
&\stackrel{(\ref{est_utilde_honehalf}), (\ref{est_jhu_honehalf})}{\lesssim} p^2 \,h \left\|\mathfrak{u}\right\|_{\ell^2}^2 + h \left\|\mathfrak{u}\right\|_{\ell^2}^2 \lesssim h\,p^2 \left\|\mathfrak{u}\right\|_{\ell^2}^2.
  \end{align*}
This shows (\ref{eqn:condition_numbers_honehalf}).  The $H^1$ estimate (\ref{eqn:condition_numbers_h1}) 
follows along the same lines: An elementwise inverse estimate gives 
$$
\left\|u\right\|^2_{H^1(\Gamma)} \lesssim \frac{p^4}{h^2} \left\|u\right\|^2_{L^2(\Gamma)} 
\stackrel{~\eqref{eqn:condition_numbers_l2}}{\lesssim} p^4 \|\mathfrak{u}\|^2_{\ell^2},
$$
and the splitting $\mathfrak{u} = \tilde{\mathfrak{u}} + \mathfrak{J}_h\mathfrak{u}$ produces 
\begin{align*}
\left\|\mathfrak{u}\right\|^2_{\ell^2} \lesssim 
\left\|\tilde{\mathfrak{u}}\right\|^2_{\ell^2} + 
\left\|\mathfrak{J}_h \mathfrak{u}\right\|^2_{\ell^2} 
{\stackrel{(\ref{est_utilde_coeff}), (\ref{est_jhu_coeff})}\lesssim}
\left( p^2 + h^{-2} \right) \left\|u\right\|^2_{H^1(\Gamma)}. 
\tag*{\qedhere}
\end{align*}
\end{proof}
\end{theorem}

\begin{corollary}
\label{thm:condition_numbers_D}
  The spectral condition number of the unpreconditioned Galerkin matrix $\widetilde D^p_h$ can be bounded by
  \begin{align*}
    \kappa \left(\widetilde{D}^p_h\right) &\leq C \left(\frac{p^2}{h} + p^6 \right)
  \end{align*}
  with a constant $C>0$ that depends only on $\Gamma$ and the $\gamma$-shape regularity of $\mathcal{T}$.
  \begin{proof}
    The bilinear form induced by the stabilized hypersingular operator is elliptic and continuous with respect 
    to the $\widetilde{H}^{1/2}$-norm. 
    By applying the estimates \eqref{eqn:condition_numbers_honehalf} to the Rayleigh quotients we get the stated result.
  \end{proof}
\end{corollary}

\begin{remark}
  In this section we did not consider the effect of diagonal scaling. The numerical results in 
  Section~\ref{sec:numerics} suggest that it improves the $p$-dependence of the condition number significantly.
\hbox{}\hfill\rule{0.8ex}{0.8ex}
\end{remark}

\subsection{Quadrilateral meshes}
  The present paper focuses on meshes consisting of triangles. Nevertheless, in order to put the results 
  of Section~\ref{sec:p_condition} in perspective, we include a short section on quadrilateral meshes. 
  In \cite{maitre_pourquier}, estimates similar to those of Lemma~\ref{thm:condition_estimates_tref}
  have been derived for the case of the Babu\v{s}ka-Szab\'o basis on the reference square $\widehat{S}:=[-1,1]^2$. 
  We give a brief summary of the definitions and results.
  \begin{definition}[Babu\v{s}ka-Szab\'o basis]
    \label{def:babuska_szabo_basis}
\begin{enumerate}[(i)]
\item
    On the reference interval $\widehat{I}:=[-1,1]$,  the basis functions are based on the integrated 
    Legendre polynomials $L_j$ as defined in \eqref{eq:def_legendre_poly}:
    \begin{align*}
      \varphi_0(x)&:=\frac{1}{2} (1-x), & \varphi_1(x)&:=\frac{1}{2}(1+x), &  
      \varphi_j(x)&:=\frac{1}{\left\|\ell_{j-1}\right\|_{L^2(I)}} L_{j}(x) \quad \forall \, 2 \leq j \leq p.
    \end{align*}
\item
    On the reference square $\widehat S$, the basis of the ``tensor product space'' ${\mathcal Q}^p(\widehat S)$ 
    is given by 
     the tensor product of the 1D basis functions: ${\{ \varphi_i \otimes \varphi_j : 0 \leq i,j \leq p \}}$. 
\end{enumerate}
  \end{definition}

  For this basis, the following estimates hold:
  \begin{proposition}[{\cite[Theorem 1]{maitre_pourquier}}, {\cite[Theorem 4.1]{hu_guo_katz_condition_bounds_p}}]
    \label{prop:matire_pourquier}
    Let $u \in \mathcal{Q}^p(\widehat{S})$ and let $\mathfrak{u}$ denote its coefficient vector with 
    respect to the basis of Definition~\ref{def:babuska_szabo_basis}. Then the following estimates hold:
    \begin{align}
      \left\|u\right\|_{L^2(\widehat{S})}^2 &\lesssim \left\|\mathfrak{u}\right\|_{\ell^2}^2 \lesssim p^8 \left\|u\right\|_{L^2(\widehat{S})}^2, & \quad 
      \left\|u\right\|_{H^1(\widehat{S})}^2 &\lesssim \left\|\mathfrak{u}\right\|_{\ell^2}^2 \lesssim p^4 \left\|u\right\|_{H^{1}(\widehat{S})}^2.
    \end{align}
  \end{proposition}
  
  \begin{remark}
    In \cite[Thm.~{1}]{maitre_pourquier}, the estimates were only shown for the inner degrees of freedom, i.e., 
    if $u|_{\partial\widehat{S}} = 0$.  This restriction is removed in \cite[Thm.~{4.1}]{hu_guo_katz_condition_bounds_p}.
\hbox{}\hfill\rule{0.8ex}{0.8ex}
  \end{remark}

  \begin{theorem}
\label{thm:L2H1-quads}
    Let $\mathcal{T}$ be a quasi-uniform, shape-regular affine mesh of quadrilaterals of size $h$, 
and let $F_K: \widehat{S} \to K$ be the affine element map for $K \in \mathcal{T}$. 
    Let $u \in \widetilde{\mathcal{Q}}^p(\mathcal{T}):=\left\{ u \in \widetilde{H}^{1/2}(\Gamma): u|_K \circ F_K \in \mathcal{Q}^p(\widehat{S}) \;\; \forall K \in \mathcal{T} \right\}$,
    and let $\mathfrak{u}$ denote its coefficient vector with respect to the basis of Definition~\ref{def:babuska_szabo_basis}.
    Then there exist constants $c_0,c_1,C_0,C_1 > 0$ that depend only on $\Gamma$ and the $\gamma$-shape regularity of $\mathcal{T}$ such that:
    \begin{align*}
      c_0 h^{-2} \left\|u\right\|_{L^2(\Gamma)}^2 &\leq \left\|\mathfrak{u}\right\|_{\ell^2}^2 \leq C_0 h^{-2} p^8 \left\|u\right\|_{L^2(\Gamma)}^2, &
      c_1 \left\|u\right\|_{H^1(\Gamma)}^2 &\leq \left\|\mathfrak{u}\right\|_{\ell^2}^2 \leq C_1 \left(p^4 + h^{-2}\right) \left\|u\right\|_{H^1(\Gamma)}^2.
    \end{align*}
    \begin{proof}
      The proof is completely analogous to that of Theorem~\ref{thm:condition_numbers_h1}. 
      The only additional ingredient to 
      Proposition~\ref{prop:matire_pourquier} is an operator $J_h: L^2 \to \widetilde{\mathcal{Q}}^{1}(\mathcal{T})$ 
      that is bounded with respect to the $L^2$ and the $H^1$-norm, reproduces homogeneous 
      Dirichlet boundary conditions for the case of open surfaces, and has the approximation property
            $\displaystyle  \left\|u-J_h u\right\|_{L^2(\Gamma)} \lesssim h \left|u\right|_{H^1(\Gamma)}.$
            Such an operator was proposed, e.g., in \cite{bernardi_girault_regularization_operator}. 
      The important estimates that need to be shown are
      (we again write $\widetilde{u}:=u-J_h u$):
      \begin{align*}
        \left\|u\right\|_{L^2(\Gamma)}^2 &\lesssim h^2 \left\|\mathfrak{u}\right\|_{\ell^2}^2, &
        \left\|u\right\|_{H^{1}(\Gamma)}^2&\lesssim \left\| \mathfrak{u}\right\|_{\ell^2}^2 \\
        \left\|\mathfrak{u}\right\|_{\ell^2}^2 &\lesssim \frac{p^8}{h^2} \left\|u\right\|_{L^2(\Gamma)}^2, &
        \left\|\mathfrak{u}\right\|_{\ell^2}^2 &\lesssim \left\|\mathfrak{\widetilde{u}}\right\|_{\ell^2}^2 + \left\|\mathfrak{J_h u}\right\|_{\ell^2}^2 
        \lesssim{p^4} \left\|u\right\|_{H^1(\Gamma)}^2 + h^{-2} \left\|u\right\|_{L^2(\Gamma)}^2. \qedhere
      \end{align*}      
    \end{proof}
  \end{theorem}

  \begin{remark} 
\label{rem:conditioning-babuska-szabo-quads}
   In the case of triangular meshes, Proposition~\ref{prop:characterization_poly_interp} allowed us to infer 
   $\widetilde{H}^{1/2}$-condition number estimates in Corollary~\ref{thm:condition_numbers_D} 
   by interpolating the discrete norm 
   equivalences in $L^2$ and $H^1$ of Theorem~\ref{thm:condition_numbers_h1}. For meshes consisting of 
   quadrilaterals, the result corresponding to 
   Proposition~\ref{prop:characterization_poly_interp} is currently not available in the literature. If we conjecture 
                   $\left[ \left( \widetilde{\mathcal{Q}}^{p}(\mathcal{T}), \left\|\cdot\right\|_{L^2(\Gamma)}\right), 
      \left( \widetilde{\mathcal{Q}}^{p}(\mathcal{T}), \left\|\cdot\right\|_{H^1(\Gamma)}\right) \right]_{1/2}     
    =\left( \widetilde{\mathcal{Q}}^{p}(\mathcal{T}), \left\|\cdot\right\|_{\widetilde{H}^{1/2}(\Gamma)}\right)$ 
     with equivalent norms, then Theorem~\ref{thm:L2H1-quads} implies the following estimates: 
    \begin{align*}
      h^{-1} \left\|u\right\|^2_{\widetilde{H}^{1/2}(\Gamma)} &\lesssim \left\|\mathfrak{u}\right\|^2_{\ell^2} 
      \lesssim \left(h^{-1} p^6 + h^{-2}\right) \left\|u\right\|_{\widetilde{H}^{1/2}(\Gamma)}^2, \\
      \kappa\left(\widetilde{D}_h^p\right) &\lesssim h^{-1} + p^6.
\tag*{\hbox{}\hfill\rule{0.8ex}{0.8ex}}
    \end{align*}
      \end{remark}

  \begin{remark}
    \cite{maitre_pourquier} also analyzes the influence of diagonal preconditioning and shows that the condition number is improved by a factor of two in the exponents of $p$. Although we did not make any 
 theoretical investigations in this direction for the $H^{1/2}$-case, our numerical experiments 
    in Examples~\ref{example_fichera_prec} and \ref{example_screen_prec} for triangular meshes show 
    that diagonal scaling improves the $p$-dependence of the condition number from 
    $\mathcal{O}(p^{5.5})$ to $\mathcal{O}(p^{2.5})$. 
\hbox{}\hfill\rule{0.8ex}{0.8ex}
  \end{remark}
  
  \begin{remark} 
The Babu\v{s}ka-Szab\'o basis of Definition~\ref{def:babuska_szabo_basis} is not the only one used on 
quadrilaterals or hexahedra. An important representative of other bases are the Lagrange interpolation polynomials
associated with the Gau{\ss}-Lobatto points.
        This basis has the better $\mathcal{O}(p^3)$ conditioning for the stiffness matrix
    and $\mathcal{O}(p^2)$ for the mass matrix (see \cite{melenk_gl},\cite[Sect.~6]{maitre_pourquier}).
    Using the same arguments as in the proof of Theorem~\ref{thm:L2H1-quads}, we get for the global $H^1$-problem that the condition number behaves 
    like $\mathcal{O}(p\,h^{-2} + p^3)$. If the conjecture of Remark~\ref{rem:conditioning-babuska-szabo-quads} is valid, 
then we obtain for this basis for the hypersingular integral operator the condition number 
estimate $\kappa(\widetilde D_h^p) \lesssim p^{5/2} + p^{-1/2} h^{-1}$.
(See the Appendix for details.)
\hbox{}\hfill\rule{0.8ex}{0.8ex}
  \end{remark}

\section{$hp$-preconditioning}
\label{sec:main_results}
\subsection{Abstract additive Schwarz methods}
\label{section_abstract_asm}
Additive Schwarz preconditioners are based on decomposing a vector space $\mathds{V}$ into smaller subspaces $\mathds{V}_i$, $i=0,\dots, J$, on
which a local problem is solved. We recall some of the basic definitions and important results.
Details can be found in \cite[chapter 2]{toselli_widlund}.

Let $a(\cdot,\cdot): \mathds{V} \times \mathds{V} \to \mathbb{R}$ be a symmetric, positive definite bilinear form on the finite dimensional vector space $\mathds{V}$.
For a given $f \in \mathds{V}'$ consider the problem of finding $u \in \mathds{V}$ such that
\begin{align*}
  a(u,v)&=f(v) \quad  \forall v \in \mathds{V}.
\end{align*} 
We will write $A$ for the corresponding Galerkin matrix.

Let $\mathds{V}_i \subset \mathds{V},\; i=0,\dots, J$, be finite dimensional vector spaces with corresponding
prolongation operators $R_i^T: \mathds{V}_i \to \mathds{V}$.
We will commit a slight abuse of notation and also denote the matrix representation of the operator by $R^T_i$, 
and $R_i$ is its transposed matrix.
We also assume that $\mathds{V}$ permits a (in general not direct) decomposition into
\begin{align*}
  \mathds{V}&= R_0^T \mathds{V}_0 + \sum_{i=1}^{J}{R_i^T \mathds{V}_i}.
\end{align*}
We assume that for each subspace $\mathds{V}_i$ a  symmetric and positive definite bilinear form
\begin{align*}
\widetilde{a}_i( \cdot, \cdot):\mathds{V}_i \times \mathds{V}_i \to \mathbb{R}, \quad i=0, \dots, J,
\end{align*}
is given.
We write $\widetilde{A}_i$ for the matrix representation of $\widetilde{a}_i(\cdot,\cdot)$.
Sometimes these bilinear forms are referred to as the ``local solvers''; 
in the simplest case of ``exact local solvers'' they are just restrictions of $a(\cdot,\cdot)$, i.e.,
$ \widetilde{a}_i(u_i,v_i):=a\left(R_i^T u_i, R_i^T v_i\right)$ for all $u_i,v_i \in \mathds{V}_i.$
Then, the corresponding additive Schwarz preconditioner is given by
\begin{align*}
  B^{-1}:=\sum_{i=0}^{J}{R_i^T \widetilde{A}^{-1}_i R_i }.
\end{align*}

The following proposition allows us to bound the condition number
of the preconditioned system $B^{-1}A$.
The first part is often referred to as the Lemma of Lions
(see \cite{zhang_multilevel_schwarz,lions_schwarz_alternating_1,nepomnyaschik_asm}).
\begin{proposition}
\label{thm_asm_estimate}
  \begin{enumerate}[(a)]
    \item Assume that there exists a constant $C_0 >0 $ such that every $u \in \mathds{V}$ admits a decomposition
       $ u = \sum_{i=0}^{J}{ R_i^T \, u_i}$ with $u_i \in \mathds{V}_i$
      such that
      \begin{align*}
       \sum_{i=0}^{J}{\widetilde{a}_i(u_i,u_i)}\leq C_0 \; a(u,u).
      \end{align*}
      Then, the minimal eigenvalue $\lambda_{min}(B^{-1}A)$ of $B^{-1} A$ satisfies 
      $  \lambda_{min}\left(B^{-1} A\right) \geq C_0^{-1}.$
    \item Assume that there exists $C_1 >0$ such that for every decomposition
      $u=\sum_{i=0}^{J}{R_i^T \,v_i}$ with $v_i \in \mathds{V}_i$ the following estimate holds:
      \begin{align*}
        a(u,u) \leq C_1 \sum_{i=0}^{J}{\widetilde{a}_i(v_i,v_i)}.
      \end{align*}
      Then, the maximal eigenvalue $\lambda_{max}(B^{-1} A)$ of $B^{-1} A$ satisfies 
        $\lambda_{max}\left(B^{-1} A\right) \leq C_1$.

      \item { These  two estimates together give an estimate for the condition number of the preconditioned linear system: }
\begin{align*}
  \kappa \left(B^{-1} A\right)&:=\frac{\lambda_{max}}{\lambda_{min}}\leq C_0 C_1. \tag*{\qed}
\end{align*}
\end{enumerate}

\end{proposition}
 \subsection{An $hp$-stable preconditioner}\label{sec:hpprecond}
In order to define an additive Schwarz preconditioner, we decompose the boundary element space $\mathds{V}:=\widetilde{S}^{p}(\mathcal{T})$ into
several overlapping subspaces.
We define $\mathds{V}_h^1:= \widetilde{S}^{1}(\mathcal{T})$ as the space of globally continuous and piecewise linear functions on $\mathcal{T}$ that 
vanish on $\partial \Gamma$ and denote the corresponding canonical embedding operator by $R_{h}^T: \mathds{V}_h^1 \to \mathds{V}$.
We also define for each vertex $\boldsymbol{z} \in {\mathcal{V}}$ the local space
\begin{align*}
  \mathds{V}_{\boldsymbol{z}}^p:=\{ u \in \widetilde{S}^{p}(\mathcal{T}) | \operatorname{supp}(u) 
\subset \overline{\omega_{\boldsymbol{z}}} \}
\end{align*}
and denote the canonical embedding operators by $R_{\boldsymbol{z}}^T: \mathds{V}_{\boldsymbol{z}}^p \to \mathds{V}$.
The space decomposition then reads
\begin{align}
\label{abstract_space_splitting}
\mathds{V} = \mathds{V}_h^1 + \sum_{\boldsymbol{z} \in {\mathcal{V}} } {\mathds{V}_{\boldsymbol{z}}^p}.
\end{align}
We will denote the restriction of the Galerkin matrix $\widetilde{D}_h^p$ to the 
subspaces $\mathds{V}_h^1$ and $\mathds{V}_{\boldsymbol{z}}^p$ as
$\widetilde{D}_h^1$ and $\widetilde{D}^p_{h,\boldsymbol{z}}$, respectively.

\begin{lemma}
\label{lemma:ppreconditioner_stable_splitting}
  There exist constants $c_1,c_2 > 0 $, which depend only on $\Gamma$ and the $\gamma$-shape regularity of $\mathcal{T}$,  such that the following holds:
\begin{enumerate}[(a)]
\item
  For every $u \in  \widetilde{S}^p(\mathcal{T})$ there exists a decomposition $u= u_1 + \sum_{\boldsymbol{z} \in {\mathcal{V}}}{u_{\boldsymbol{z}}}$ with
  $u_1 \in \mathds{V}_h^1$ and $u_{\boldsymbol{z}} \in \mathds{V}^p_{\boldsymbol{z}}$ and
  \begin{align*}
     \left\|u_1\right\|_{\widetilde{H}^{1/2}(\Gamma)}^2 + \sum_{\boldsymbol{z} \in {\mathcal{V}}}{\left\|u_{\boldsymbol{z}}\right\|_{\widetilde{H}^{1/2}(\Gamma)}^2} &\leq c_1 \; \left\|u\right\|_{\widetilde{H}^{1/2}(\Gamma)}^2.
  \end{align*}
\item Any decomposition $u=v_1 + \sum_{\boldsymbol{z} \in {\mathcal{V}}}{v_{\boldsymbol{z}}}$ with $v_1 \in \mathds{V}^1_h$ and $v_{\boldsymbol{z}} \in \mathds{V}^p_{\boldsymbol{z}}$ satisfies
  \begin{align*}
    \left\|u\right\|_{\widetilde{H}^{1/2}(\Gamma)}^2 &\leq c_2 \left(\left\|v_1\right\|_{\widetilde{H}^{1/2}(\Gamma)}^2 + \sum_{\boldsymbol{z} \in {\mathcal{V}}}{ \left\|v_{\boldsymbol{z}}\right\|_{\widetilde{H}^{1/2}(\Gamma)}^2} \right).
  \end{align*}
\end{enumerate}
\begin{proof}
  The first estimate is the assertion of Proposition~\ref{prop:stable_space_decomposition}, (\ref{item:prop:stable_space_decomposition-iv}).  
  The second estimate can be shown by a so-called coloring argument, along the same lines as in \cite[Lemma 2]{heuer_asm_indef_hyp}. It is based on 
  the following estimate (see \cite[Lemma 4.1.49]{book_sauter_schwab} or \cite[Lemma 3.2]{petersdorff_rwp_elasticity}):
  Let $w_j$, $j=1,\dots, n$ be functions in $\widetilde{H}^{s}(\Gamma)$ for $s \geq 0$ with pairwise disjoint support. Then it holds
  \begin{align}
\label{eq:SS}
    \left\|\sum_{i=1}^{n}{w_i}\right\|_{\widetilde{H}^{s}(\Gamma)}^2 &\leq  C \sum_{i=1}^{n}{\left\|w_i\right\|^2_{\widetilde{H}^{s}(\Gamma)}},
  \end{align}
  where $C>0$ depends only on $\Gamma$.
  By $\gamma$-shape regularity, the number of elements in any vertex patch, and therefore also the number of vertices in a patch, is uniformly bounded by some constant
  $N_{c}$ which depends solely on $\gamma$.
  Thus, we can divide the vertices into sets $J_1,\dots,J_{N_{c}}$ such that
  $\bigcup_{i=1}^{N_c}{ J_i} = {\mathcal{V}}$ and $ \left|\omega_{\boldsymbol{z}} \cap \omega_{\boldsymbol{z'}} \right| = 0$ for all $\boldsymbol{z},\boldsymbol{z'}$ in the same index set $J_i$.
  Repeated application of the triangle inequality together with \eqref{eq:SS} then gives:
\begin{align*}
  \left\|u\right\|_{\widetilde{H}^{1/2}(\Gamma)}^2 \leq 2\left\|v_1\right\|_{\widetilde{H}^{1/2}(\Gamma)}^2 + 2 \left\|\sum_{\boldsymbol{z} \in {\mathcal{V}}}{ v_{\boldsymbol{z}}}\right\|_{\widetilde{H}^{1/2}(\Gamma)}^2 
  &\leq 2 \left\|v_1\right\|_{\widetilde{H}^{1/2}(\Gamma)}^2 + 2 N_c \sum_{i=1}^{N_c}{ \left\|\sum_{\boldsymbol{z} \in J_i}{v_{\boldsymbol{z}}}\right\|_{\widetilde{H}^{1/2}(\Gamma)}^2 }  \\
  &\leq 2 \left\|v_1\right\|_{\widetilde{H}^{1/2}(\Gamma)}^2 + 2\, N_c \,C  \sum_{\boldsymbol{z} \in {\mathcal{V}}}{ \left\|v_{\boldsymbol{z}}\right\|_{\widetilde{H}^{1/2}(\Gamma)}^2 }.
\qedhere
\end{align*}
\end{proof}
\end{lemma}
The previous lemma only made statements about the $\widetilde{H}^{1/2}(\Gamma)$-norm. 
\begin{theorem}
\label{thm:ppreconditioner}
  Let $\mathcal{T}$ be a $\gamma$-shape regular triangulation of $\Gamma$. Then there is a constant $C > 0$ that depends
solely on $\Gamma$ and the $\gamma$-shape regularity of ${\mathcal{T}}$ such that the following is true: 
The preconditioner 
    \begin{align*}
    B^{-1}:= R^T_h \left(\widetilde{D}_h^1 \right)^{-1} R_h  \; + \; \sum_{\boldsymbol{z} \in {\mathcal{V}}}{ R^T_{\boldsymbol{z}} \left( \widetilde{D}_{h,\boldsymbol{z}}^p\right)^{-1}} R_{\boldsymbol{z}},
  \end{align*}
which is implied by the space decomposition~\eqref{abstract_space_splitting}, leads to the spectral condition 
number estimate 
  \begin{align*}
    \kappa(B^{-1} \widetilde{D}_h^p) &\leq C.
  \end{align*}
  \begin{proof}
    The bilinear form $\left<\widetilde{D} \cdot ,\cdot \right>_{\Gamma}$ is equivalent to $\left\|{\cdot}\right\|^2_{\widetilde{H}^{1/2}(\Gamma)}$. 
    Hence, the combination of Lemma~\ref{lemma:ppreconditioner_stable_splitting} and 
    Proposition~\ref{thm_asm_estimate}  give the boundedness of the condition number. 
  \end{proof}
\end{theorem}

 \subsection{Multilevel preconditioning on adaptive meshes}\label{sec:lmld}
\label{section_adaptive_meshes}

The preconditioner of Theorem~\ref{thm:ppreconditioner} relies on the space decomposition
(\ref{abstract_space_splitting}). In this section, we discuss how the space $\widetilde{S}^1(\mathcal{T})$ of
piecewise linear function can be further decomposed in a multilevel fashion. Our setting will be one
where $\mathcal{T}$ is the finest mesh of a sequence $(\mathcal{T}_\ell)_{\ell=0}^L$ of nested meshes that 
are generated by newest vertex bisection (NVB); see Figure~\ref{fig:nvb:bisec} for a description. We 
point the reader to \cite{stevenson,kpp} for a detailed discussion of NVB. A key feature of NVB is that it creates
sequences of meshes that are uniformly shape regular. We mention in passing that further properties of NVB 
were instrumental in proving optimality of $h$-adaptive algorithms in both FEM~\cite{ckns,ffp_adaptive_fem} 
and BEM~\cite{partOne,partTwo,fkmp13,gantumur}. 
Before discussing the details of the multilevel space decomposition, we stress that 
the preconditioner described in Section~\ref{sec:hpprecond} is independent of the chosen
refinement strategy (such as NVB) as long as it satisfies the assumptions in Section~\ref{sec:hpprecond}, whereas 
the condition number estimates for the local multilevel preconditioner discussed in the present 
Section~\ref{sec:lmld} depend on the fact that the underlying refinement strategy is based on NVB.  
\bigskip

\begin{figure}[t]
 \centering

 \includegraphics[width=0.2\textwidth]{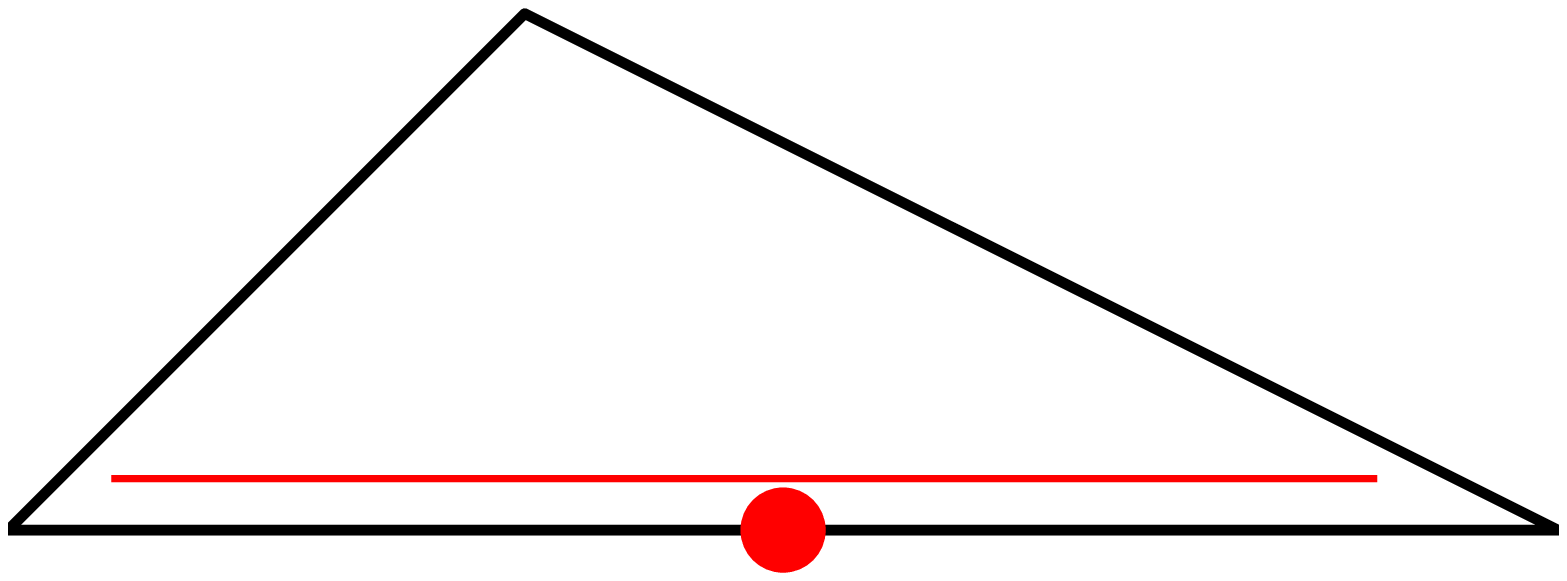} \quad\nolinebreak
 \includegraphics[width=0.2\textwidth]{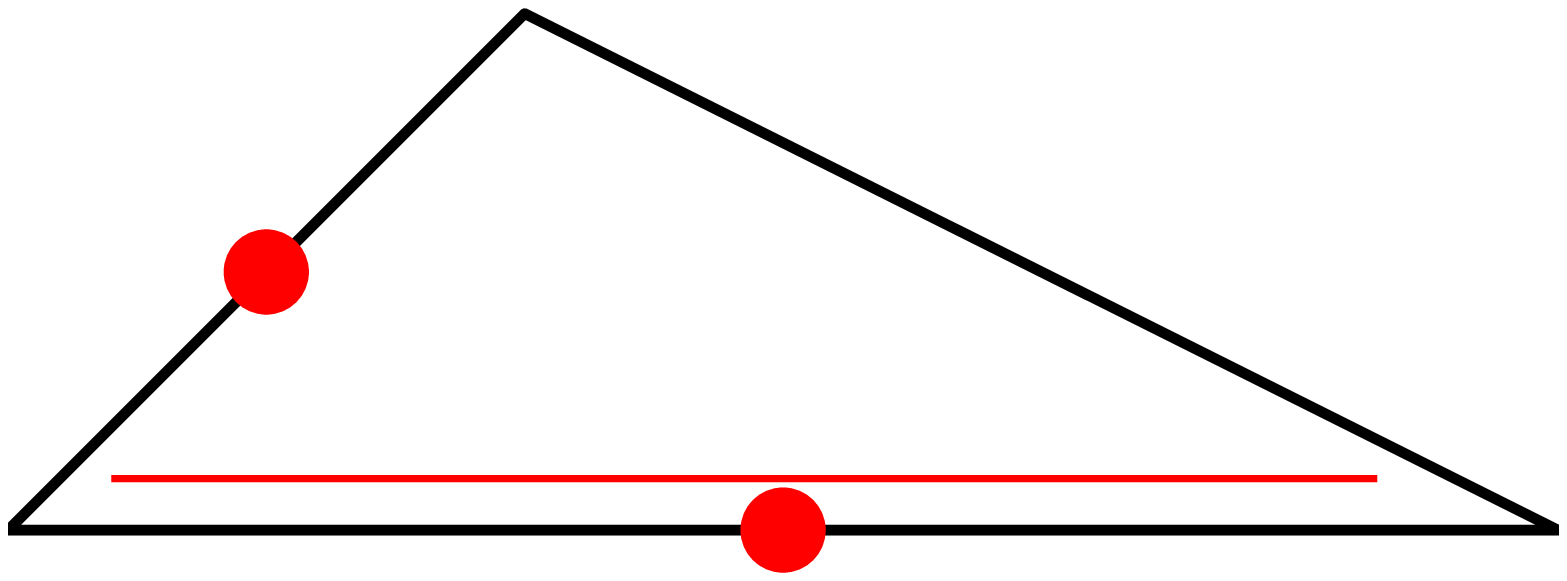} \quad\nolinebreak
 \includegraphics[width=0.2\textwidth]{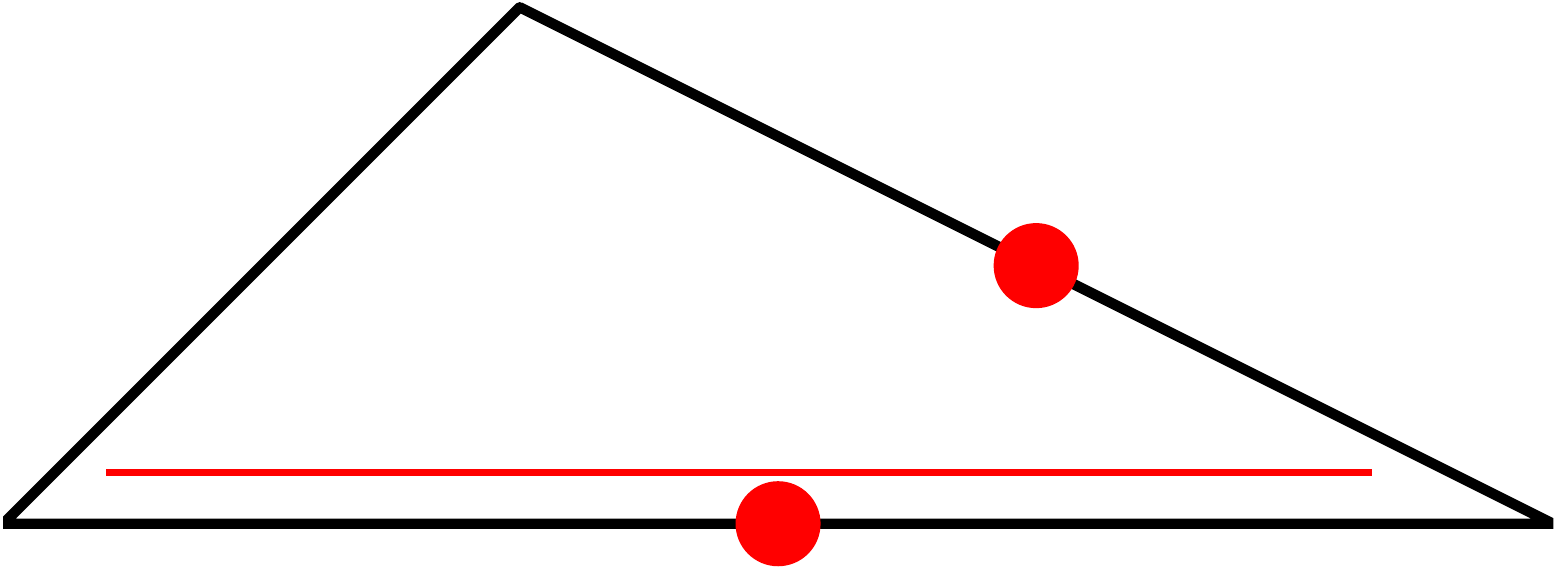} \quad\nolinebreak
 \includegraphics[width=0.2\textwidth]{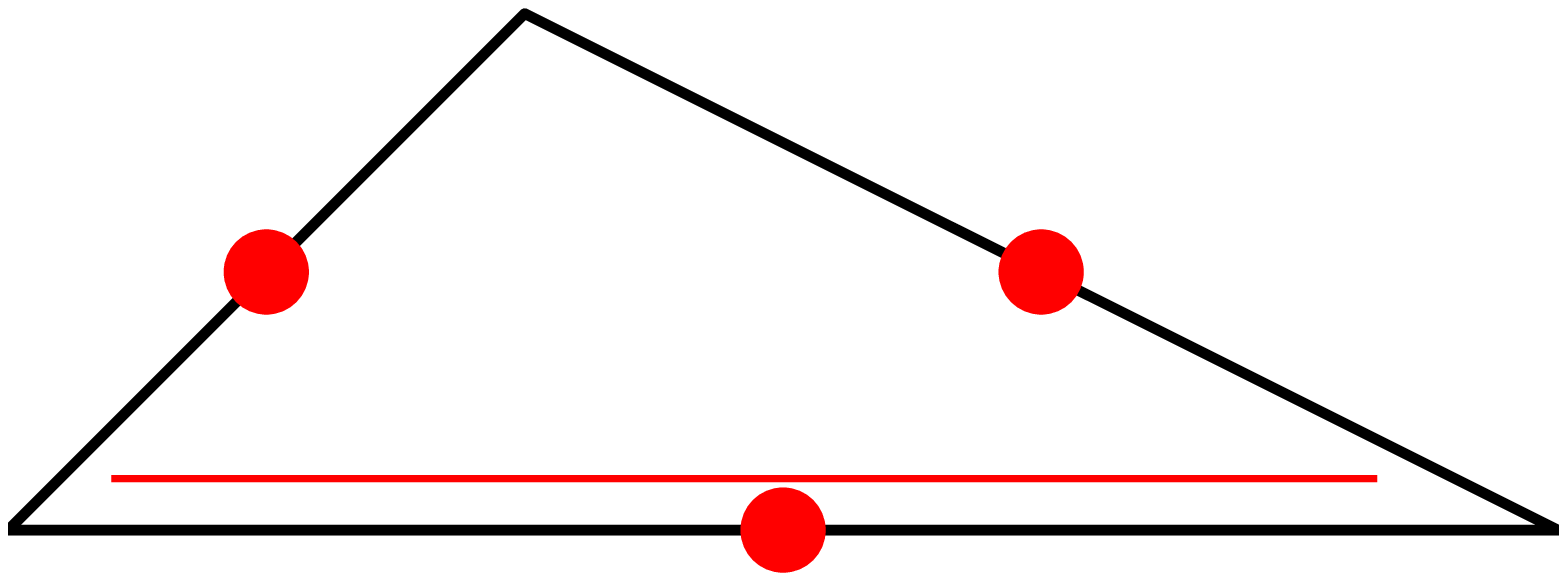} \\
 \includegraphics[width=0.2\textwidth]{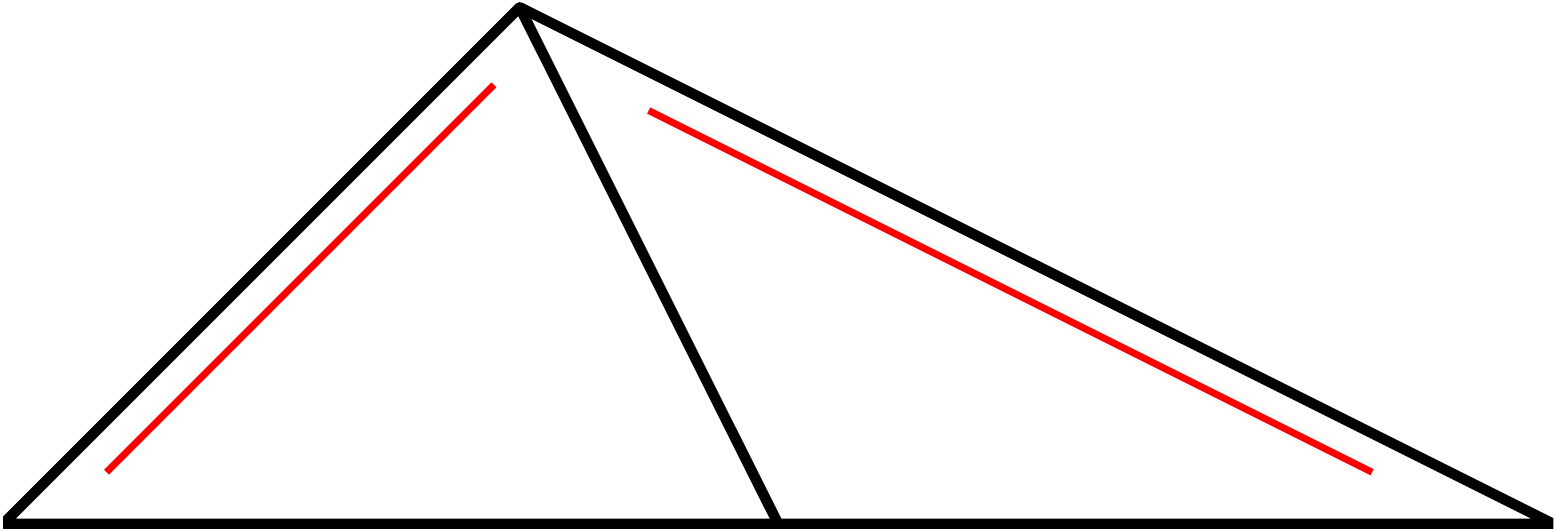} \quad
 \includegraphics[width=0.2\textwidth]{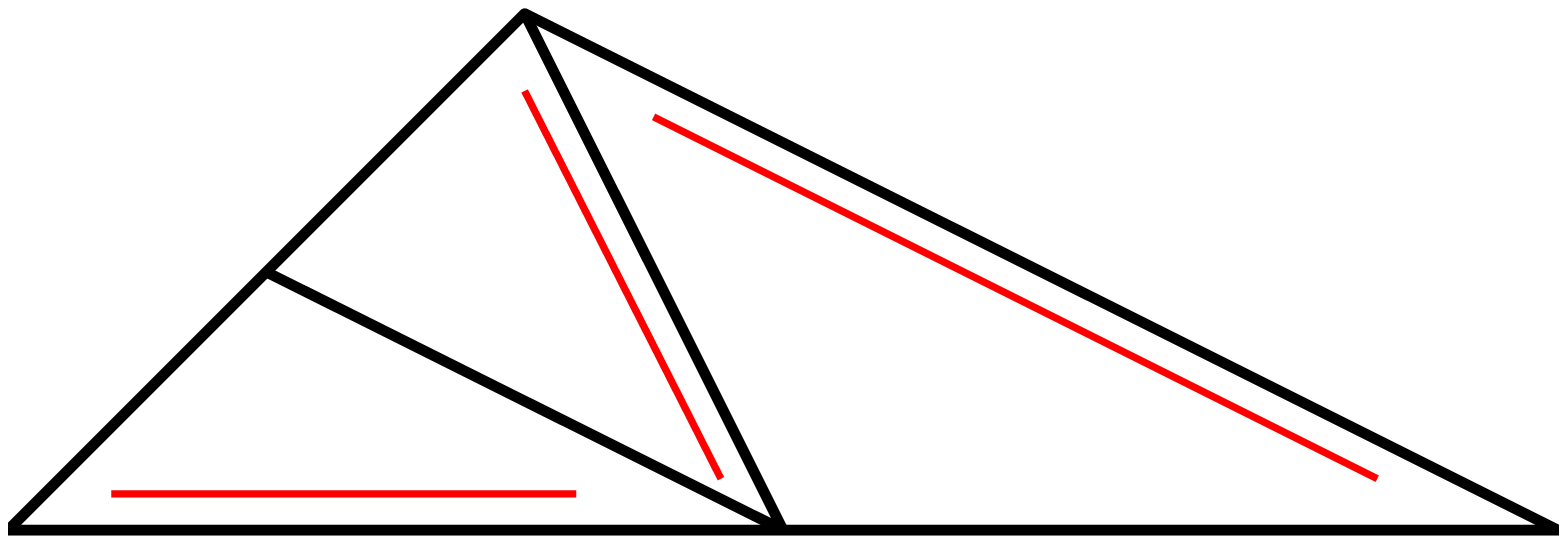}\quad
 \includegraphics[width=0.2\textwidth]{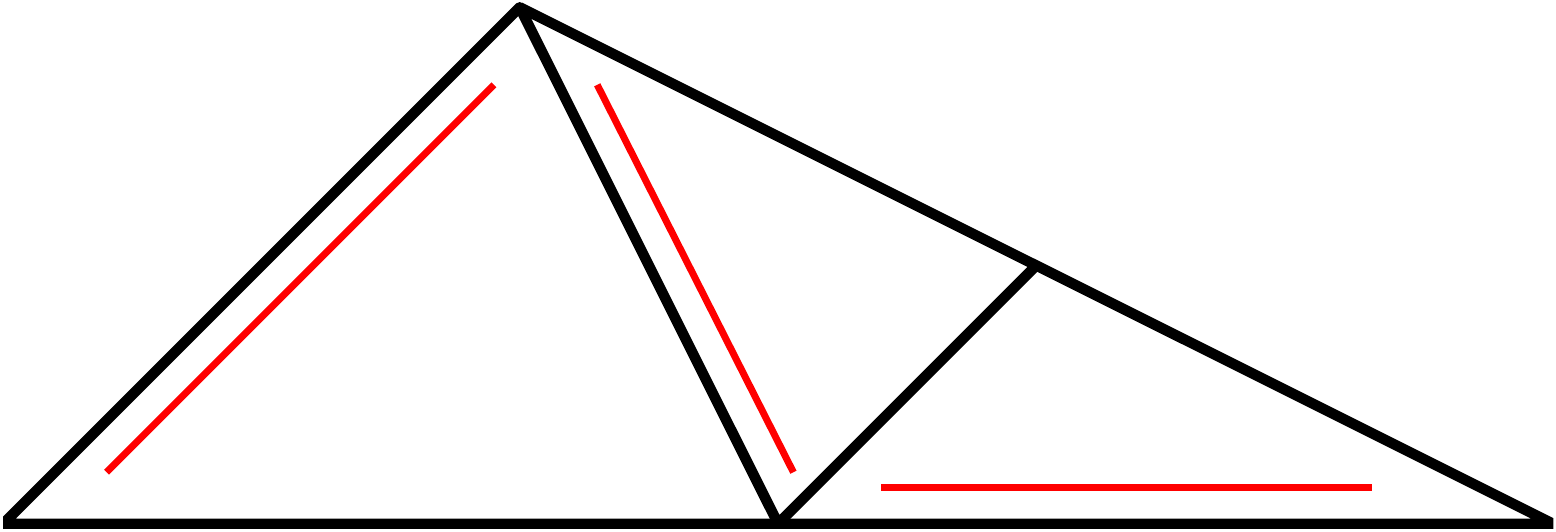}\quad
 \includegraphics[width=0.2\textwidth]{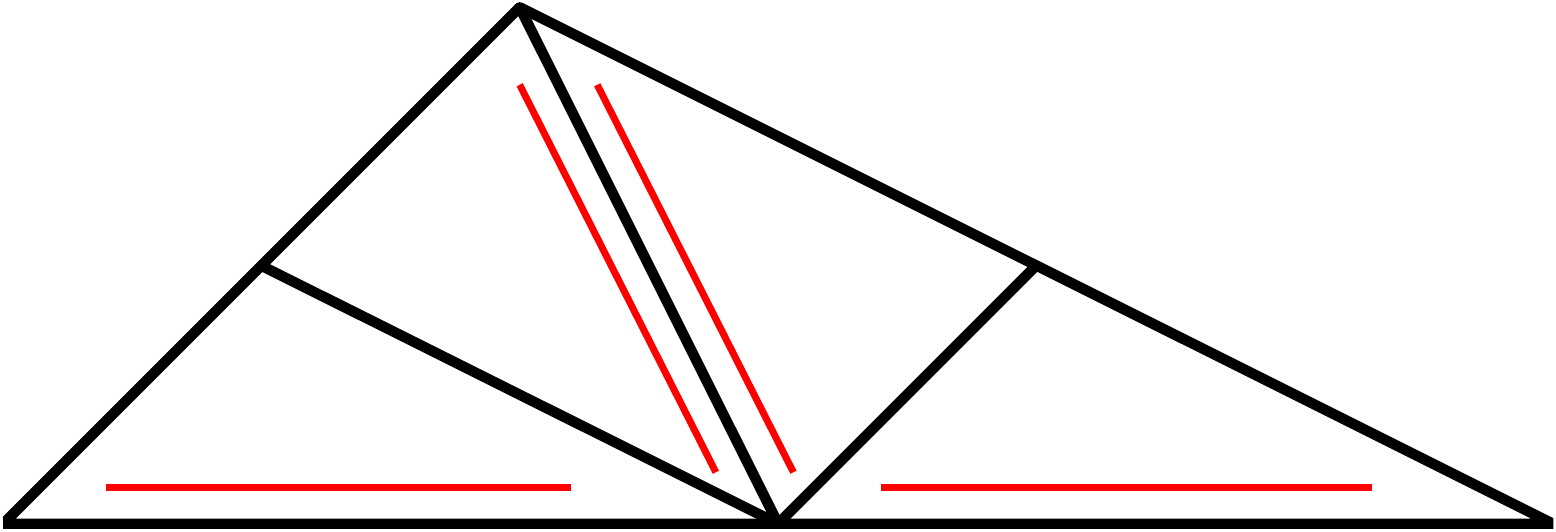}
 \caption{For each element $K\in\mathcal{T}_\ell$ there exists exactly one \emph{reference edge} indicated by the red line
 (upper left plot). The element $K$ is refined by bisecting its reference edge. This leads to a new node (red dot) and
 two son elements. The reference edges of the son elements are opposite to the \emph{newest vertex} (lower left plot).
 Hanging nodes are avoided as follows: Assume that some of the edges of the triangle, but at least the reference edge,
 are marked for refinement (upper plots). The triangle will be split into two, three or four son elements by iterative
 application of the \emph{newest vertex bisection (NVB)}.
 }
 \label{fig:nvb:bisec}

\end{figure}

Adaptive algorithms create sequences of meshes $(\mathcal{T}_\ell)_{\ell=0}^L$. Typically, the procedure starts 
with an \emph{initial} triangulation $\mathcal{T}_0$ and the further members of the sequences are created 
inductively. That is, mesh $\mathcal{T}_\ell$ is obtained from $\mathcal{T}_{\ell-1}$ by
refining some elements of $\mathcal{T}_{\ell-1}$. In an adaptive environment, these elements are determined 
by a marking criterion (``marked elements'') and a mesh closure condition. 
Usually, the following assumptions on the mesh refinement are made:
\begin{itemize}
  \item $\mathcal{T}_\ell$ is regular for all $\ell\in\mathbb{N}_0$, i.e., there exist no hanging nodes;
  \item The meshes $\mathcal{T}_0,\mathcal{T}_1,\dots$ are uniformly $\gamma$-shape-regular, i.e., 
with $|K|$ denoting the surface area of an element $K\in\mathcal{T}_\ell$ and $\mathrm{diam}(K)$ 
the Euclidean diameter, we have 
    \begin{align}
      \sup_{\ell\in\mathbb{N}_0} \max_{K\in\mathcal{T}_\ell} \frac{\mathrm{diam}(K)^2}{|K|} \leq \gamma.
      \label{eq:refinement-shape-regularity}
    \end{align}
\end{itemize}

We consider a sequence of triangulations $\mathcal{T}_0,\dots,\mathcal{T}_L$, which is created by iteratively applying NVB. The corresponding sets of vertices are denoted ${\mathcal{V}}_0,\dots,{\mathcal{V}}_L$.
For a vertex ${\boldsymbol{z}}\in{\mathcal{V}}_\ell$, the associated patch is denoted by $\omega_{\ell,{\boldsymbol{z}}}$.

In the construction of the $p$-preconditioner in Section~\ref{sec:hpprecond} we only considered a single mesh $\mathcal{T}$.
For the remainder of the paper, the $p$ part will always be constructed with respect to the finest mesh $\mathcal{T}_L$.
For a simpler presentation we set $\mathcal{T} := \mathcal{T}_L$ and ${\mathcal{V}} := {\mathcal{V}}_L$.

\subsubsection{A refined splitting for adaptive meshes}

\begin{figure}
  \centering
  \includegraphics[width=0.3\textwidth]{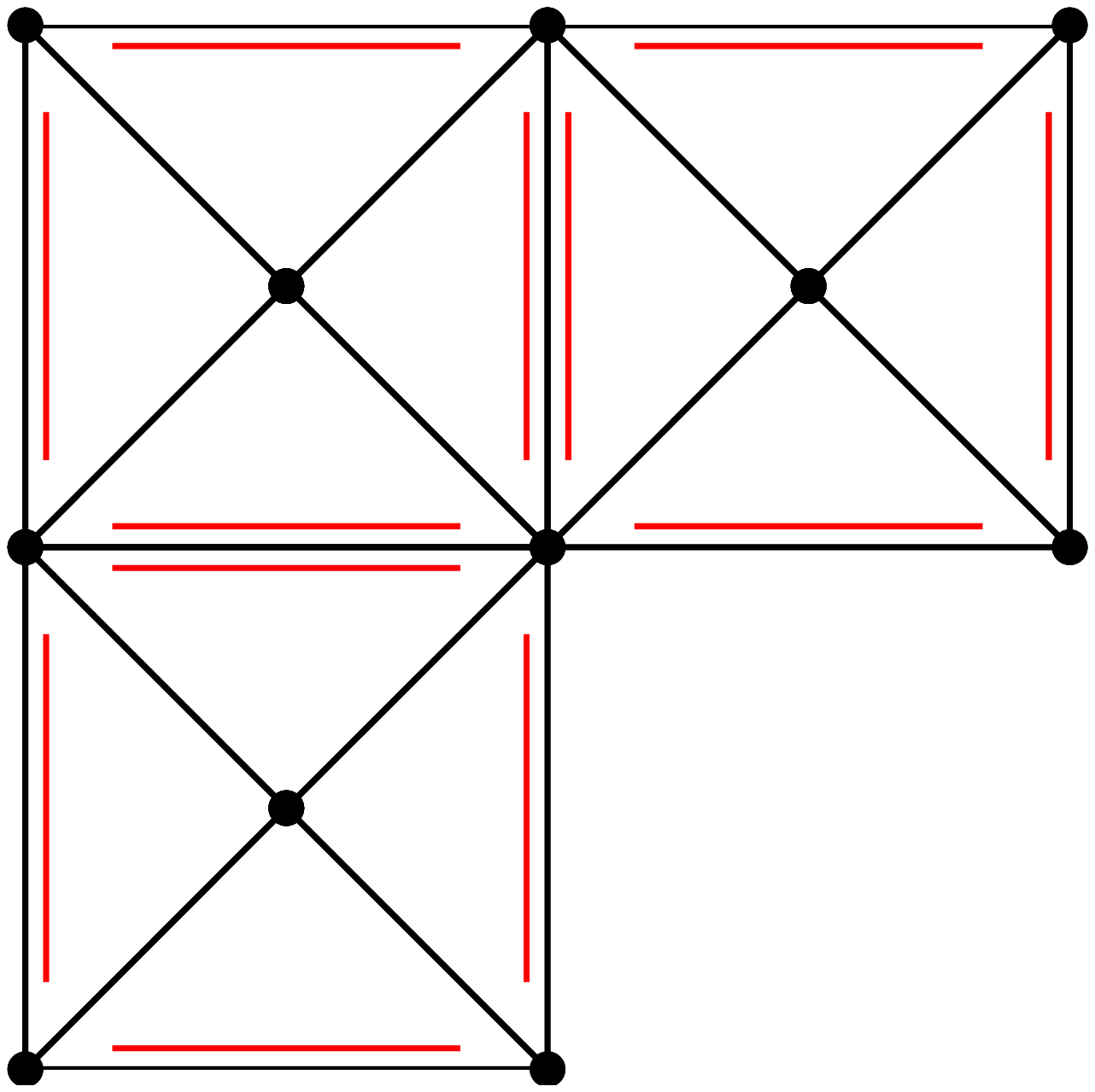}
  \includegraphics[width=0.3\textwidth]{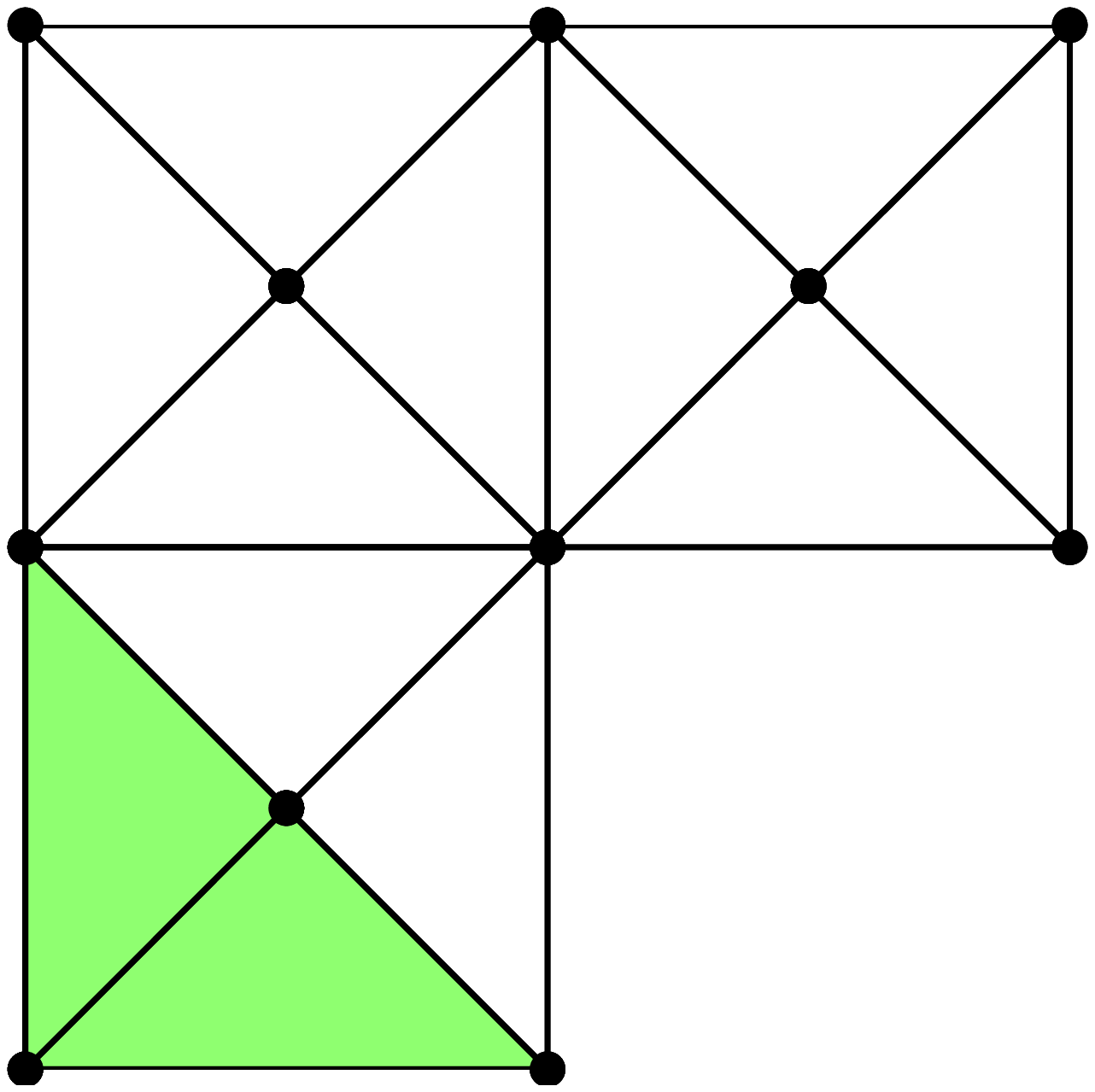}
  \includegraphics[width=0.3\textwidth]{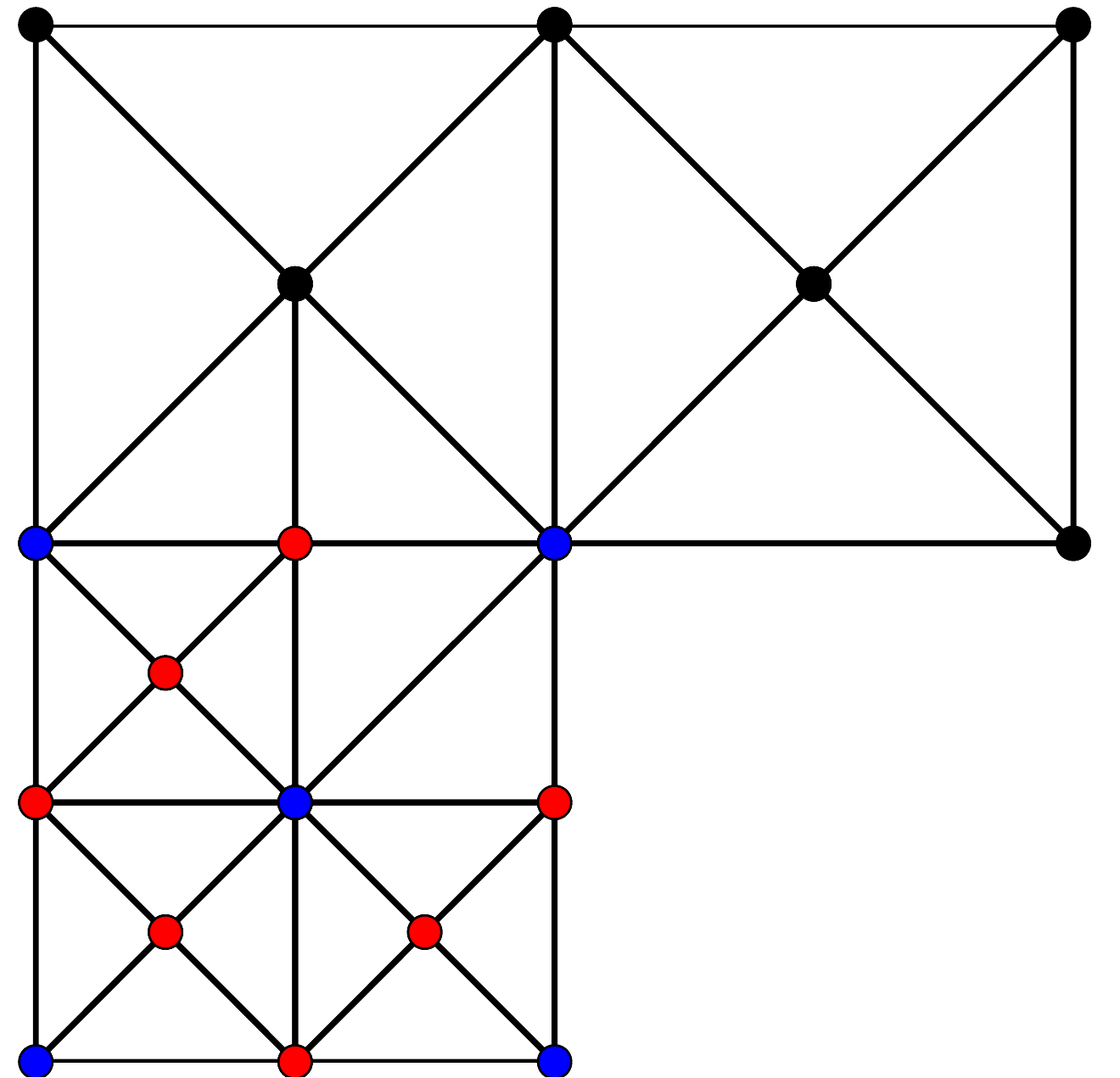}
  \caption{Visualization of the definition~\eqref{eq:def:nodesTilde} of the local subsets $\widetilde{\mathcal{V}}_\ell$: Starting
  with a triangulation $\mathcal{T}_{\ell-1}$ (left), we mark two elements, indicated by green triangles (middle), for
  refinement. Using iterated NVB refinement we obtain the mesh $\mathcal{T}_\ell$ (right).
  The set $\widetilde{\mathcal{V}}_\ell$ consists of the new vertices (red) and neighboring vertices, where the corresponding
  vertex patches have changed (blue).}
  \label{fig:nodesTilde}
\end{figure}

The space decomposition from~\eqref{abstract_space_splitting} involves the global lowest-order space $\mathds{V}_h^1 =
\widetilde{S}^1(\mathcal{T}_L)$.
Therefore, the computation of the corresponding additive Schwarz operator needs the inversion of a global problem, 
which is, in practice, very costly, and often even infeasible.
To overcome this disadvantage, we consider a refined splitting of the space $\mathds{V}_h^1$ that relies on the hierarchy of the
adaptively refined meshes $\mathcal{T}_0,\dots,\mathcal{T}_L$.
The corresponding local multilevel preconditioner was introduced and analyzed in~\cite{ffps,dissTF}.
See also~\cite{hiptwuzheng2012,wuchen06,xch10,xcn09} for local multilevel preconditioners for (adaptive) FEM
, \cite{tran_stephan_asm_h_96,hiptmair_mao_BIT_2012} for (uniform) BEM, and \cite{amcl03,maischak_multilevel_asm,tran_stephan_mund_hierarchical_prec} for
(restricted) approaches for adaptive BEM.

Set $\widetilde{\mathcal{V}}_0  := {\mathcal{V}}_0$ and define the local subsets
\begin{align}\label{eq:def:nodesTilde}
  \widetilde{\mathcal{V}}_\ell := {\mathcal{V}}_\ell \backslash {\mathcal{V}}_{\ell-1} \cup \{ {\boldsymbol{z}}\in{\mathcal{V}}_{\ell-1} \,:\,
  \omega_{\ell,{\boldsymbol{z}}} \subsetneq \omega_{\ell-1,{\boldsymbol{z}}} \} \quad\text{for } \ell\geq 1
\end{align}
of newly created vertices plus some of their neighbors, see Figure~\ref{fig:nodesTilde} for a visualization.
Based on these sets, we consider the space decomposition
\begin{align}\label{eq:splitting:refined}
  \mathds{V}_h^1 = \sum_{\ell=0}^L \sum_{{\boldsymbol{z}}\in\widetilde{\mathcal{V}}_\ell} \mathds{V}^1_{\ell,{\boldsymbol{z}}} \quad\text{with}\quad \mathds{V}^1_{\ell,{\boldsymbol{z}}} :=
  \mathrm{span}\{\varphi_{\ell,{\boldsymbol{z}}}\},
\end{align}
where $\varphi_{\ell,{\boldsymbol{z}}}\in \widetilde{S}^1(\mathcal{T}_\ell)$ is the nodal hat function with $\varphi_{\ell,{\boldsymbol{z}}}({\boldsymbol{z}}) = 1$ and $\varphi_{\ell,{\boldsymbol{z}}}({\boldsymbol{z}}') =
0$ for all ${\boldsymbol{z}}'\in
{\mathcal{V}}_\ell \backslash \{{\boldsymbol{z}}\}$.
The basic idea of this splitting is that we do a diagonal scaling only in the regions where the meshes have been refined.
We will use local exact solvers, i.e.,
\begin{align*}
  \widetilde a_{\ell,{\boldsymbol{z}}}(u_{\ell,{\boldsymbol{z}}},v_{\ell,{\boldsymbol{z}}}) := \left<\widetilde{D}  (R_{\ell,{\boldsymbol{z}}})^T u_{\ell,{\boldsymbol{z}}} , (R_{\ell,{\boldsymbol{z}}})^T
  v_{\ell,{\boldsymbol{z}}}  \right>_{\Gamma}
  \quad\text{for all } u_{\ell,{\boldsymbol{z}}},v_{\ell,{\boldsymbol{z}}} \in \mathds{V}^1_{\ell,{\boldsymbol{z}}},
\end{align*}
where $(R_{\ell,{\boldsymbol{z}}})^T : \mathds{V}^1_{\ell,{\boldsymbol{z}}} \to \mathds{V}_h^1$ denotes the canonical embedding operator.
Let $\widetilde{D}_h^1$ denote the Galerkin matrix of $\widetilde{D}$ with respect to the basis
$(\varphi_{L,{\boldsymbol{z}}})_{{\boldsymbol{z}}\in{\mathcal{V}}_L}$ of $\mathds{V}_h^1$ and define $\widetilde{D}_{\ell,{\boldsymbol{z}}}^1 := \widetilde
a_{\ell,{\boldsymbol{z}}} (\varphi_{\ell,{\boldsymbol{z}}},\varphi_{\ell,{\boldsymbol{z}}})$.
Then, the \emph{local multilevel diagonal (LMLD)} preconditioner associated to the
splitting~\eqref{eq:splitting:refined} reads
\begin{align}
  (B_h^1)^{-1} := \sum_{\ell=0}^L \sum_{{\boldsymbol{z}}\in\widetilde{\mathcal{V}}_\ell} (R_{\ell,{\boldsymbol{z}}})^T \left(\widetilde{D}_{\ell,{\boldsymbol{z}}}^1\right)^{-1}
  R_{\ell,{\boldsymbol{z}}}.
\end{align}
We stress that this preconditioner corresponds to a diagonal scaling with respect to the local subset of vertices
$\widetilde{\mathcal{V}}_\ell$ on each level $\ell = 0,\dots,L$.
Further details and the proof of the following result are found in~\cite{ffps,dissTF}.
\begin{proposition}\label{thm:hpreconditioner}
  The splitting~\eqref{eq:splitting:refined} together with $\widetilde a_{\ell,{\boldsymbol{z}}}(\cdot,\cdot)$ 
and the operators $R_{\ell,\boldsymbol{z}}^T$ satisfies the
  requirements of Proposition~\ref{thm_asm_estimate} with constants depending only on $\Gamma$ 
  and the initial triangulation $\mathcal{T}_0$.
  For the additive Schwarz operator $P_h^1 := (B_h^1)^{-1} \widetilde{D}_h^1$, there holds in particular 
  \begin{align}
     c \left<\widetilde{D} u_h ,u_h \right>_{\Gamma} \leq \left<\widetilde{D} P_h^1 u_h , u_h \right>_{\Gamma} \leq C \left<\widetilde{D} u_h ,u_h \right>_{\Gamma} \quad\text{for all } u_h\in \mathds{V}_h^1.
  \end{align}
  The constants $c$, $C>0$ depend only on $\Gamma$, the initial triangulation $\mathcal{T}_0$, and the use of NVB for refinement, 
  i.e., $\mathcal{T}_{\ell+1} = \operatorname{refine}(\mathcal{T}_{\ell}, \mathcal{M}_{\ell})$ 
with arbitrary set $\mathcal{M}_{\ell} \subseteq \mathcal{T}_{\ell}$ of marked elements.
  \qed
\end{proposition}

We replace the space $\mathds{V}_h^1$ in~\eqref{abstract_space_splitting} by the refined splitting~\eqref{eq:splitting:refined}
and end up with the space decomposition
\begin{align}\label{eq:hpsplitting}
  \mathds{V} = \sum_{\ell=0}^L \sum_{{\boldsymbol{z}}\in\widetilde{{\mathcal{V}}}_\ell} \mathds{V}^1_{\ell,{\boldsymbol{z}}} + \sum_{{\boldsymbol{z}}\in{\mathcal{V}}_L} \mathds{V}_{L,\boldsymbol{z}}^p.
\end{align}
The following Lemma \ref{lemma_replace_global_space} shows that the preconditioner resulting from the decomposition~\eqref{eq:hpsplitting} is
$hp$-stable. 
The result formalizes the observation that the combination of stable subspace decompositions leads again to a stable
subspace decomposition.
It is a simple consequence of the well-known theory for additive Schwarz methods; see
Section~\ref{section_abstract_asm}. Therefore, details are left to the reader.

\begin{lemma}
\label{lemma_replace_global_space}
  Let $\mathds{V}$ be a finite dimensional vector space, and let $\mathds{V}_j$, $R_{\mathds{V},j}^T$, and $\widetilde{a}_{\mathds{V},j}(\cdot,\cdot)$ for $j=0, \dots, J$
  be a decomposition of $\mathds{V}$ in
  the sense of Section \ref{section_abstract_asm} that satisfies the assumptions of Proposition~\ref{thm_asm_estimate} with constants $C_{0,\mathds{V}}$ and $C_{1,\mathds{V}}$.
  Consider an additional decomposition $\mathds{W}_{\ell}$, $R_{\mathds{W},{\ell}}^T $
    and $\widetilde{a}_{\mathds{W},\ell}(\cdot,\cdot)$ with $\ell=0, \dots, L$
  of $\mathds{V}_0$
  that also satisfies the requirements of Proposition~\ref{thm_asm_estimate} for the bilinear form $\widetilde{a}_{\mathds{V},0}(\cdot,\cdot)$
  with constants $C_{0,\mathds{W}}$ and $C_{1,\mathds{W}}$.
  Define a new additive Schwarz preconditioner as:
  \begin{align*}
    \widetilde{B}^{-1}:= R_{\mathds{V},0}^T \left(\sum_{\ell=0}^{L}{R_{\mathds{W},\ell}^T \widetilde{A}_{\mathds{W},\ell}^{-1} R_{\mathds{W},\ell} } \; \right)R_{\mathds{V},0} \; 
    + \; \sum_{j=1}^{J}{R_{\mathds{V},j}^T \widetilde{A}_{\mathds{V},j}^{-1} R_{\mathds{V},j}}.
  \end{align*}
  This new preconditioner satisfies the assumptions of Proposition~\ref{thm_asm_estimate} with $C_0=\max\left(1,C_{0,\mathds{W}}\right) C_{0,\mathds{V}}$ and
  $C_1=\max\left(1,C_{1,\mathds{W}}\right) C_{1,\mathds{V}}$.
  \qed
\end{lemma}

\begin{theorem}
  \label{thm:p_precond_with_multilevel}
  Assume that $\mathcal{T}$ is generated from a regular and shape-regular initial triangulation $\mathcal{T}_0$ by successive application of NVB.
  Based on the space decomposition~\eqref{eq:hpsplitting} define the preconditioner
  \begin{align*}
    B_2^{-1}:= R_h^T (B_h^1)^{-1} R_h  \; + \; \sum_{\boldsymbol{z} \in {\mathcal{V}}_L}{R^T_{\boldsymbol{z}} \left(\widetilde{D}_{h,\boldsymbol{z}}^p\right)^{-1}} R_{\boldsymbol{z}}.
  \end{align*}
  Then, for constants $c$, $C>0$ that depend only on $\Gamma$, $\mathcal{T}_0$, and the use of NVB refinement, the 
  extremal eigenvalues of $B_2^{-1} \widetilde D_h^p$ satisfy 
  \begin{align*}
    c\leq \lambda_\mathrm{min}(B_2^{-1} \widetilde D_h^p) 
    \leq \lambda_\mathrm{max} (B_2^{-1} \widetilde D_h^p) \leq C.
  \end{align*}
  In particular, the condition number $\kappa(B_2^{-1} \widetilde D_h^p)$ is bounded independently of $h$ and
  $p$.
\begin{proof}
  The proof follows from Lemma~\ref{lemma:ppreconditioner_stable_splitting}, Proposition~\ref{thm:hpreconditioner}, and
  Lemma~\ref{lemma_replace_global_space}.
\end{proof}
\end{theorem}
 
\subsection{Spectrally equivalent local solvers}
\label{section_inexact_locals}
For each vertex patch, we need to store the dense matrix $\left(  \widetilde{D}^p_{h,\boldsymbol{z}}\right)^{-1}$.
For higher polynomial orders, storing these blocks is a significant part of the memory consumption of the 
preconditioner. To reduce these costs, we can make use of the fact that the abstract additive Schwarz theory 
allows us to replace the local bilinear forms $a(R^T_i u_i,R^T_i v_i)$ with spectrally equivalent forms, 
as long as they satisfy the conditions stated in Proposition~\ref{thm_asm_estimate}.
This is for example the case, if the decomposition is stable for the exact local solvers and if there exist constants 
$c_1$, $c_2 > 0$ such that
\begin{align*}
  c_1 \, \widetilde{a}_i(u_i,u_i) &\leq a(R^T_i u_i,R^T_i u_i) \leq c_2 \widetilde{a}_i(u_i,u_i) \quad \forall u_i \in \mathds{V}_i.
\end{align*}
The new preconditioner will be based on a finite number of reference patches, for which the Galerkin matrix has to be inverted.

First we prove the simple fact that we can drop the stabilization term from (\ref{eq:weakform}) when assembling the local bilinear forms:
\begin{lemma}
  There exists a constant $c_1>0$ that depends only on $\Gamma$ and the $\gamma$-shape regularity of $\mathcal{T}$
  such that for any vertex patch $\omega_{\boldsymbol{z}}$ the following estimates
  hold:
  \begin{align*}
    \left<Du,u \right>_{\Gamma} &\leq \left<\widetilde{D} u ,u \right>_{\Gamma} \leq c_1 \left<Du,u \right>_{\Gamma} \qquad \forall u \in \mathds{V}^p_{\boldsymbol{z}}.
  \end{align*}
  \begin{proof}
    The first estimate is trivial, as $\widetilde{D}$ only adds an additional non-negative term.
    For the second inequality, we note that the functions in $\mathds{V}^p_{\boldsymbol{z}}$ all vanish outside of $\omega_{\boldsymbol{z}}$
    and therefore $\mathds{V}^p_{\boldsymbol{z}} \cap \operatorname{ker}(D) = \{0\}$.
    We transform to the reference patch, use the fact that $\widehat{D}$ is elliptic on $\widetilde{H}^{1/2}(\widehat{\omega}_{\boldsymbol{z}})$,
    and transform back by applying Lemma  \ref{item:lemma:reference-patch-hypsing_scaling}:
    \begin{align*}
      \left\|u\right\|_{L^2(\omega_{\boldsymbol{z}})}^2 &\lesssim h_{\boldsymbol{z}}^2 \left\|\widehat{u}\right\|_{L^2\left(\widehat{\omega}_{\boldsymbol{z}}\right)}^2 
      \lesssim h_{\boldsymbol{z}}^2 \left\|\widehat{u}\right\|_{\widetilde{H}^{1/2}(\widehat{\omega}_{\boldsymbol{z}})}^2 
      \lesssim h_{\boldsymbol{z}}^2 \left<\widehat{D} \widehat{u},\widehat{u} \right>_{\widehat{\omega}_{\boldsymbol{z}}}  
      \lesssim h_{\boldsymbol{z}} \, \left<D u,u \right>_{\omega_{\boldsymbol{z}}}.
    \end{align*}
    Thus, we can simply estimate the stabilization:
    \begin{align*}
      \alpha^2 \left<u,\mathds{1} \right>_{\omega_{\boldsymbol{z}}}^2 &\leq \alpha^2 \left\|u\right\|^2_{L^2(\omega_{\boldsymbol{z}})} \left\|\mathds{1}\right\|^2_{L^2(\omega_{\boldsymbol{z}})} 
      \leq C \alpha^2 \left\|\mathds{1}\right\|_{L^2(\omega_{\boldsymbol{z}})}^2 \,  h_{\boldsymbol{z}} \; \left<Du,u \right>_{\Gamma}.
    \end{align*}
    This gives the full estimate with the constant
    $c_1:= \max\left(1,\alpha^2 \left\|\mathds{1}\right\|_{L^2(\omega_{\boldsymbol{z}})}^2 \, C h_{\boldsymbol{z}}\right) 
      \leq \max\left(1, C \alpha^2 h_{\boldsymbol{z}}^3 \right)$.
  \end{proof}
\end{lemma}

\begin{remark}
  The proof of the previous lemma shows that this modification does not significantly affect the stability of the preconditioner and its effect will even vanish with
  $h$-refinement.
\hbox{}\hfill\rule{0.8ex}{0.8ex}
\end{remark}

We are now able to define the new local bilinear forms as:
\begin{definition}
\label{definition_perturbed_local_forms}
  Take $\boldsymbol{z} \in {\mathcal{V}}$ and let $F_{\boldsymbol{z}}: \widehat{\omega}_{\boldsymbol{z}} \to \omega_{\boldsymbol{z}}$ be the pullback mapping to the reference patch
  as in Definition~\ref{definition:reference-patches}. Set
  \begin{align*}
    \widetilde{a}_{\boldsymbol{z}}(u,v):=h_{\boldsymbol{z}}
    \left<\widehat{D} \left( \;u \circ F_{\boldsymbol{z}}\right), v \circ F_{\boldsymbol{z}} \right>_{\widehat{\omega}_{\boldsymbol{z}}} 
    \qquad \forall u,v \in \mathds{V}^p_{\boldsymbol{z}}.
  \end{align*}
  (see Lemma~\ref{item:lemma:reference-patch-hypsing_scaling} for the definition of $\widehat{D}$).
  We denote the Galerkin matrix corresponding the bilinear form $\widetilde{a}_{\boldsymbol{z}}$ on the reference patch by $\widehat{D}_{h,\operatorname{ref}(\boldsymbol{z})}^p$.
\end{definition}

The above definition only needs to evaluate $\left<\widehat{D} \widehat{u},\widehat{v} \right>_{\Gamma}$ on the reference patch.
Since the reference patch depends only on the number of elements belonging to the patch,
the number of blocks that need to be stored, depends only on the shape regularity and is independent of the number of vertices in the triangulation $\mathcal{T}$.

\begin{theorem}  
\label{thm:hp_reference_solver_preconditioner}
  Assume that $\mathcal{T}$ is generated from a regular and shape-regular initial triangulation $\mathcal{T}_0$ by successive application of NVB.
  The preconditioner using the local solvers from Definition \ref{definition_perturbed_local_forms} is optimal, i.e., for
  \begin{align*}
    B_3^{-1}:= R_h^T (B_h^1)^{-1} R_h  \; +
    \; \sum_{\boldsymbol{z} \in {\mathcal{V}}_L}{ h_{\boldsymbol{z}}^{-1} R^T_{\boldsymbol{z}} \left(\widehat{D}_{h,\operatorname{ref}(\boldsymbol{z})}^p\right)^{-1}} R_{\boldsymbol{z}},
  \end{align*}
  the condition number of the preconditioned system satisfies
  \begin{align*}
    \kappa(B_{3}^{-1} \tilde{D}_h^p) \leq C,
  \end{align*}
  where $C > 0$ depends only on $\Gamma$, $\mathcal{T}_0$ and the use of NVB refinement.
  It is in particular independent of $h$ and $p$.
  \begin{proof}
    The scaling properties of $\left<D u,u \right>_{\Gamma}$ were stated in Lemma~\ref{item:lemma:reference-patch-hypsing_scaling}.
    Therefore, we can conclude the argument by using the standard additive Schwarz theory. 
  \end{proof}
\end{theorem}

\subsubsection{Numerical realization}

\begin{figure}[h!t]
\begin{center}
\includegraphics{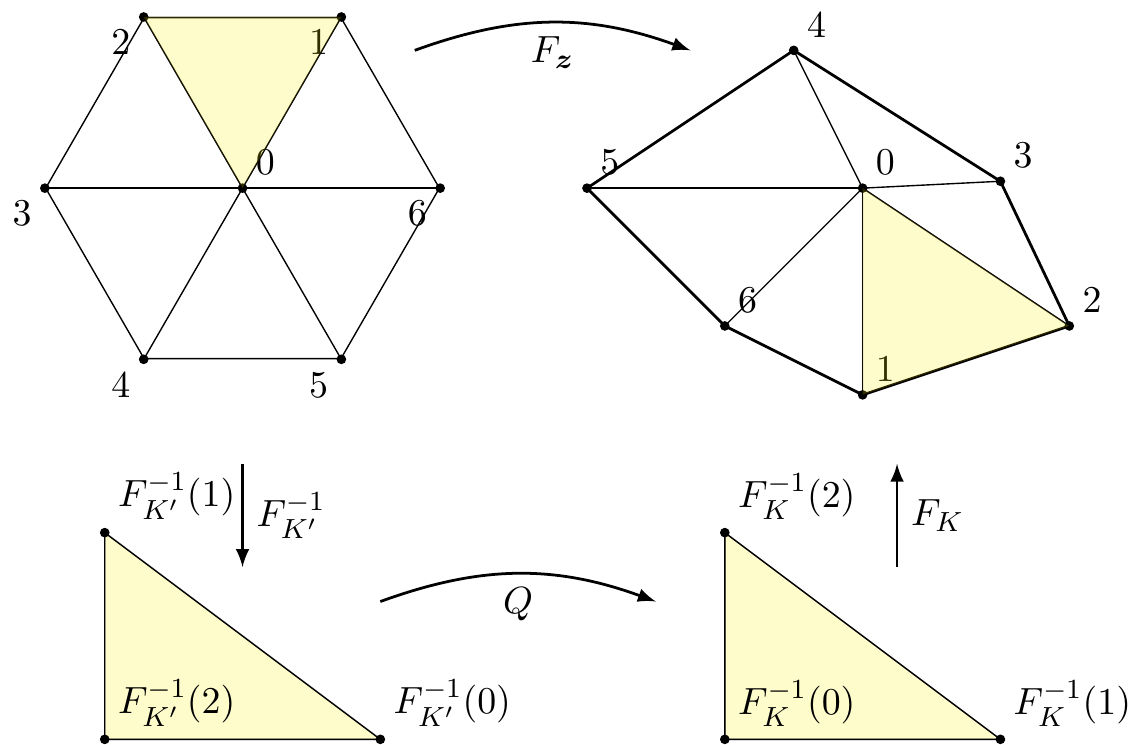}
\end{center}
\caption{The mapping to the reference patch described as a combination of element maps.}
\label{fig_patch_mapping}
\end{figure}

When implementing the preconditioner as defined above, it is important to note that for a basis function $\varphi_i$ on $\omega_{\boldsymbol{z}}$ 
the transformed function $\varphi_i \circ F_{\boldsymbol{z}}^{-1}$ does not necessarily correspond to the $i$-th basis function on 
$\widehat{\omega}_{\boldsymbol{z}}$. Depending on the chosen
basis we may run into orientation difficulties. This can be resolved in the following way:

Let $\boldsymbol{z} \in {\mathcal{V}}$ be fixed.
Choose a numbering for the vertices $\boldsymbol{z}_i$ and elements $K_i$ of $\omega_{\boldsymbol{z}}$ such that adjacent elements have adjacent numbers (for example, enumerate
clockwise or counter-clockwise). We also choose a similar enumeration on the reference patch and denote it as $\widehat{\boldsymbol{z}}_i$ and $\widehat{K}_i$.
The enumeration is such that the reference map $F_{\boldsymbol{z}}$ maps $\boldsymbol{z}_i$ to $\widehat{\boldsymbol{z}_i}$ and $K_i$ to $\widehat{K}_i$.
Let $N_{\boldsymbol{z}}$ be the number of vertices in the patch.

For elements $K \subset \omega_{\boldsymbol{z}}$ and $K' \subset \widehat{\omega}_{\boldsymbol{z}}$, the bases on $\omega_{\boldsymbol{z}}$ and on $\widehat{\omega}_{\boldsymbol{z}}$ are
locally defined by the pullback of polynomials on the reference triangle $\widehat{K}$.
We denote the element maps as $F_K: \widehat{K} \to K $ and $F_K': \widehat{K} \to K'$, respectively.
The basis functions are then given as $\varphi_j:=\widehat{\varphi}_j \circ F_K$ on $\omega_{\boldsymbol{z}}$ and
$\psi_j:=\widehat{\psi}_j \circ F_{K'}$ on $\widehat{\omega}_{\boldsymbol{z}}$.
Corresponding local element maps do not necessarily map the same vertices of the reference element $\widehat K$ to vertices 
with the same numbers in the local ordering. 
Hence, we need to introduce another map $Q: \widehat{K} \to \widehat{K}$ that represents a vertex permutation.
Then, we can write the patch-pullback restricted to $K'$ as $F_{\boldsymbol{z}}|_{K'}=F_K \circ Q \circ F_{K'}^{-1}$ (see Figure~\ref{fig_patch_mapping}).
We observe:
\begin{enumerate}[i)]
  \item For the hat function the mapping is trivial:
    $\varphi_{\boldsymbol{z}} \circ F_{\boldsymbol{z}}=\psi_{\widehat{\boldsymbol{z}}}$.
  \item For the edge basis, permuting the vertices on the reference element only changes the sign of the corresponding edge functions.
    Thus, we have $\varphi^{{\mathcal E}_m}_{j} \circ F_{\boldsymbol{z}}= (-1)^j \psi^{{\mathcal E}_m}_{j}$, if the orientation of the edge in the global triangulation does not
    match the orientation of the reference patch.    
  \item The inner basis functions transformation under $Q$ is not so simple. Since the basis functions all have support on a single
    element we can restrict our consideration to this element and assemble the necessary basis transformations for all 5 permutations of vertices on
    the reference triangle without losing the memory advantage of using the reference patch.  
\end{enumerate}

\begin{remark}
  One could also exploit the symmetry (up to a sign change) of the permutation of $\lambda_1$ and $\lambda_2$ in the definition of the
  inner basis functions to reduce the number of basis transformation matrices needed from 5 to 2. 
\hbox{}\hfill\rule{0.8ex}{0.8ex}
\end{remark}

\section{Numerical results}
\label{sec:numerics}
The following numerical experiments confirm that the proposed preconditioners
(Theorem~\ref{thm:ppreconditioner}, Theorem~\ref{thm:p_precond_with_multilevel},
and Theorem~\ref{thm:hp_reference_solver_preconditioner})
do indeed yield a system with a condition number that is bounded uniformly in $h$ and $p$, whereas the
condition number of the unpreconditioned system grows in $p$ with a rate slightly smaller than predicted in 
Corollary~\ref{thm:condition_numbers_D}: We observe numerically $\kappa \sim \mathcal{O}(p^{5.5})$.
Diagonal preconditioning appears to reduce the condition number to $\mathcal{O}(p^{2.5})$. 
All of the following experiments were performed using the BEM++ software library (\cite{bempp_preprint}; \url{www.bempp.org}) with the AHMED software library 
for $\mathcal{H}$-matrix compression, \cite{bebendorf:2008},  \cite{ahmed_homepage}.
We used the polynomial basis described in Section~\ref{sect:discretization}.

\begin{figure}[tb]
\includegraphics{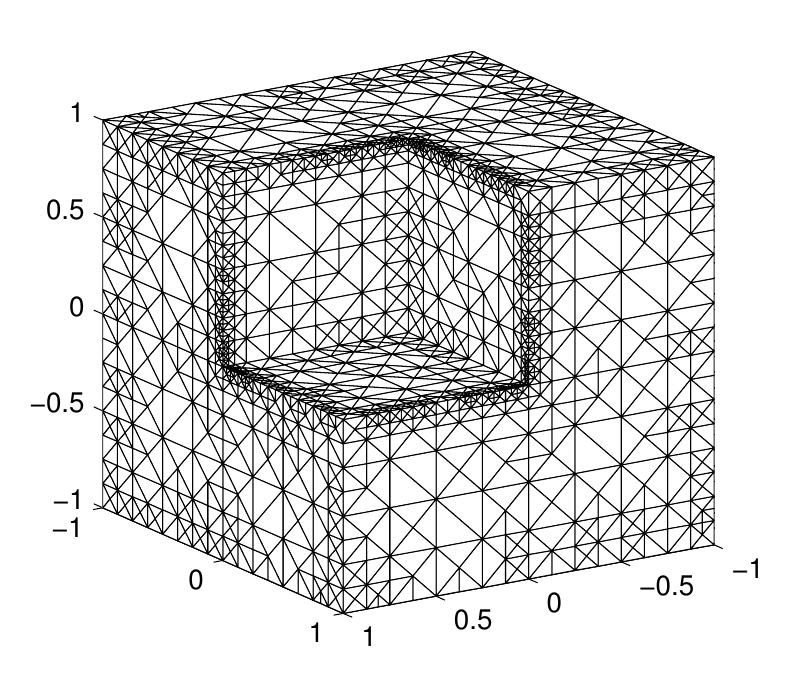}
\includegraphics{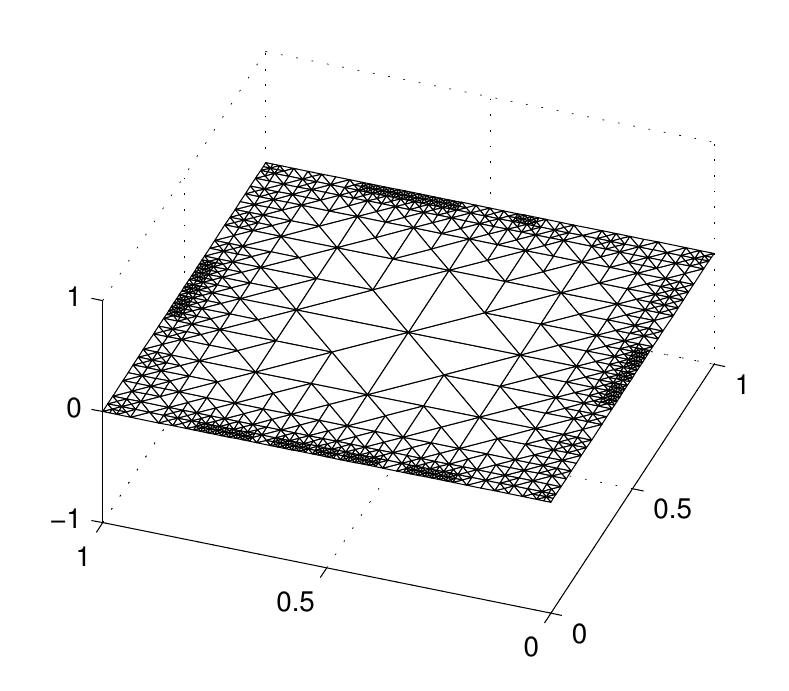}
\caption{Adaptive meshes on the Fichera cube and for a screen problem.}
\label{fig_meshes}
\end{figure}

\begin{example}[unpreconditioned $p$-dependence]
\label{example_unprec}
  We consider a quadratic screen in $\mathbb{R}^3$ (see Figure~\ref{fig_meshes}, right).
We study the $p$-dependence 
  of the unpreconditioned system on different uniformly refined meshes.  
  In accordance with the estimates of Corollary~\ref{thm:condition_numbers_D}, 
  Figure~\ref{fig:unpreconditioned_systems} shows that one has, depending on the mesh size $h$,  
  a preasymptotic phase in which the 
  $\mathcal{O}(h^{-1}p^2)$ term dominates, and an $h$-independent asymptotic 
  $\mathcal{O}(p^{5.5})$ behavior. The latter is slightly better than the prediction 
  of $\mathcal{O}(p^6)$ of Corollary~\ref{thm:condition_numbers_D}. 
\hbox{}\hfill\rule{0.8ex}{0.8ex}
\begin{figure}[h]
\includegraphics{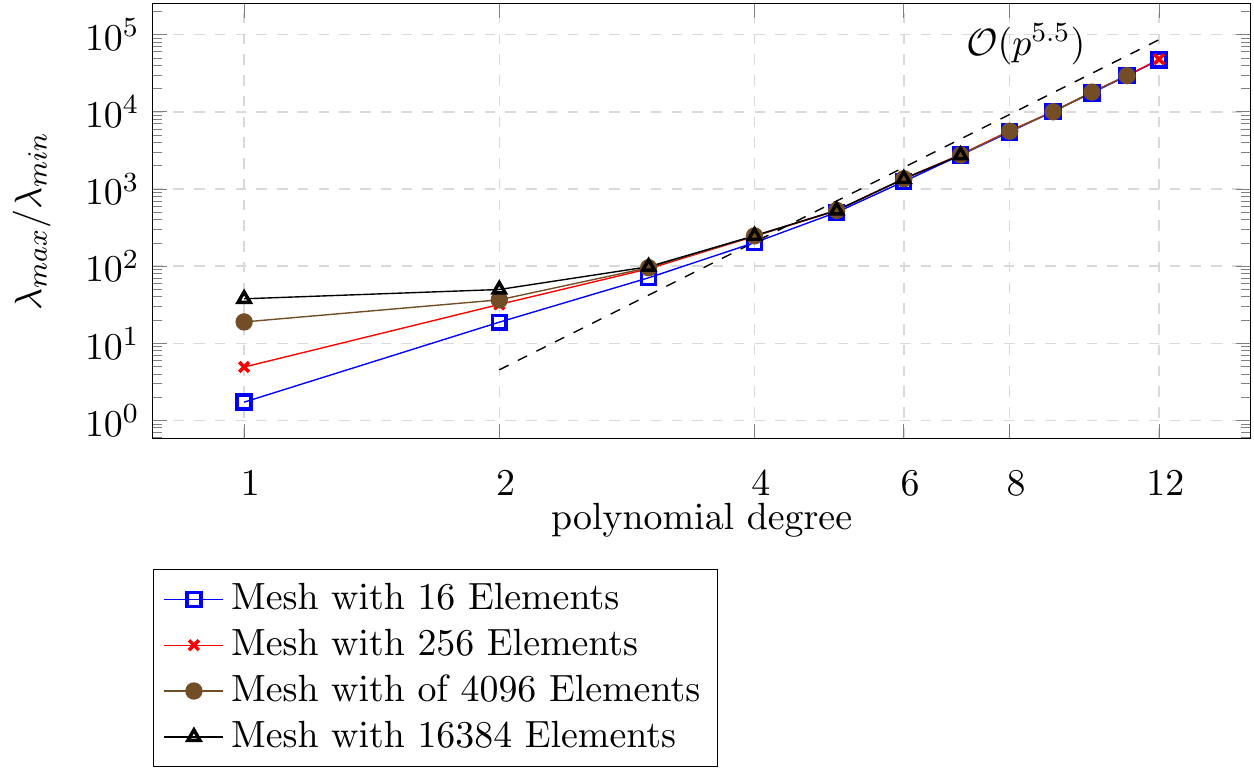}

\caption{Comparison of the condition number of $\widetilde{D}^p_h$ for the screen problem on different uniform meshes (Example~\ref{example_unprec})}
\label{fig:unpreconditioned_systems}
\end{figure}
\end{example}

\begin{example}[Fichera's cube]
\label{example_fichera_prec}
We compare the preconditioner that uses the local multilevel preconditioner for the $h$-part and the
inexact local solvers based on the reference patches to the unpreconditioned system and to simple diagonal
scaling.
We consider the problem on a closed surface, namely, the surface of the Fichera cube with side length $2$, and employ a stabilization parameter
$\alpha=0.2$.
To generate the adaptive meshes, we used NVB, where in each step, the set of marked elements originated from
a lowest order adaptive algorithm with a ZZ-type error estimator (as described in \cite{aff_hypsing}).
The left part of Figure~\ref{fig_meshes} shows an example of one of the meshes used.

Figure~\ref{fig_fichera_p} confirms that the condition number of the preconditioned system does not depend 
on the polynomial degree of the discretization. Figure~\ref{fig_fichera_h} confirms the robustness of the 
preconditioner with respect to the adaptive refinement level. 
The unpreconditioned and the diagonally preconditioned system do not show a bad behavior with respect to $h$, 
probably due to the already large condition number for $p > 1$.
\hbox{}\hfill\rule{0.8ex}{0.8ex}
\begin{figure}[htb]
\includegraphics{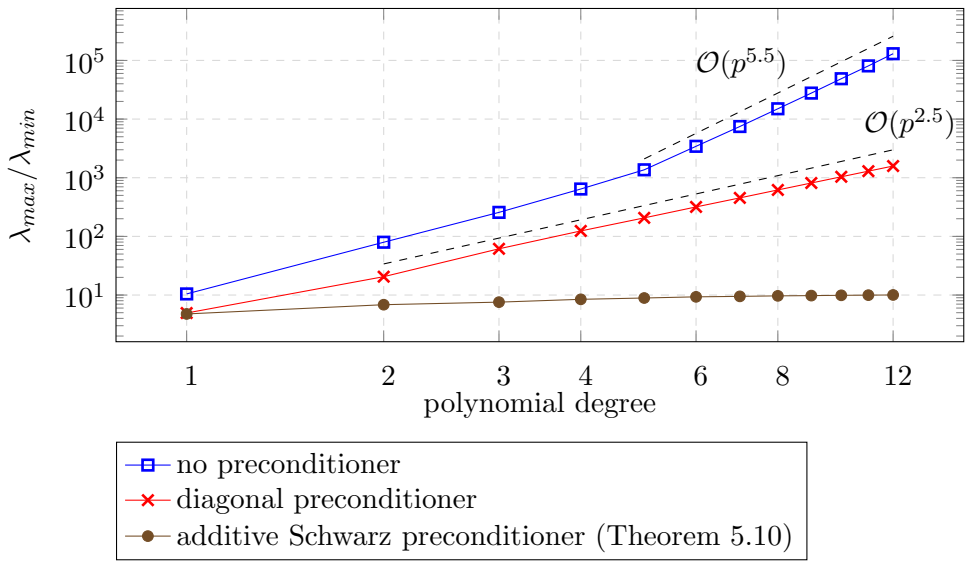}
\caption{Fichera cube, condition numbers for fixed uniform mesh with $70$ elements (Example~\ref{example_fichera_prec}).}
\label{fig_fichera_p}
\end{figure}
\begin{figure}[h!tb]
\includegraphics{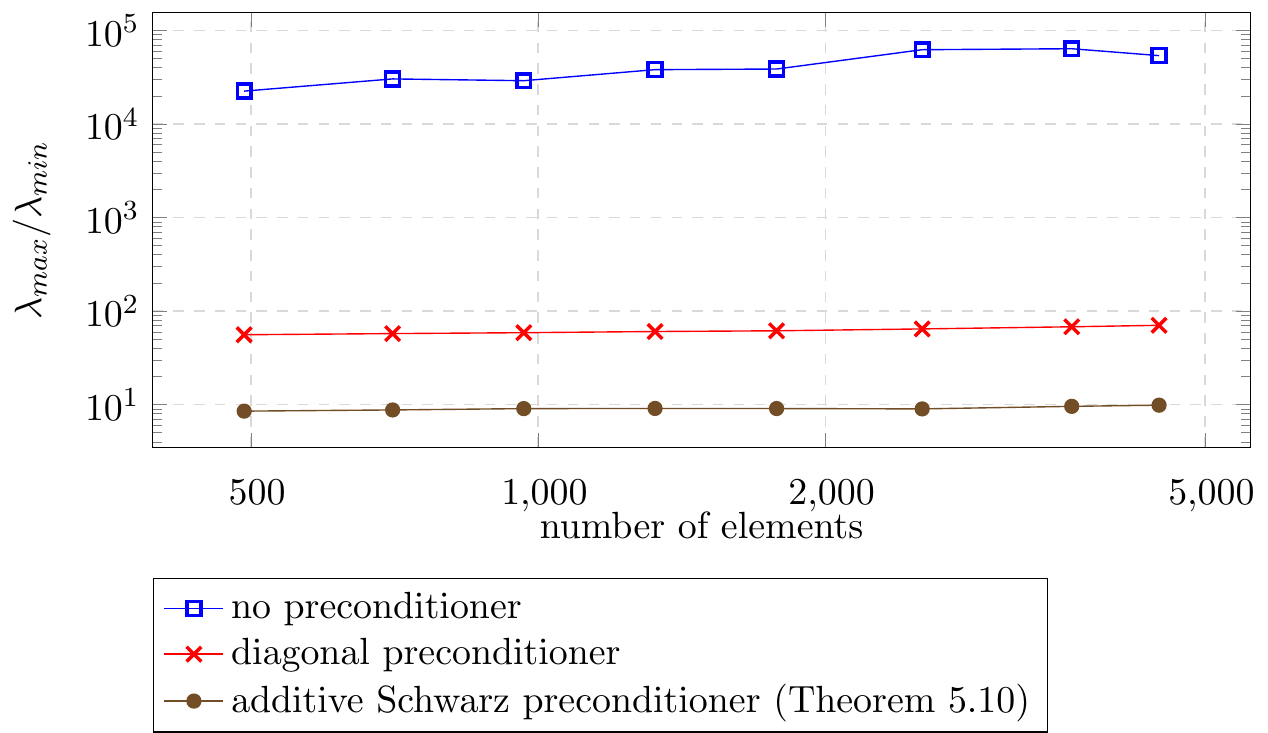}
\caption{Fichera cube, adaptive $h$-refinement for $p=3$ (Example~\ref{example_fichera_prec}).}
\label{fig_fichera_h}
\end{figure}
\end{example}

\begin{example}[screen problem]
\label{example_screen_prec}
We consider the screen problem in $\mathbb{R}^3$ with a quadratic screen of side length~1 (see Figure~\ref{fig_meshes},
  right),
  which    represents the case $\Gamma \neq \partial \Omega$ and $\alpha=0$ in~\eqref{eq:def_stabilized_op}, and
  perform the same experiments as we did for Fichera's cube in Example~\ref{example_fichera_prec}.
  In Figure~\ref{fig_screen_p} we again observe that the condition number is independent of the polynomial degree. 
  Figures~\ref{fig_screen_h}--\ref{fig_screen_h5} demonstrate the independence of the mesh size $h$.
\hbox{}\hfill\rule{0.8ex}{0.8ex}
\begin{figure}[htb]

\includegraphics{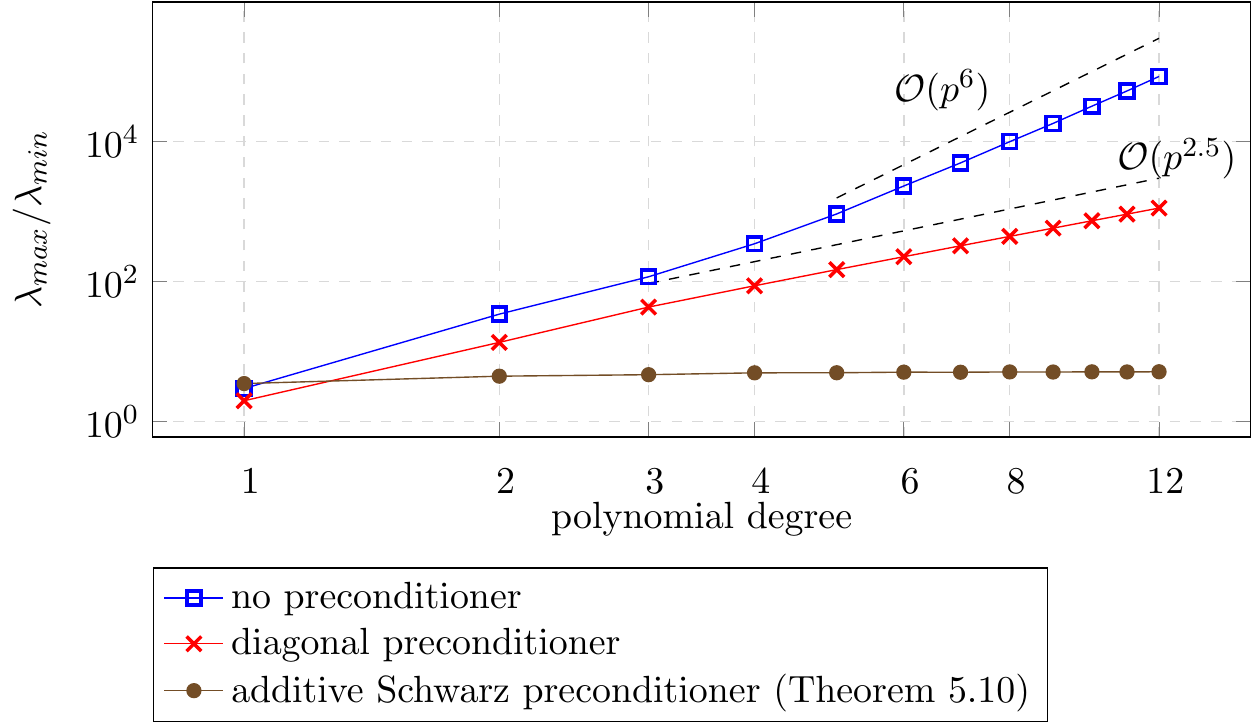}
\caption{Screen problem, condition numbers for uniform mesh with $45$ elements (Example~\ref{example_screen_prec}). }
\label{fig_screen_p}
\end{figure}

\begin{figure}[ht]
\includegraphics{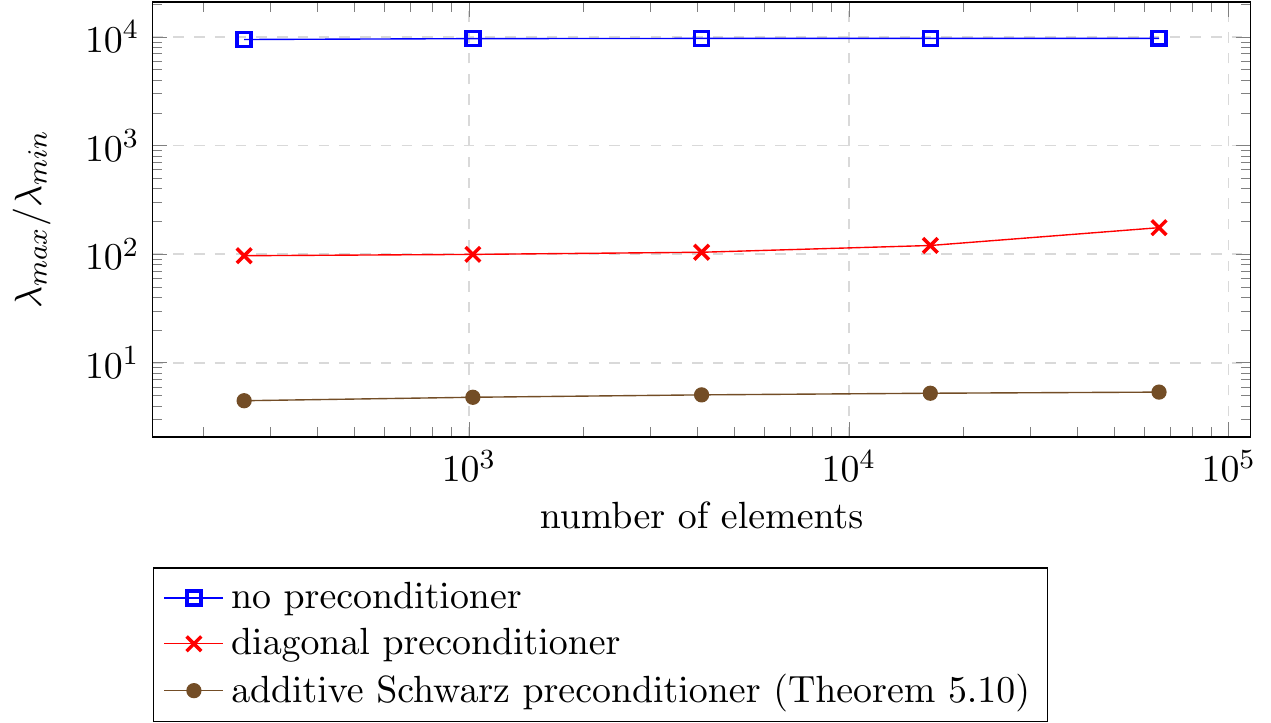}
\caption{Screen problem, uniform $h$-refinement for $p=4$ (Example~\ref{example_screen_prec}).}
\label{fig_screen_h}
\end{figure}

\begin{figure}[ht]
\includegraphics{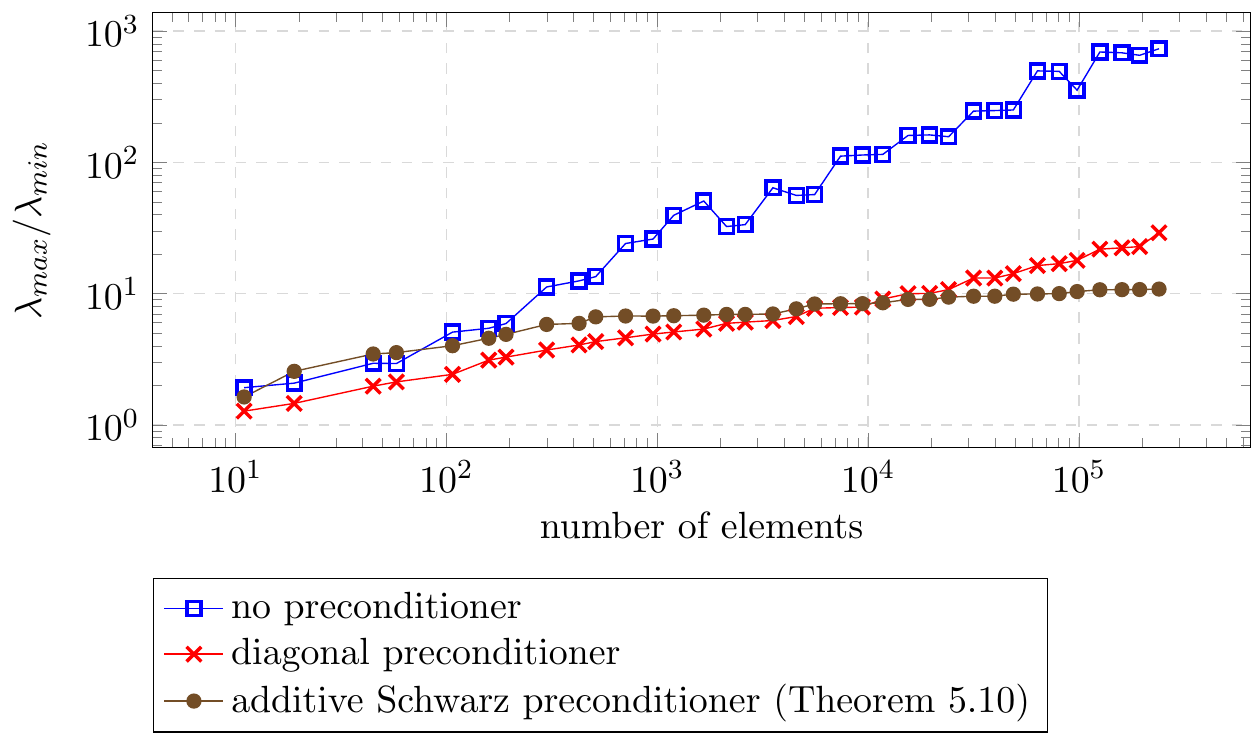}
\caption{Screen problem, adaptive $h$-refinement for $p=1$ (Example~\ref{example_screen_prec}).}
\label{fig_screen_h3}
\end{figure}

\begin{figure}[ht]
\includegraphics{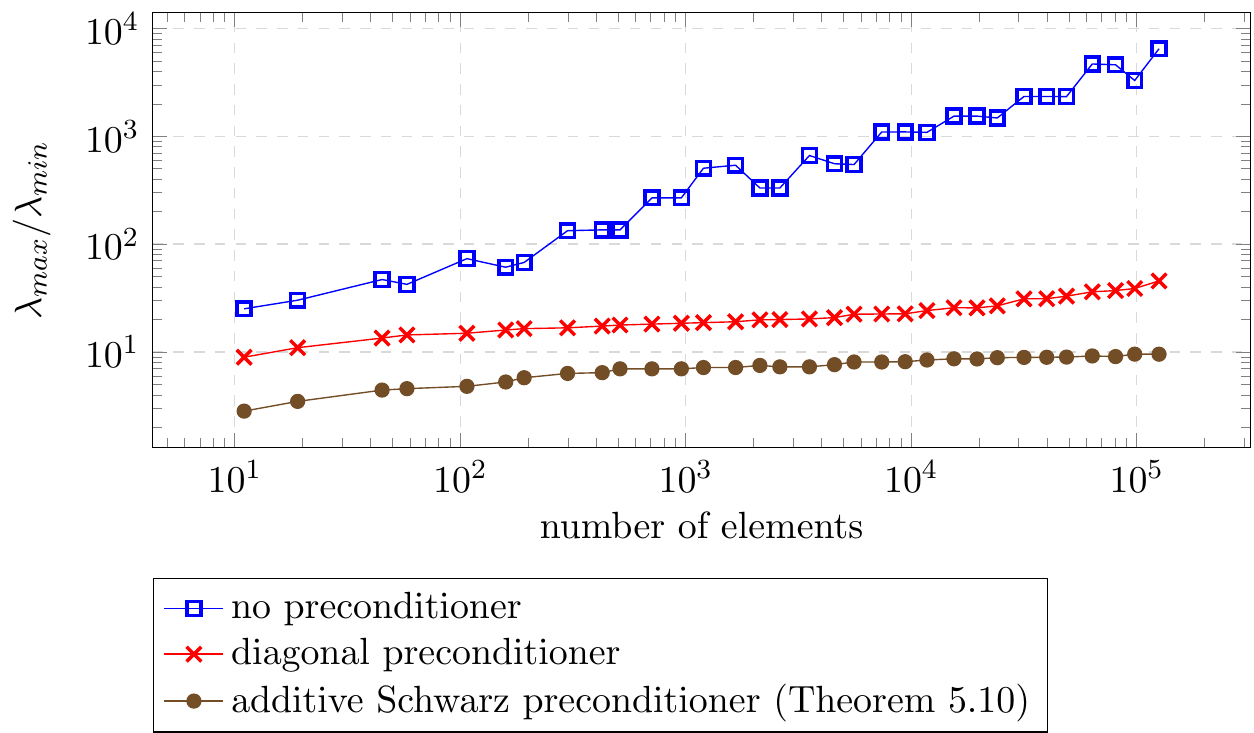}
\caption{Screen problem, adaptive $h$-refinement for $p=2$ (Example~\ref{example_screen_prec}).}
\label{fig_screen_h4}
\end{figure}

\begin{figure}[ht]
\includegraphics{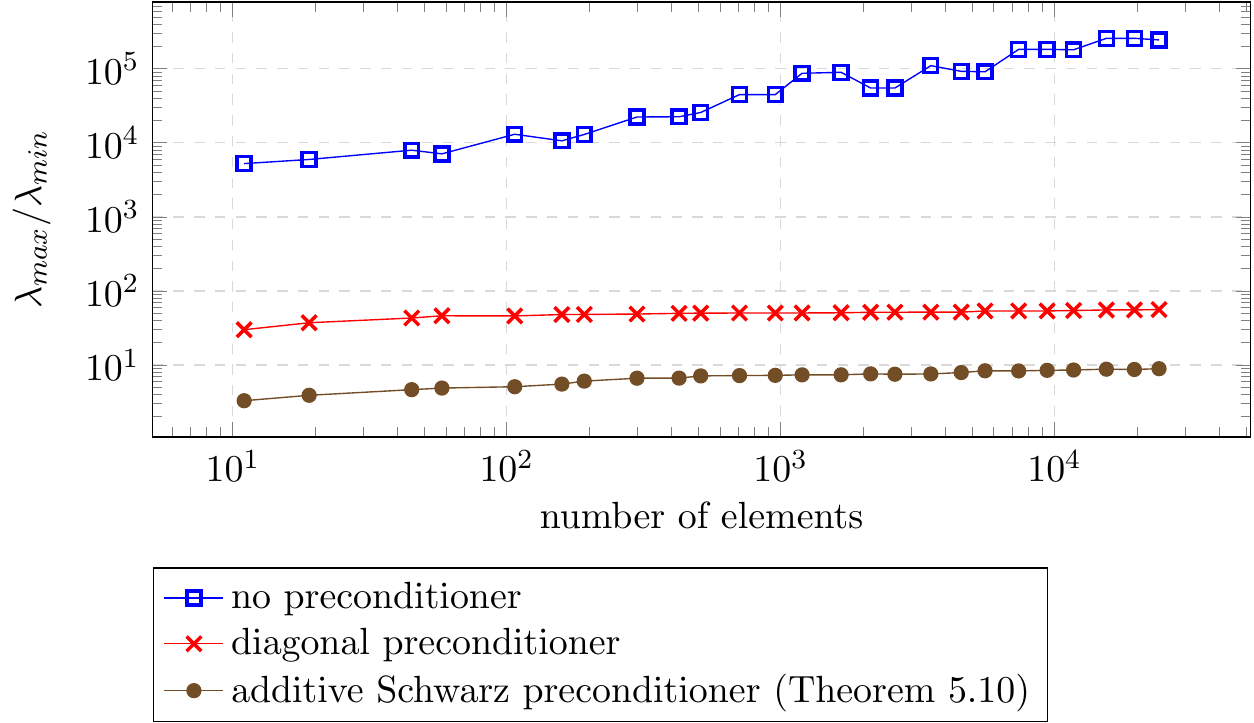}
\caption{Screen problem, adaptive $h$-refinement for $p=3$ (Example~\ref{example_screen_prec}).}
\label{fig_screen_h5}
\end{figure}
\end{example}

\afterpage{\clearpage}

\begin{example}[inexact local solvers]
  \label{example_prec_compare}
We compare the different preconditioners proposed in this paper.
  While the numerical experiments all show that the preconditioner is indeed robust in $h$ and $p$,
  the constant differs if we use the different simplifications
  described in the Sections~\ref{section_adaptive_meshes} and \ref{section_inexact_locals} to the preconditioner.
  In Figures~\ref{fig_compare_preconditioners_p} and \ref{fig_compare_preconditioners_h}, we can observe 
  the different constants for the geometry given by Fichera's cube of Example~\ref{example_fichera_prec}.
\hbox{}\hfill\rule{0.8ex}{0.8ex}
\begin{figure}[ht]
\includegraphics{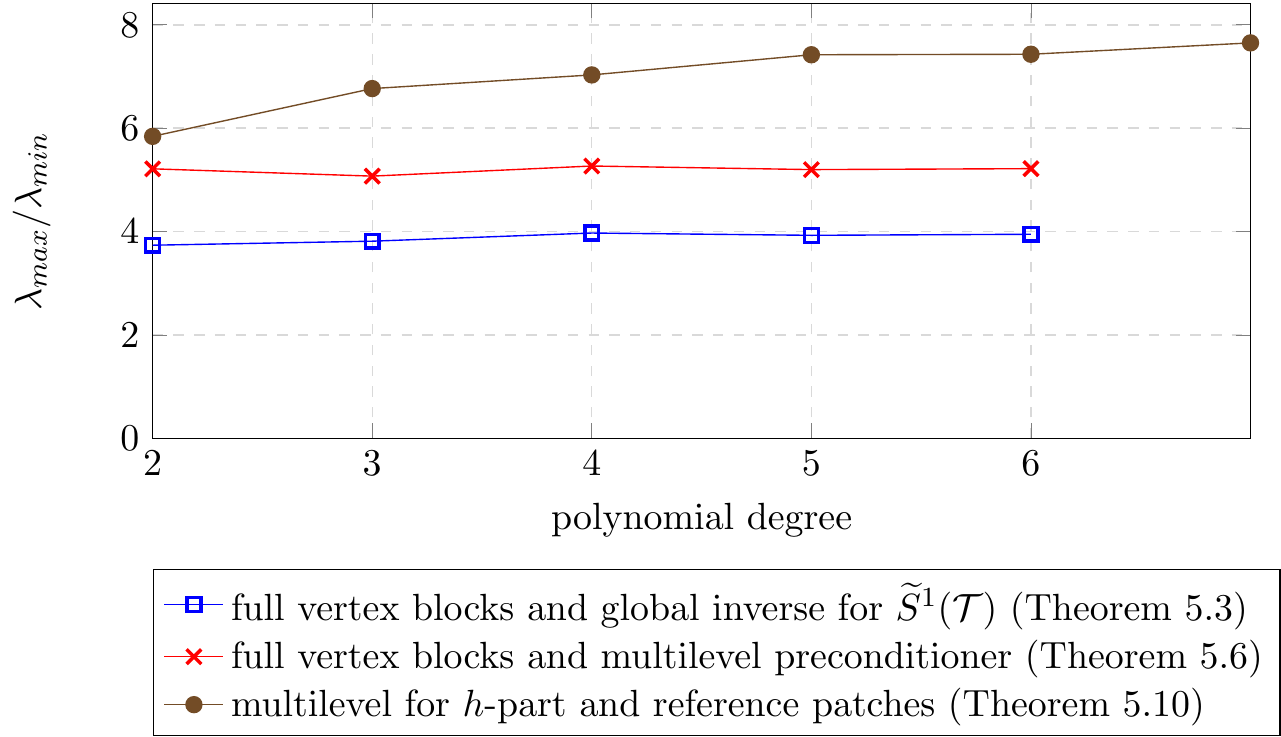}
\caption{Comparison of the different proposed preconditioners for a fixed uniform mesh with $70$ elements on the Fichera cube (Example~\ref{example_prec_compare}).}
\label{fig_compare_preconditioners_p}
\end{figure}

\begin{figure}[ht]
\includegraphics{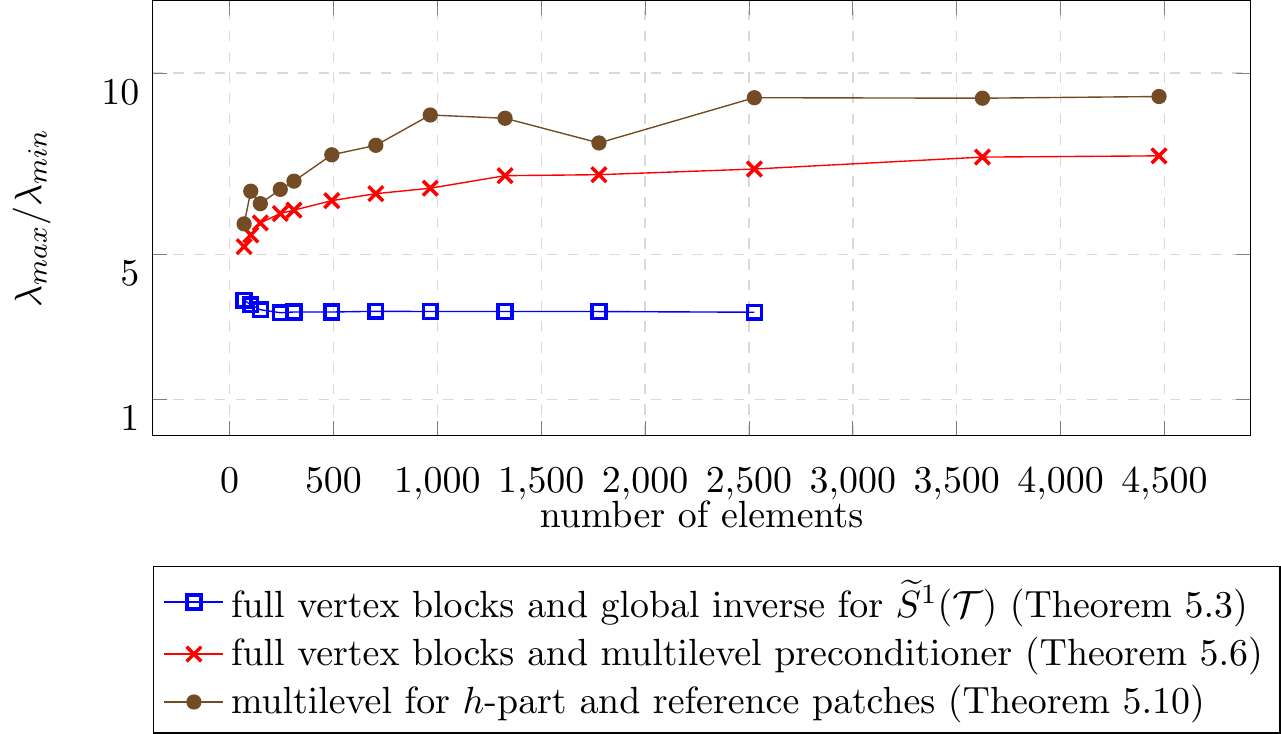}
\caption{Comparison of the different proposed preconditioners for adaptive mesh refinement on the Fichera cube with $p=2$.}
\label{fig_compare_preconditioners_h}
\end{figure}
\end{example}

\clearpage

\begin{example}[inexact local solvers]
\label{example_memory}
We continue with the geometry of Example~\ref{example_prec_compare}, i.e., Fichera's cube. 
We motivated Section~\ref{section_inexact_locals} by stating the large memory requirement of the preconditioner when storing the dense local block inverses. 
It can be seen in Table~\ref{table_memory_requirement} that the reference 
patch based preconditioner resolves this issue: we present the memory requirements for the various approaches when 
excluding the memory requirement for the treatment of the lowest order space $\mathds{V}_h^1$.  
For comparison, we included the storage requirements for the full matrix $\widetilde{D}_h^p$ 
and the $\mathcal{H}$-matrix approximation with accuracy $10^{-8}$ which is denoted as $D^{p,\mathcal{H}}_h$.
While we still get linear growth in the number of elements, due to some bookkeeping requirements, 
such as element orientation etc., which could theoretically also be avoided,
we observe a much reduced storage cost.
For $p=3$ and $55,298$ degrees of freedom, the memory requirement is less than $2.5\%$ of the full block storage. 
For $p=4$ and $393,218$ degrees of freedom the memory requirement is just $0.6\%$ and for higher polynomial orders, this ratio would become even smaller.
Comparing only the number of blocks that need to be stored,
we see that in this particular geometry we only need to store the inverse for $6$ reference blocks.
\hbox{}\hfill\rule{0.8ex}{0.8ex}

\begin{table}[h!]
  \begin{center}
                                                                                                                                \clearpage{}\begin {tabular}{SSSSSS}\toprule \multicolumn {1}{c}{$p$}&\multicolumn {1}{c}{$N_{dof}$}&\multicolumn {1}{c}{$\operatorname {mem}\left (\widetilde {D}^p_h\right )/{N_{dof}}$}&\multicolumn {1}{c}{$\operatorname {mem}\left (\widetilde {D}^{p,\mathcal {H}}_h\right )/N_{dof}$}&\multicolumn {1}{c}{$\operatorname {mem}\left (B^{-1}\right )/N_{dof}$}&\multicolumn {1}{c}{$\operatorname {mem}\left (B_{3}^{-1}\right )/N_{dof}$}\\& & [KB] & [KB] & [KB] & [KB] \\ \midrule 2&98&0.76562&0.76945&0.094547&0.039222\\%
2&298&2.3281&2.3334&0.098259&0.027318\\%
2&986&7.7031&7.1907&0.10025&0.020522\\%
2&3558&27.797&15.006&0.10068&0.018394\\%
2&8950&69.922&19.817&0.10078&0.017902\\%
\toprule
3&218&1.7031&1.7103&0.31314&0.083787\\%
3&668&5.2188&5.1523&0.32508&0.041261\\%
3&2216&17.312&10.921&0.332&0.017895\\%
3&5969&46.633&16.452&0.3343&0.011556\\%
3&16310&127.42&22.738&0.33274&0.0091824\\%
3&20135&157.3&24.04&0.33372&0.0089222\\%
\toprule
4&386&3.0156&3.0197&0.67086&0.16973\\%
4&1954&15.266&10.879&0.70233&0.04829\\%
4&5634&44.016&17.39&0.71085&0.019619\\%
4&14226&111.14&21.209&0.71364&0.010424\\%
4&35794&279.64&27.596&0.71428&0.0067908\\%
\toprule
5&602&4.7031&4.7083&1.1694&0.29324\\%
5&1852&14.469&14.476&1.2122&0.12945\\%
5&8802&68.766&68.772&1.2384&0.029459\\%
5&16577&129.51&129.4&1.2468&0.016961\\%
5&45302&353.92&353.79&1.2404&0.0079898\\%
5&55927&436.93&436.78&1.2443&0.0070062\\\bottomrule \end {tabular}\clearpage{}
  \end{center}
  \caption{Comparison of the memory requirement relative to the number of degrees of freedom~$N_{dof}$     between storing the full block structure and the reference block based preconditioner from Section~\ref{section_inexact_locals} (Example~\ref{example_memory}).}  
  \label{table_memory_requirement}
\end{table}
\end{example}

 \clearpage
\textbf{Acknowledgments:} The research was supported by the Austrian Science Fund (FWF) through the 
doctoral school ``Dissipation and Dispersion in Nonlinear PDEs'' (project W1245, A.R.) 
and ``Optimal Adaptivity for BEM and FEM-BEM Coupling'' (project P27005, T.F, D.P.). 
T.F. furthermore acknowledges funding through the Innovative Projects Initiative of 
Vienna University of Technology and the CONICYT project ``Preconditioned linear solvers
for nonconforming boundary elements'' (grant FONDECYT 3150012).

\def\appendixname{Appendix}
\appendix
\section*{Appendix}
From \cite[Prop.~{2.8}]{melenk_gl} for the stiffness matrix and from classical estimates for the quadrature weights of the Gau{\ss}-Lobatto quadrature, we have on the reference square $\widehat{S}$ 
\begin{align*} 
p^{-2} \|{\mathfrak u}\|^2_{\ell^2} \lesssim \|u\|^2_{H^1(\widehat S)} \lesssim p \|{\mathfrak u}\|^2_{\ell^2}, 
\qquad 
p^{-4} \|{\mathfrak u}\|^2_{\ell^2} \lesssim \|u\|^2_{L^2(\widehat S)} \lesssim p^{-2} \|{\mathfrak u}\|^2_{\ell^2} 
\qquad \forall u \in {\mathcal Q}^p(\widehat S)
\end{align*}
For quasi-uniform meshes, we therefore obtain 
\begin{align*}
h^{-2} p^2 \|u\|^2_{L^2(\Gamma)} & \lesssim \|{\mathfrak u}\|^2_{\ell^2} \lesssim h^{-2} p^4 \|u\|^2_{L^2(\Gamma)}.
\end{align*}
Furthermore, we have 
\begin{align*}
\|{\mathfrak J}_h {\mathfrak u}\|^2_{\ell^2} &\lesssim h^{-2} \|J_h u\|^2_{L^2(\Gamma)} 
\lesssim h^{-2} \|u\|^2_{L^2(\Gamma)} \lesssim  p^{-2} \|{\mathfrak u}\|^2_{\ell^2},\\
\|{\mathfrak u} - {\mathfrak J}_h {\mathfrak u}\|^2_{\ell^2} &\lesssim 
\sum_{K \in {\mathcal T}} \|({\mathfrak u} - {\mathfrak J}_h{\mathfrak u})|_K)\|^2_{\ell^2} 
\lesssim p^2 \sum_{K \in {\mathcal T}} \|\widehat u - \widehat {J_h u}\|^2_{H^1(\widehat S)}
\lesssim p^2 \sum_{K \in {\mathcal T}} | u - J_h u|^2_{H^1(K)} + h^{-2} \|u - J_h u\|^2_{L^2(K)} \\
&\lesssim p^2 \|u\|^2_{H^1(\Gamma)}.
\end{align*}
We obtain 
\begin{align*}
\|u\|^2_{L^2(\Gamma)} \lesssim h^2 p^{-2} \|{\mathfrak u}\|^2_{\ell^2}, 
\qquad  
\|u\|^2_{H^1(\Gamma)}  = \|u\|^2_{L^2(\Gamma)} + |u|^2_{H^1(\Gamma)} \lesssim 
h^2 p^{-2} \|{\mathfrak u}\|^2_{\ell^2} + p \|{\mathfrak u}\|^2_{\ell^2} 
\lesssim p \|{\mathfrak u}\|^2_{\ell^2},
\end{align*}
so that interpolation yields 
$$
\|u\|^2_{\widetilde H^{1/2}(\Gamma)}  \lesssim h p^{-1/2} \|{\mathfrak u}\|^2_{\ell^2}. 
$$
For the converse estimate, we observe 
\begin{align*}
\|{\mathfrak u} - {\mathfrak J}_h {\mathfrak u}\|^2_{\ell^2} \lesssim p^2 \|u\|^2_{H^1(\Gamma)}, 
\qquad \qquad 
\|{\mathfrak u} - {\mathfrak J}_h {\mathfrak u}\|^2_{\ell^2} \lesssim h^{-2} p^4 \|u\|^2_{L^2(\Gamma)} ,
\end{align*}
so that an interpolation argument (which we assume to be admissible!)  produces 
\begin{align*}
\|{\mathfrak u} - {\mathfrak J}_h {\mathfrak u}\|^2_{\ell^2} &\lesssim p^3 h^{-1} \|u\|^2_{\widetilde H^{1/2}(\Gamma)}. 
\end{align*}
Hence, 
\begin{align*}
\|{\mathfrak u}\|^2_{\ell^2} &\lesssim 
\|{\mathfrak u} - {\mathfrak J}_h {\mathfrak u}\|^2_{\ell^2} + 
\|{\mathfrak J}_h {\mathfrak u}\|^2_{\ell^2} 
\lesssim p^3 h^{-1} \|u\|^2_{\widetilde H^{1/2}(\Gamma)} + h^{-2} \|u\|^2_{L^2(\Gamma)} 
\lesssim \left( p^3 h^{-1} + h^{-2} \right) \|u\|^2_{\widetilde H^{1/2}(\Gamma)}. 
\end{align*}
Putting things together, we get 
$$
p^{1/2} h^{-1} \|u\|^2_{\widetilde H^{1/2}(\Gamma)} \lesssim \|{\mathfrak u}\|^2_{\ell^2} 
\lesssim \left( p^3 h^{-1} + h^{-2} \right) \|u\|^2_{\widetilde H^{1/2}(\Gamma)},
$$
which in turn gives the condition number estimate 
$$
\kappa(\widetilde D^p_h) \lesssim p^{5/2} + h^{-1} p^{-1/2}. 
$$
For the $H^1$-condition number, we note the estimates 
\begin{align*}
\|u\|^2_{H^1(\Gamma)}  &= 
|u|^2_{H^1(\Gamma)}  + 
\|u\|^2_{L^2(\Gamma)}  \lesssim p \|{\mathfrak u}\|^2_{\ell^2} + h^2 p^{-2} \|{\mathfrak u}\|^2_{\ell^2} 
\lesssim p \|{\mathfrak u}\|^2_{\ell^2},  \\
\|{\mathfrak u}\|^2_{\ell^2} &\lesssim 
\|{\mathfrak u} - {\mathfrak J}_h {\mathfrak u}\|^2_{\ell^2} + 
\|{\mathfrak J}_h {\mathfrak u}\|^2_{\ell^2} \lesssim p^2 \|u\|^2_{H^1(\Gamma)} + h^{-2} \|u\|^2_{L^2(\Gamma)}
\lesssim \left( p^2 + h^{-2} \right)\|u\|^2_{H^1(\Gamma)}, 
\end{align*}
so that we get 
$$
p^{-1} \|u\|^2_{H^1(\Gamma)} \lesssim \|{\mathfrak u}\|^2_{\ell^2} \lesssim \left( p^2 + h^{-2}\right) \|u\|^2_{H^1(\Gamma)}
$$
\section*{References}
\bibliographystyle{elsart-num-sort}
\bibliography{asm_bibliography}

\begin{thebibliography}{10}
\expandafter\ifx\csname url\endcsname\relax
  \def\url#1{\texttt{#1}}\fi
\expandafter\ifx\csname urlprefix\endcsname\relax\def\urlprefix{URL }\fi

\bibitem{ainsworth-guo00}
M.~Ainsworth, B.~Guo, An additive {S}chwarz preconditioner for $p$-version
  boundary element approximation of the hypersingular operator in three
  dimensions, Numer. Math. 85 (2000) 343--366.

\bibitem{amcl03}
M.~Ainsworth, W.~McLean, Multilevel diagonal scaling preconditioners for
  boundary element equations on locally refined meshes, Numer. Math. 93~(3)
  (2003) 387--413.

\bibitem{aff_hypsing}
M.~Aurada, M.~Feischl, T.~F{\"u}hrer, M.~Karkulik, D.~Praetorius, Energy norm
  based error estimators for adaptive {BEM} for hypersingular integral
  equations, Appl. Numer. Math. (2015) published online first.

\bibitem{bebendorf:2008}
M.~Bebendorf, {Hierarchical Matrices: A Means to Efficiently Solve Elliptic
  Boundary Value Problems}, vol.~63 of Lect. Notes Comput. Sci. Eng.,
  Springer-Verlag, 2008, iSBN 978-3-540-77146-3.

\bibitem{ahmed_homepage}
M.~Bebendorf, Another software library on hierarchical matrices for elliptic
  differential equations ({AHMED}),
  \texttt{http://bebendorf.ins.uni-bonn.de/AHMED.html} (Jun. 2014).

\bibitem{bernardi_girault_regularization_operator}
C.~Bernardi, V.~Girault, A local regularization operator for triangular and
  quadrilateral finite elements, SIAM J. Numer. Anal. 35~(5) (1998) 1893--1916.

\bibitem{bernardi-maday97}
C.~Bernardi, Y.~Maday, Spectral methods, in: P.~Ciarlet, J.~Lions (eds.),
  Handbook of Numerical Analysis, Vol. 5, North Holland, Amsterdam, 1997.

\bibitem{ckns}
J.~M. Cascon, C.~Kreuzer, R.~H. Nochetto, K.~G. Siebert, Quasi-optimal
  convergence rate for an adaptive finite element method, SIAM J. Numer. Anal.
  46~(5) (2008) 2524--2550.

\bibitem{dubiner91}
M.~Dubiner, Spectral methods on triangles and other domains, J. Sci. Comp. 6
  (1991) 345--390.

\bibitem{falk-winther13}
R.~Falk, R.~Winther, The bubble transform: A new tool for analysis of finite
  element methods, Found. Comput. Math. (2015) 1--32.

\bibitem{partOne}
M.~Feischl, T.~F\"uhrer, M.~Karkulik, J.~M. Melenk, D.~Praetorius,
  {Q}uasi-optimal convergence rates for adaptive boundary element methods with
  data approximation, part {I}: weakly-singular integral equation, Calcolo 51
  (2014) 531--562.

\bibitem{partTwo}
M.~Feischl, T.~F{\"u}hrer, M.~Karkulik, J.~M. Melenk, D.~Praetorius,
  Quasi-optimal convergence rates for adaptive boundary element methods with
  data approximation. {P}art {II}: {H}yper-singular integral equation,
  Electron. Trans. Numer. Anal. 44 (2015) 153--176.

\bibitem{ffp_adaptive_fem}
M.~Feischl, T.~F{\"u}hrer, D.~Praetorius, Adaptive {FEM} with optimal
  convergence rates for a certain class of non-symmetric and possibly
  non-linear problems, SIAM J. Numer. Anal. 52(2) (2014) 601 -- 625.

\bibitem{ffps}
M.~Feischl, T.~F\"uhrer, D.~Praetorius, E.~P. Stephan, Efficient additive
  {S}chwarz preconditioning for hypersingular integral equations on locally
  refined triangulations, ASC Report {\bf 25/2013}, Vienna University of
  Technology.

\bibitem{fkmp13}
M.~Feischl, M.~Karkulik, J.~M. Melenk, D.~Praetorius, Quasi-optimal convergence
  rate for an adaptive boundary element method, SIAM J. Numer. Anal. 51~(2)
  (2013) 1327--1348.

\bibitem{dissTF}
T.~F{\"u}hrer, Zur {K}opplung von finiten {E}lementen und {R}andelementen,
  Ph.D. thesis, Vienna University of Technology, in German (2014).

\bibitem{gantumur}
T.~Gantumur, Adaptive boundary element methods with convergence rates, Numer.
  Math. 124~(3) (2013) 471--516.

\bibitem{heuer_asm_indef_hyp}
N.~Heuer, Additive {S}chwarz methods for indefinite hypersingular integral
  equations in {${\bf R}^3$}---the {$p$}-version, Appl. Anal. 72~(3-4) (1999)
  411--437.

\bibitem{hiptmair_mao_BIT_2012}
R.~Hiptmair, S.~Mao, Stable multilevel splittings of boundary edge element
  spaces, BIT 52~(3) (2012) 661--685.

\bibitem{hiptwuzheng2012}
R.~Hiptmair, H.~Wu, W.~Zheng, Uniform convergence of adaptive multigrid methods
  for elliptic problems and {M}axwell's equations, Numer. Math. Theory Methods
  Appl. 5~(3) (2012) 297--332.

\bibitem{book_hsiao_wendland}
G.~C. Hsiao, W.~L. Wendland, Boundary integral equations, vol. 164 of Applied
  Mathematical Sciences, Springer-Verlag, Berlin, 2008.

\bibitem{hu_guo_katz_condition_bounds_p}
N.~Hu, X.-Z. Guo, I.~N. Katz, Bounds for eigenvalues and condition numbers in
  the {$p$}-version of the finite element method, Math. Comp. 67~(224) (1998)
  1423--1450.

\bibitem{melenk_appendix}
M.~Karkulik, J.~M. Melenk, A.~Rieder, Optimal additive {S}chwarz methods for
  the {$p$-BEM}: the hypersingular integral operator (in preparation).

\bibitem{kpp}
M.~Karkulik, D.~Pavlicek, D.~Praetorius, On {2D} newest vertex bisection:
  Optimality of mesh-closure and {$H^1$}-stability of {$L_2$}-projection,
  Constr. Approx. 38 (2013) 213--234.

\bibitem{karniadakis_sherwin}
G.~E. Karniadakis, S.~J. Sherwin, Spectral/{$hp$} element methods for {CFD},
  Numerical Mathematics and Scientific Computation, Oxford University Press,
  New York, 1999.

\bibitem{koornwinder75}
T.~Koornwinder, Two-variable analogues of the classical orthogonal polynomials,
  in: Theory and application of special functions ({P}roc. {A}dvanced {S}em.,
  {M}ath. {R}es. {C}enter, {U}niv. {W}isconsin, {M}adison, {W}is., 1975),
  Academic Press, New York, 1975, pp. 435--495. Math. Res. Center, Univ.
  Wisconsin, Publ. No. 35.

\bibitem{lions_schwarz_alternating_1}
P.-L. Lions, On the {S}chwarz alternating method. {I}, in: First
  {I}nternational {S}ymposium on {D}omain {D}ecomposition {M}ethods for
  {P}artial {D}ifferential {E}quations ({P}aris, 1987), SIAM, Philadelphia, PA,
  1988, pp. 1--42.

\bibitem{maischak_multilevel_asm}
M.~Maischak, A multilevel additive {S}chwarz method for a hypersingular
  integral equation on an open curve with graded meshes, Appl. Numer. Math.
  59~(9) (2009) 2195--2202.

\bibitem{maitre_pourquier}
J.-F. Maitre, O.~Pourquier, Condition number and diagonal preconditioning:
  comparison of the {$p$}-version and the spectral element methods, Numer.
  Math. 74~(1) (1996) 69--84.

\bibitem{nepomnyaschik_asm}
A.~M. Matsokin, S.~V. Nepomnyaschikh, A {S}chwarz alternating method in a
  subspace, Soviet Math. 29(10) (1985) 78--84.

\bibitem{book_mclean}
W.~McLean, Strongly elliptic systems and boundary integral equations, Cambridge
  University Press, Cambridge, 2000.

\bibitem{melenk_gl}
J.~M. Melenk, On condition numbers in {$hp$}-{FEM} with {G}auss-{L}obatto-based
  shape functions, J. Comput. Appl. Math. 139~(1) (2002) 21--48.

\bibitem{oswald_99}
P.~Oswald, Interface preconditioners and multilevel extension operators, in:
  Eleventh {I}nternational {C}onference on {D}omain {D}ecomposition {M}ethods
  ({L}ondon, 1998), DDM.org, Augsburg, 1999, pp. 97--104.

\bibitem{pavarino_94}
L.~F. Pavarino, Additive {S}chwarz methods for the {$p$}-version finite element
  method, Numer. Math. 66~(4) (1994) 493--515.

\bibitem{book_sauter_schwab}
S.~A. Sauter, C.~Schwab, Boundary element methods, vol.~39 of Springer Series
  in Computational Mathematics, Springer-Verlag, Berlin, 2011, translated and
  expanded from the 2004 German original.

\bibitem{schoeberl_asm_fem}
J.~Sch{\"o}berl, J.~M. Melenk, C.~Pechstein, S.~Zaglmayr, Additive {S}chwarz
  preconditioning for {$p$}-version triangular and tetrahedral finite elements,
  IMA J. Numer. Anal. 28~(1) (2008) 1--24.

\bibitem{schwab_p_fem}
C.~Schwab, {$p$}- and {$hp$}-finite element methods, Numerical Mathematics and
  Scientific Computation, The Clarendon Press, Oxford University Press, New
  York, 1998, theory and applications in solid and fluid mechanics.

\bibitem{bempp_preprint}
W.~\'{S}migaj, T.~Betcke, S.~Arridge, J.~Phillips, M.~Schweiger, Solving
  boundary integral problems with {BEM++}, ACM Trans. Math. Softw. 41~(2)
  (2015) 6:1--6:40.

\bibitem{book_steinbach}
O.~Steinbach, Numerical approximation methods for elliptic boundary value
  problems, Springer, New York, 2008, finite and boundary elements, Translated
  from the 2003 German original.

\bibitem{stevenson}
R.~Stevenson, The completion of locally refined simplicial partitions created
  by bisection, Math. Comp. 77~(261) (2008) 227--241.

\bibitem{tartar07}
L.~Tartar, An introduction to {S}obolev spaces and interpolation spaces, vol.~3
  of Lecture Notes of the Unione Matematica Italiana, Springer, Berlin, 2007.

\bibitem{toselli_widlund}
A.~Toselli, O.~Widlund, Domain decomposition methods---algorithms and theory,
  vol.~34 of Springer Series in Computational Mathematics, Springer-Verlag,
  Berlin, 2005.

\bibitem{tran_stephan_asm_h_96}
T.~Tran, E.~P. Stephan, Additive {S}chwarz methods for the {$h$}-version
  boundary element method, Appl. Anal. 60~(1-2) (1996) 63--84.

\bibitem{tran_stephan_mund_hierarchical_prec}
T.~Tran, E.~P. Stephan, P.~Mund, Hierarchical basis preconditioners for first
  kind integral equations, Appl. Anal. 65~(3-4) (1997) 353--372.

\bibitem{triebel95}
H.~Triebel, Interpolation theory, function spaces, differential operators, 2nd
  ed., Johann Ambrosius Barth, Heidelberg, 1995.

\bibitem{petersdorff_rwp_elasticity}
T.~von Petersdorff, Randwertprobleme der {E}lastizit{\"a}tstheorie f{\"u}r
  {P}olyeder - {S}ingularit{\"a}ten und {A}pproximation mit
  {R}andelementmethoden, Ph.D. thesis, Technische Hochschule Darmstadt, in
  German (1989).

\bibitem{wuchen06}
H.~Wu, Z.~Chen, Uniform convergence of multigrid {V}-cycle on adaptively
  refined finite element meshes for second order elliptic problems, Sci. China
  Ser. A 49~(10) (2006) 1405--1429.

\bibitem{xcn09}
J.~Xu, L.~Chen, R.~H. Nochetto, Optimal multilevel methods for
  {$H(\mathrm{grad})$}, {$H(\mathrm{curl})$}, and {$H(\mathrm{div})$} systems
  on graded and unstructured grids, in: Multiscale, nonlinear and adaptive
  approximation, Springer, Berlin, 2009, pp. 599--659.

\bibitem{xch10}
X.~Xu, H.~Chen, R.~H.~W. Hoppe, Optimality of local multilevel methods on
  adaptively refined meshes for elliptic boundary value problems, J. Numer.
  Math. 18~(1) (2010) 59--90.

\bibitem{zaglmayr_diss}
S.~Zaglmayr, High order finite element methods for electromagnetic field
  computation, Ph.D. thesis, Johannes Kepler University Linz (2006).

\bibitem{zhang_multilevel_schwarz}
X.~Zhang, Multilevel {S}chwarz methods, Numer. Math. 63~(4) (1992) 521--539.

\end{thebibliography}

\end{document}